\documentclass[10pt,a4paper]{article}
\usepackage[english]{babel}
\usepackage{amsmath}
\usepackage{amsfonts}
\usepackage[latin1]{inputenc}
\usepackage[T1]{fontenc}
\usepackage{amssymb}
\usepackage[ruled,algosection,linesnumbered]{algorithm2e}
\usepackage{color}
\usepackage{cite}
\usepackage{amsthm}
\usepackage{latexsym}
\usepackage{graphicx}
\usepackage[lofdepth,lotdepth]{subfig}
\usepackage{bbm}
\usepackage{url}
\usepackage{longtable}
\newtheorem{theorem}{Theorem}[section]

\newtheorem{proposition}[theorem]{Proposition}
\newtheorem{lemma}[theorem]{Lemma}

\newtheorem{crllr}[theorem]{Corollary}

\newcommand{\diam}{\operatorname{diam}} 
 
\newcommand{\dist}{\operatorname{diag}}

\newcommand{\spanlin}{\operatorname{span}} 
\newcommand{\erf}{\operatorname{erf}} 
\newcommand{\la}{\langle} 
\newcommand{\ra}{\rangle} 
\usepackage[left=2cm,right=2cm,top=2cm,bottom=2cm]{geometry}

\title{Hierarchical model reduction of nonlinear partial differential equations based on the adaptive empirical projection method and reduced basis techniques}
\author{Kathrin Smetana\footnote{Institute for Applied Mathematics, University of M\"{u}nster, Einsteinstr. 62, D-48149 M\"{u}nster, Germany, e-mail: \texttt{kathrin.smetana@wwu.de}},
Mario Ohlberger\footnote{Institute for Applied Mathematics, University of M\"{u}nster, Einsteinstr. 62, D-48149 M\"{u}nster, Germany, e-mail: \texttt{ohlberger@uni-muenster.de}}} 
\begin{document}
\maketitle

\begin{abstract}
In this paper we extend the hierarchical model reduction framework based on reduced basis techniques recently introduced in \cite{OS10} for the application to nonlinear partial differential equations. 
The major new ingredient to accomplish this goal is the introduction of the adaptive empirical projection method, which is an adaptive integration algorithm based on the (generalized) empirical interpolation method \cite{BMNP04,MaMu13}. Different from other partitioning concepts for the empirical interpolation method we perform an adaptive decomposition of the spatial domain. We project both the variational formulation and the range of the nonlinear operator onto reduced spaces. Those reduced spaces combine the full dimensional (finite element) space in an identified dominant spatial direction and a reduction space or collateral basis space spanned by modal  orthonormal basis functions in the transverse direction. Both the reduction and the collateral basis space are constructed in a highly nonlinear fashion by introducing a parametrized problem in the transverse direction and associated parametrized operator evaluations, and by applying reduced basis methods to select the bases from the corresponding snapshots. 
Rigorous a priori and a posteriori error estimators, which do not require additional regularity of the nonlinear operator are proven for the adaptive empirical projection method and then used to derive a rigorous a posteriori error estimator for the resulting hierarchical model reduction approach.
Numerical experiments for an elliptic nonlinear diffusion equation demonstrate a fast convergence of the proposed dimensionally reduced approximation to the solution of the full-dimensional problem. Runtime experiments verify a close to linear scaling of the reduction method in the number of degrees of freedom used for the computations in the dominant direction. 
\end{abstract}
\vspace*{\baselineskip}


\section{Introduction}
Many phenomena in nature and in particular fluid dynamics exhibit a dominant spatial direction along which the essential dynamics occur. Examples are blood flow problems or the flow in river beds which can be both modeled by the incompressible Navier-Stokes equations (cf.~\cite{Tem01,ForQuaVen09}) or groundwater flow in unsaturated soils which may be described by the Richards equation (cf.~\cite{Bear88,BOSS13}). This feature can be exploited to derive a dimensionally reduced model for the dominant direction, which should however include information on the transverse dynamics to improve the accuracy of the approximation. This paper is devoted to the derivation of an efficient dimensional reduction approach for steady nonlinear partial differential equations (PDEs) of the general type
\begin{equation} \label{nonlinear_pde}
\text{Find} \enspace p \in H^{1}_{0}(\Omega): \quad	\langle A(p), v\rangle = \langle f, v\rangle  \quad \forall v \in H^1_0(\Omega).
\end{equation}
Here, $\Omega \subset \mathbb{R}^d$ is a bounded domain with Lipschitz boundary, $f \in H^{-1}(\Omega)$ is a given right hand side, $A: H^{1}_{0}(\Omega) \rightarrow H^{-1}(\Omega)$ denotes a nonlinear elliptic operator and $\langle \cdot , \cdot \rangle$ is the dual pairing of $H^{-1}(\Omega)$ and $H^{1}_{0}(\Omega)$. Note that the steady Richards equation is a PDE of  type  \eqref{nonlinear_pde}. Moreover, treating the steady incompressible Navier-Stokes equations just requires to replace $H^{1}_{0}(\Omega)$ in \eqref{nonlinear_pde} by a divergence-free space, which we do not address in this paper to simplify the presentation. For the same reason we also restrict to $d=2$ and assume that the domain $\Omega$ is given as a tensor product, i.e. $\Omega:= \Omega_{1D}\times \omega$ with $\Omega_{1D},\omega \subset \mathbb{R}$. 

We use the term `dimensional reduction' in the sense of Vogelius and Babu\v{s}ka \cite{VogBab81a}, which means that a dimensional reduction method reduces the space dimension of the considered PDE by at least one. Needless to say that a dimensional reduction method may therefore be also seen as a model order reduction procedure. \\

There are a large variety of dimensional reduction methods and low rank tensor based approximations.  The asymptotic expansion technique \cite{Stoker1992,Bear88} is based on an expansion of the solution dependent on the presumed small ratio between the length of the domain in transverse and dominant direction. This method neglects the transverse dynamics and is only valid if the considered domain is very thin, or equivalently, the solution is constant along the vertical direction, which is often not the case.

To overcome this difficulty  in the work by Vogelius and Babu\v{s}ka \cite{VogBab81a,VogBab81b,VogBab81c} the hierarchical model reduction (HMR) approach has been introduced in the context of heat conduction in plates and shells. The idea of HMR is to perform a Galerkin projection of the full variational problem onto a reduced, $m$-dimensional space, which combines the full solution space in the dominant direction with a $m$-dimensional reduction space in the transverse direction, spanned by modal orthonormal basis functions. This yields a (possibly nonlinear) system of $m$ equations in one space dimension. The application and applicability of the HMR approach for linear advection-diffusion problems that exhibit a dominant flow direction has been studied and demonstrated by Perotto, Ern and Veneziani in \cite{ErnPerVen2008, PerErnVen10} in a more general geometric setting. Exploiting that HMR yields a hierarchy of reduced models determined by the reduction space, the dimension of the models is chosen adaptively 
in different subdomains of $\Omega$, employing an iterative substructuring method to couple the local models \cite{PerErnVen10,PerVen14}. In all these contributions the $m$-dimensional reduction space is spanned by a priori chosen boundary-adapted Legendre or trigonometric polynomials. 

The key idea of the hierarchical model reduction method based on reduced basis techniques (RB-HMR), introduced in \cite{OS10,OhlSme2011}, is to use a highly nonlinear approximation in the sense of 
\cite{Dev98} for the construction of the reduction space. This is realized by first deriving a parametrized one-dimensional problem in the transverse direction from the full problem, where the parameters reflect the influence from the unknown solution in the dominant direction. In a second step, reduced basis (RB) methods are used to generate a snapshot set from the solution manifold of the parametrized transverse problems and 
to construct the reduction space from these snapshots by a proper orthogonal decomposition (POD). In this way, both in the construction of the solution manifold and the subsequent choice of
the basis functions, information on the full solution is included, and the RB-HMR approach benefits from the good approximation properties of RB methods \cite{Pinkus85,DePeWo12,KahVol07}. This yields often an exponentially fast convergence of the RB-HMR method even for non-smooth functions and a more rapid convergence of the RB-HMR method than the classical HMR approach based on polynomials \cite{OS10}. 
It has also been demonstrated in \cite{OS10} for linear problems that thanks to its rapid convergence and the fact that the parametrized problems are of lower dimension than the full problem, the RB-HMR approach yields in many cases a very accurate approximation at a smaller runtime, including the costs for the construction of the reduction space, than a corresponding full dimensional finite element method (FEM) solve. For these reasons, we consider the RB-HMR approach in this paper.

While HMR constitutes an interpolation between the full model and the lower dimensional model via the dimension of the reduction space, in the geometrical multiscale approach models in one space dimension or lumped models, as say electronic network models, are locally enhanced with the full dimensional model by a domain decomposition scheme (cf.~\cite{ForGerNobQua2001,ForNobQuaVen1999}).
Finally, similar to HMR also the proper generalized decomposition method (cf.~\cite{AmMoChKe06,ChiAmmCue2010,CanEhrLel11} and references therein) is a tensor based approximation, but the tensor products in the expansion are computed iteratively by solving the Euler-Lagrange equations corresponding to the considered problem. \\

The key challenge in applying dimensional reduction to nonlinear PDEs is the efficient evaluation of the nonlinear operator, which requires in principle computations that scale with the degrees of freedom of the full system and not the reduced one as in the linear case. This is a general issue for projection-based model order reduction methods for nonlinear PDEs or nonlinear systems and several ways to tackle this problem have been proposed. Common to all these approaches is a first step in which an additional basis --- a so-called collateral basis --- is constructed say via a POD or a greedy algorithm to approximate the range of the nonlinear operator.  The methods then differ in the way the coefficients are computed. 

The Gauss-Newton with approximated tensors method \cite{CaFaCoAm2013,CarBouFar2011} is based on the gappy POD \cite{AWWB08,BuDaWi04} and therefore employs a projection via a gappy inner product defined as a linear combination of evaluations in certain points in the spatial domain. The discrete empirical interpolation method (DEIM) \cite{ChaSor10} and the slightly more general empirical operator interpolation method (EOIM) \cite{DroHaaOhl2012} employ the empirical interpolation method (EIM)  \cite{BMNP04,MNPP09,Tonn2011,GMNP07}. The latter allows for the interpolation of a parametrized function via Lagrangian interpolants, where both the collateral basis and the interpolation points are constructed by a greedy algorithm. It has been applied for the approximation of parametrized nonlinear PDEs within the framework of RB methods for instance in \cite{GMNP07}. A rigorous a posteriori error estimator for parametrically smooth functions has been introduced in \cite{EfGrPa10}. Moreover, a (non-rigorous) hierarchical a posteriori error estimator, which compares the solution with an approximation obtained by employing a richer collateral basis space, has been derived in \cite{BMNP04} for the EIM, in \cite{DroHaaOhl2012} for the EOIM, and in \cite{WirSorHaa12} for the DEIM. A rigorous a posteriori error estimator for the DEIM is presented in \cite{ChaSor10,ChaSor12} but the constant in the estimate depends on the underlying discretization. To facilitate an approximation of functions of low regularity both in the EOIM and the recently introduced generalized empirical interpolation method (GEIM) \cite{MaMu13} the EIM is generalized by considering (also) the evaluation of linear functionals. A priori error analysis for the GEIM as introduced in \cite{MaMu13} has been provided in \cite{MaMuTu13}. In this paper we apply the POD to construct the collateral basis as the POD is optimal in an $L^{2}$-sense and use the GEIM to select interpolating functionals. \\

As the dependency of the range of the nonlinear operator on the parameter and the spatial variables is in general non-smooth, we expect that we need many collateral basis functions and interpolating functionals to obtain an accurate approximation. To speed up the (online) computations often localized approximations are considered for instance by constructing (offline) a partition of the parameter space \cite{Wie15,EftSta12,DroHaaOhl2012} or the time domain \cite{DrHaOh11} and computing local collateral bases associated with each element of the partition. At the online stage, the correct basis is chosen following a certain criterion. Recently, it has been proposed to employ machine learning techniques to form clusters of similar snapshots and compute a collateral basis for each cluster in the offline stage for the  Gauss-Newton with approximated tensors method \cite{AmZaFa12} and the DEIM \cite{PeBuWiBu13}. The appropriate local collateral space is then chosen at the online stage either by a distance measure \cite{AmZaFa12} or classification strategies based on machine learning \cite{PeBuWiBu13}. 

However, in all partitioning methods based on the EIM \cite{Wie15,EftSta12,DroHaaOhl2012,DrHaOh11,PeBuWiBu13}  the number of interpolating points equals the number of (local) collateral basis functions, which may lead to an insufficient resolution of the (non-smooth) collateral basis functions and thus a considerably less accurate approximation. Therefore, we propose to perform an adaptive partitioning of the spatial domain driven by a suitable error indicator until a certain tolerance is reached and define the global interpolant as a sum of the local interpolants. We employ this adaptive (generalized) empirical interpolant to approximate the nonlinear term in the inner products of the coefficients of the orthogonal projection on the collateral basis. This yields an automatic numerical integration program based on the (G)EIM which we call the adaptive empirical projection method (EPM). We emphasize that in case of a nonlinear term, which is smooth with respect to the spatial variable, this higher regularity is maintained as we project onto the global collateral basis and employ the localized interpolants only within the inner products of the coefficients. We prove rigorous a priori and a posteriori error estimators for the adaptive EPM, which do not require additional regularity of the nonlinear operator and are independent of the underlying finite element discretization. Note that we do not propose to employ the adaptive EPM instead of the above mentioned partitioning or clustering methods but rather suggest to combine them.\\


To extend the RB-HMR approach to nonlinear PDEs of type \eqref{nonlinear_pde}, we therefore propose to proceed in the following way. We employ a highly nonlinear approximation for the construction of the collateral basis. To generate a manifold of parametrized  one-dimensional operator evaluations we use the solutions of the parametrized dimensionally reduced problem, derived as in the linear case, and the associated parametrization. During an adaptive training set extension procedure the sets of solution and operator snapshots are simultaneously generated. The collateral basis space is constructed by applying a POD to the operator snapshots and for the computation of the coefficients we employ the adaptive EPM. 

The rigorous a priori and a posteriori error estimators for the adaptive EPM are employed for the derivation of a rigorous a posteriori error estimator based on the Brezzi-Rappaz-Raviart theory \cite{BRR81,CalRap1997} which estimates both the error contribution caused by model reduction and by the approximation of the nonlinear operator. Hence, another contribution of this paper is the extension of the results in \cite{VerPat05} and particularly \cite{CaToUr09} from quadratically nonlinear to general nonlinear PDEs of type \eqref{nonlinear_pde}. This a posteriori error estimator is used within the context of the adaptive snapshot generation procedure. 
Numerical experiments for the elliptic nonlinear diffusion equation show that in many cases the proposed error estimator provides a sharp upper bound for the error. Moreover, the numerical experiments demonstrate 
a fast convergence of the RB-HMR approach and a close to linear scaling in the number of degrees of freedom of the discretization used in the dominant direction. \\


The article is organized as follows. In Section \ref{epm} we introduce the adaptive EPM for the approximation of parametrized functions in $L^{2}(\omega)$. The approximation properties of the adaptive EPM are discussed and rigorous a priori and a posteriori error estimates are derived. In the subsequent Section 
\ref{sect-hmr-nonlin} the problem adapted RB-HMR framework \cite{OS10} 
is generalized to nonlinear problems, using the approximation properties of the adaptive EPM. The resulting 
model reduction algorithm is discussed in detail and analyzed rigorously based on the Brezzi-Rappaz-Raviart theory \cite{BRR81,CalRap1997}. Next, we analyze the convergence behavior and the computational efficiency of the RB-HMR approach numerically for an elliptic nonlinear diffusion problem in Section \ref{numerics_nonlin}. Furthermore, we investigate the reliability and effectivity of the proposed error estimators and test the applicability of the a priori and a posteriori bounds for the adaptive EPM. In Section \ref{conclusion} we provide some conclusions and final remarks.  

\section{The adaptive Empirical Projection Method}\label{epm}

In this section we introduce the adaptive EPM which aims at approximating all elements of a manifold $\mathcal{M} := \{ u(\mu,\cdot), \, \mu \in \mathcal{D}\} \subset L^{2}(\omega)$, where $u(\mu,\cdot) \in L^{2}(\omega)$ equals for instance the evaluation of a nonlinear differential operator in the solution of a PDE parametrized by $\mu$. Here $\omega=(y_{0},y_{1}) \subset \mathbb{R}$ and $\mathcal{D} \subset \mathbb{R}^{p}$ denotes the $p$-dimensional parameter domain. Needless to say that we may identify the manifold $\mathcal{M}$ with a (target) function $u: \mathcal{D} \times \omega$ for which we assume $u \in L^{2}(\mathcal{D}\times \omega)$. Moreover, we require that we have a snapshot set $\mathcal{M}_{\Xi}  := \{ u(\mu,\cdot), \, \mu \in \Xi\}$  of the function $u$ at our disposal, where $\Xi \subset \mathcal{D}$ is a finite dimensional training set of size $|\Xi| = n$. The space $W_{k}=\spanlin \{\kappa_{1},...,\kappa_{k}\}$ with $(\kappa_{i},\kappa_{j})_{L^{2}(\omega)} = \delta_{ij}$ is then defined through a POD, i.e.
\begin{equation}\label{POD_nonlin_epm}
W_{k}:= \arg \underset{\dim(\tilde{W}_{k})=k }{\underset{\tilde{W}_{k} \subset \spanlin \{  \mathcal{M}_{\Xi}  \}}{\inf}} \left( \frac{1}{n} \sum_{\mu \in \Xi} \underset{\tilde{w}_{k} \in\tilde{W}_{k}}{\inf} \|u(\mu,\cdot) - \tilde{w}_{k}\|_{L^{2}({\omega})}^{2}\right).
\end{equation}
\begin{algorithm}[t]
\SetAlgoNoLine
\caption{GEIM - Construction of interpolating functionals\label{eim}}
\textsc{GEIM}($\mathcal{K}^{I}$, $\Sigma^{I}$, $I$)\\
Set
$
\sigma_{1}^{I} := \underset{\sigma^{I} \in \Sigma^{I}}{\arg \sup} |\sigma^{I}(\kappa_{1}^{I})|,  \quad q_{1}^{I} = \frac{\kappa_{1}^{I}}{\sigma_{1}^{I}(\kappa_{1}^{I})}, \quad \text{and} \quad B_{11}^{I,1}  = 1.
$\\
\For{$m=1,\dots,k_{I}$}{
Solve for the coefficients $\alpha_{j}^{m-1}$:
$
\sum_{j=1}^{m-1} \alpha_{j}^{m-1} \sigma_{i}^{I}(q_{j}^{I})  =  \sigma_{i}^{I}(\kappa_{m}^{I}), \enspace i = 1,...,m-1. 
$\\
Compute the residual
$
r_{m}({y}) := \kappa_{m}^{I}({y}) - \sum_{j=1}^{m-1}  \alpha_{j}^{m-1} q_{j}^{I}({y}).
$\\
Set
$
\sigma_{m}^{I} =  \underset{\sigma^{I} \in \Sigma^{I}}{\arg \sup} | \sigma^{I}(r_{m}) |, \quad q_{m}^{I} = \frac{r_{m}}{\sigma_{m}^{I}(q_{m}^{I})}, \quad \text{and} \quad B_{ij}^{I,m}  = \sigma_{i}^{I}(q_{j}^{I}), \quad 1 \leq i,j \leq m.
$\label{selection_functional}}
\Return  $\mathcal{S}^{I}$, $\mathcal{Q}^{I}:=\{q_{1}^{I},...,q_{k_{I}}^{I}\}$, $B^{I,k_{I}}$
\end{algorithm}
We approximate the function $u(\mu,\cdot)$ in the integrals of the orthogonal projection 
$$
P_{k}[u](\mu,y) :=\sum_{l = 1}^{k} \int_{{\omega}} u(\mu,z)\, \kappa_{l}({z})\,  d{z}\,\, \kappa_{l}(y) 
$$
by a (generalized) empirical interpolant $\mathcal{I}_{L}[u]$. The key idea of the adaptive EPM is that we adaptively decompose the domain $\omega$ into subdomains, construct local interpolants on each subdomain and then define the (global) interpolant $\mathcal{I}_{L}[u]$ as the sum of all local interpolants. To construct the latter we employ the set of functions $\kappa_{1},...,\kappa_{k}$, restrict them to the respective subdomain, apply a local POD to obtain a localized linear independent set of functions, and apply the GEIM \cite{MaMu13} locally to select the evaluating linear functionals from a (given) dictionary. Before we describe the adaptive EPM in detail we recall the GEIM and adapt some theoretical findings for the GEIM to our setting. \\
We suppose that we have given a dictionary $\Sigma$ of linear functionals $\sigma \in L^{2}(\omega)'$ of the form $\sigma(v) = (v,\varsigma)_{L^{2}(\omega)}$ for $v \in L^{2}(\omega)$ whose (unique) Riesz representation $\varsigma \in L^{2}(\omega)$  satisfies  $\|\varsigma \|_{L^{2}(\omega)}=1$. For $I\subset \omega$ and $\varsigma \in L^{2}(\omega)$, $v \in L^{2}(I)$ we then define localized functionals $\sigma^{I}$ as $\sigma^{I}(v) := (v,\varsigma)_{L^{2}(I)}$ and denote the corresponding localized dictionary by $\Sigma^{I}$. Note that the functions $\varsigma \in L^{2}(\omega)$ are the same for the functional $\sigma$ and its localized version $\sigma^{I}$. Additionally, we assume that the dictionaries $\Sigma^{I}$ are unisolvent in the sense that if we have for any $g \in \spanlin \{\mathcal{M}\}$ that $\sigma^{I}(g|_{I}) = 0$ for all $\sigma^{I} \in \Sigma^{I}$ this implies $g = 0$ almost everywhere in $I$. The selection of the interpolating functionals $\mathcal{S}^{I}:=\{\sigma_{1}^{I},...,\sigma_{k_{I}}^{I}\}$ for a given set of linear independent functions $\mathcal{K}^{I}:=\{\kappa_{1}^{I},...,\kappa_{k_{I}}^{I}\} \subset L^{2}(I)$ is described in Algorithm \ref{eim}. For a function $v \in L^{2}(\mathcal{D}\times I)$ we then define the local interpolant $\mathcal{I}^{I}_{k_{I}}[v](\mu,y):= \sum_{j=1}^{k_{I}} \alpha_{j}^{k_{I}}(\mu) q_{j}^{I}(y)$, where the coefficients are the solutions of: $\sum_{j=1}^{k_{I}} \alpha_{j}^{k_{I}}(\mu) B^{I,k_{I}}_{i,j} = \sigma_{i}^{I}(v(\mu;\cdot))$, $i=1,...,k_{I}$ and $B^{I,k_{I}}_{i,j} = \sigma_{i}^{I}(q_{j}^{I})$. The following lemma adapts some results for the GEIM to our setting.

\begin{lemma}\label{prop-eim}
Let the set of interpolating functionals $\mathcal{S}^{I}$ be selected by Algorithm \ref{eim} and let the assumptions from the previous paragraph be fulfilled. Then we have
\begin{enumerate}
\item The matrix $B^{I,k_{I}}$ is lower triangular with unity diagonal and hence invertible. Moreover, there holds $|B^{I,k_{I}}_{ij}|\leq 1$, $1 \leq i,j \leq k_{I}$. The set of functions $\mathcal{Q}^{I}$ forms a basis for the space $W^{I}_{k_{I}}:=\spanlin \{\mathcal{K}^{I}\}$ and the selection of the interpolating functionals is well-defined. 
\item The interpolation is exact for all $w \in W^{I}_{k_{I}}$.
\item There exist unique functions $\vartheta_{j}^{I} \in W^{I}_{k_{I}}$, that satisfy $\sigma^{I}_{i}(\vartheta_{j}^{I}) = \delta_{ij}$, $1 \leq i,j \leq k_{I}$.
\end{enumerate}
\end{lemma}
\begin{proof}
To prove (i) we adapt the argumentation in \cite{BMNP04,GMNP07} to our setting. We proceed by induction. By definition we have that $W_{1}^{I}= \spanlin \{ q_{1}^{I}\}$. Let us assume that $W_{k_{I}-1}^{I} = \spanlin \{ q_{1}^{I},...,q_{k_{I}-1}^{I}\}$. The construction of $q_{k_{I}}^{I}$ is well-defined if $B^{I,k_{I}-1}$ is invertible and $|\sigma^{I}_{k_{I}}(r_{k_{I}})|>0$. The properties of the matrix $B^{I,k_{I}-1}$ can be proven as in \cite{BMNP04,GMNP07}, exploiting the definition of the linear functionals $\sigma_{1}^{I},...,\sigma_{k_{I}-1}^{I}$ in Line \ref{selection_functional} of Algorithm \ref{eim}. To show $|\sigma^{I}_{k_{I}}(r_{k_{I}})|>0$ we argue by contradiction. Assume $|\sigma^{I}_{k_{I}}(r_{k_{I}})|=0$. Thanks to the unisolvence property of the dictionary $\Sigma^{I}$ we infer that $\kappa^{I}_{k_{I}} = \sum_{j=1}^{k_{I}-1} \alpha_{j}^{k_{I}-1} q_{j}^{I}$ almost everywhere in $I$. Exploiting the induction hypothesis we can express the functions $q_{1}^{I},...,q_{k_{I}-1}^{I}$ and thus $\kappa^{I}_{k_{I}}$ in the basis $\kappa_{1}^{I},...,\kappa_{k_{I}-1}^{I}$ which is contradictory to the requirement that the set of functions $\{\kappa_{1}^{I},...,\kappa_{k_{I}}^{I}\}$ is linear independent. Assertion (ii) can be proven as in \cite{MNPP09} and assertion (iii) follows from the invertibility of $B^{I,k_{I}}$.
\end{proof}

\begin{algorithm}[t]
\SetAlgoNoLine
\caption{adaptive Empirical Projection Method \label{adapt-EPM}}
\textsc{adaptive EPM}($\mathcal{K}$, $\Sigma$, $\mathcal{M}_{\Xi}$, $\varepsilon_{\text{{\tiny tol}}}^{\text{{\tiny int}}}$, $N_{\text{{\tiny max}}}^{\text{{\tiny int}}}$, $\Xi$)\\
{\bf Initialize} $I = {\omega}$,   $a^{I}=y_{0}$,  $b^{I}=y_{1}$; Define $\mathfrak{I}$ as the partition consisting of one element $I$. \\
Compute $[\mathcal{S}^{I},\mathcal{Q}^{I}, B^{I,k_{I}}]$= GEIM($\mathcal{K}$,$\Sigma^{I}$,$I$).\label{ori-EIM}\\
\For{$j = 1,...,N_{\text{{\tiny max}}}^{\text{{\tiny int}}}$}{
\ForEach{ $I \in \mathfrak{I}$}{Compute 
\begin{minipage}{0.8\linewidth}
\begin{equation}\label{local_error}
e^{I} := \frac{1}{n} \underset{\mu \in \Xi}{\sum} \| \sum_{l=1}^{k} \int_{{\omega}} (u(\mu,z) - \mathcal{I}_{L}[u](\mu,z)) \kappa_{l}(z) \,d{z}\, \kappa_{l} \|_{L^{2}(I)}^{2}
\end{equation}
\end{minipage}\\
\If{$e^{I} > \frac{|I|}{|{\omega}|} \cdot \varepsilon_{\text{{\tiny tol}}}^{\text{{\tiny int}}}$ }{
Set $I_{left} := [a^{I},(a^{I}+b^{I})/2],$
$I_{right} := [(a^{I}+b^{I})/2,b^{I}].$\\
Define the localized dictionaries $\Sigma^{I_{left}}$ and $\Sigma^{I_{right}}$. \\
Compute $\mathcal{K}^{I_{left}}:=\text{POD}(\{\kappa_{1}|_{I_{left}},...,\kappa_{k}|_{I_{left}}\})$ and $\mathcal{K}^{I_{right}}:=\text{POD}(\{\kappa_{1}|_{I_{right}},...,\kappa_{k}|_{I_{right}}\})$\label{local_POD}\\
Compute 
\quad $[\mathcal{S}^{I_{left}},\mathcal{Q}^{I_{left}}, B^{I_{left},k_{I_{left}}}] = \text{GEIM}(\mathcal{K}^{I_{left}},\Sigma^{I_{left}},I_{left})$,\label{local GEIM1}\\
\hspace*{1.75cm} $[\mathcal{S}^{I_{right}},\mathcal{Q}^{I_{right}}, B^{I_{right},k_{I_{right}}}] = \text{GEIM}(\mathcal{K}^{I_{right}},\Sigma^{I_{right}},I_{right})$,\label{local GEIM2}\\
Update $\mathfrak{I}$
}}
Set $e_{\text{{\tiny int}}} = \sum_{I \in \mathfrak{I}} e^{I}$. \\
\If{$e_{\text{{\tiny int}}} \leq \varepsilon_{\text{{\tiny tol}}}^{\text{{\tiny int}}}$}{{\bf go to line 21}}
}
\Return  $\mathcal{S}^{\mathfrak{I}}$, $\mathcal{Q}^{\mathfrak{I}}$, $B^{\mathfrak{I},k_{\mathfrak{I}}}$, $e_{\text{{\tiny int}}}$.
\end{algorithm}

To formulate the adaptive Empirical Projection Method \ref{adapt-EPM} and hence an adaptive integration algorithm based on GEIM, we introduce a non-uniform partition $\mathfrak{I}$ of ${\omega}$ with elements $I$. $a^{I}$ and $b^{I}$ denote the left and right interval boundary of $I$.  In Algorithm \ref{adapt-EPM} we first apply the standard GEIM on the whole domain ${\omega}$ in Line \ref{ori-EIM} to the set $\mathcal{K}:=\{\kappa_{1},...,\kappa_{k}\}$. If the integration error $e^{I}$ as defined in \eqref{local_error} for $I=\omega$ is smaller than the prescribed tolerance $\varepsilon_{\text{{\tiny tol}}}^{\text{{\tiny int}}}$ we stop without refining.  Otherwise we bisect in each iteration those intervals for which $e^{I} > (|I|/|{\omega}|) \cdot \varepsilon_{\text{{\tiny tol}}}^{\text{{\tiny int}}}$ holds. Note that the error $e^{I}$ is computable as it only requires the knowledge of $u$ for $\mu \in \Xi$, which can be accessed via $\mathcal{M}_{\Xi}$. On the new intervals we first define the localized dictionaries as described above and apply a POD in Line \ref{local_POD} to generate linear independent localized sets of functions $\mathcal{K}^{I_{left}}$ and $\mathcal{K}^{I_{right}}$ such that $\spanlin \{\mathcal{K}^{I_{m}}\} = \spanlin \{\kappa_{1}|_{I_{m}},...,\kappa_{k}|_{I_{m}}\}$, $m=left,right$. Note that we may alternatively define $\mathcal{K}^{I_{m}}$ as a linear independent subset of $\{\kappa_{1}|_{I_{m}},...,\kappa_{k}|_{I_{m}}\}$ with $\spanlin \{\mathcal{K}^{I_{m}}\} = \spanlin \{\kappa_{1}|_{I_{m}},...,\kappa_{k}|_{I_{m}}\}$, $m=left,right$. Subsequently we perform a localized GEIM to select sets of localized interpolating functionals $\mathcal{S}^{I_{left}}$ and $\mathcal{S}^{I_{right}}$ which are employed to define the local interpolants $\mathcal{I}_{k_{I_{left}}}^{I_{left}}[v]$ and $\mathcal{I}_{k_{I_{right}}}^{I_{right}}[w]$ for $v \in L^{2}(\mathcal{D} \times I_{left})$ and $w \in L^{2}(\mathcal{D} \times I_{right})$, respectively. We stop either if $e_{\text{{\tiny int}}} = \sum_{I \in \mathfrak{I}} e^{I} \leq \varepsilon_{\text{{\tiny tol}}}^{\text{{\tiny int}}}$ or if the maximal number of iterations $N_{\text{{\tiny max}}}^{\text{{\tiny int}}}$ is reached. The empirical projection of $u\in L^{2}(\mathcal{D}\times\omega)$ is then defined as
\begin{eqnarray}\label{epm-approx}
P_{k}^{L}[u](\mu,y) := \sum_{n = 1}^{k} \int_{{\omega}} \mathcal{I}_{L}[u](\mu,z)\, \kappa_{n}(z)\,  dz\,\, \kappa_{n}(y), \enspace
\mathcal{I}_{L}[u](\mu,y) := \sum_{I \in \mathfrak{I}} \mathcal{I}^{I}_{k_{I}}[u](\mu,y) = \sum_{I \in \mathfrak{I}} \sum_{j=1}^{k_{I}} \sigma^{I}_{j}(u(\mu,\cdot)) \vartheta_{j}^{I}({y}), 
\end{eqnarray}
where the functions $\vartheta_{j}^{I}$, $j=1,...,k_{I}$, have been defined in Lemma \ref{prop-eim} and $L:=\sum_{I\in \mathfrak{I}} k_{I}$. Finally, we remark that $N_{\text{{\tiny max}}}^{\text{{\tiny int}}}$ has been introduced for security purposes, as, so far, we could only prove the convergence of the adaptive EPM under certain assumption which are relatively mild, though. This issue as well as rigorous a priori and a posteriori bounds are addressed in the following subsection. 

\subsection{Rigorous a priori and a posteriori error analysis for the adaptive EPM}\label{bounds_epm}



To control the projection error $\| u - P_{k}[u] \|_{L^{2}(\mathcal{D}\times\omega)}$ by the POD error on the snapshot set, we interpret the discrete $L^{2}$-norm occurring in the definition of the POD-space \eqref{POD_nonlin_epm} as a numerical approximation of the corresponding integral with the Monte Carlo method, which is one new contribution of the proof, and subsequently use ideas of Kunisch and Volkwein \cite{KunVol02}. The main new contribution of Theorem \ref{apriori-epm} is the control of the term $\| P_{k}[u] - P_{k}^{L}[u]\|_{L^{2}(\mathcal{D}\times\omega)}$, which is possible due to the design of the adaptive EPM, using the Monte Carlo quadrature. To assess the integration error of the latter, we introduce the following notion \cite{EvaSwa00}: For sequences $\{X_{n}\}_{n}, \{Y_{n}\}_{n}$ of random variables we write $X_{n}= \mathcal{O}_{P}(Y_{n})$, if for any $\varepsilon > 0$ there exists $M_{\varepsilon},N_{\varepsilon} >0$ such that $P\left(\left| X_{n}/Y_{n}\right| > M_{\varepsilon}\right)<\varepsilon$ for all $n>N_{\varepsilon}$, where $P(E) $ denotes the probability of the event $E$. \\
We also introduce the operator $B_{n}: L^{2}(\omega) \rightarrow L^{2}(\omega)$, defined as
\begin{equation}
B_{n}(v) = \frac{1}{n} \sum_{\mu \in \Xi} \left(\int_{\omega} v(z)u(\mu,z)\,dz\, u(\mu,y) \right) \quad \forall v \in L^{2}(\omega). 
\end{equation}
Note that $B_{n}$ is a bounded, self-adjoint, and nonnegative operator and further compact thanks to its finite dimensional image. We denote by $\lambda_{l}^{n}$ the eigenvalues that satisfy the eigenvalue problem: Find $(\kappa_{l}^{n},\lambda_{l}^{n}) \in (L^{2}(\omega),\mathbb{R}^{+})$ such that
\begin{equation}\label{eigenvalue problem}
B_{n}\kappa_{l}^{n} = \lambda_{l}^{n}\kappa_{l}^{n},
\end{equation}
and assume that the eigenvalues $\lambda_{l}^{n}$ are listed in non-increasing order of magnitude, i.e. $\lambda_{1}^{n} \geq ... \geq \lambda_{d(n)} ^{n}> 0$ and $\lambda_{l}^{n}=0$ for $l>d(n)$. Note that we have added the superscript $n$ at the eigenvectors $\kappa_{l}$ to highlight their dependency on $n$ and $\Xi$.

\begin{theorem}[A priori error bound for the adaptive EPM]\label{apriori-epm} We assume that the parameter values $\mu \in \Xi$ are sampled from the uniform distribution over $\mathcal{D}$. Then for every $\varepsilon >0$ there exists an $N(\varepsilon)$ such that for all $n > N(\varepsilon)$ 
\begin{equation}\label{train_size_fixed}
\| u - P_{k}^{L}[u] \|_{L^{2}(\mathcal{D}\times\omega)} \leq \left(\underset{l = k + 1}{\overset{d(n)}{\sum}} \lambda_{l}^{n}\right)^{1/2}+ e_{\text{{\tiny int}}}^{1/2} + \varepsilon.
\end{equation} 
If furthermore $\lambda_{k}^{\infty} \neq \lambda_{k + 1}^{\infty}$ there exists an $N(\varepsilon)$ such that for all $n > N(\varepsilon)$ 
\begin{equation}\label{train_size_inf}
\| u - P_{k}^{L}[u] \|_{L^{2}(\mathcal{D}\times\omega)} \leq  \sqrt{2} \left(\underset{l = k + 1}{\overset{\infty}{\sum}} \lambda_{l}^{\infty}\right)^{1/2} + e_{\text{{\tiny int}}}^{1/2} + \varepsilon,
\end{equation}
and $\lambda_{l}^{n} \rightarrow \lambda_{l}^{\infty}$ for $1 \leq l \leq k$ as $n \rightarrow \infty$ and $\kappa_{l}^{n} \rightarrow \kappa_{l}^{\infty}$ strongly in $L^{2}(\omega)$ for $1 \leq l \leq k$ and $n \rightarrow\infty$, where $\{\lambda_{l}^{\infty}\}_{l=1}^{\infty}$ are the eigenvalues and $\kappa_{l}^{\infty}$ are the eigenfunctions of the operator $B: L^{2}({\omega}) \rightarrow L^{2}({\omega})$, defined as
\begin{equation}\label{compact operator}
B(v) =  \int_{\mathcal{D}} \int_{{\omega}} v({z}) u(\mu,z)\, d{z} \, u(\mu,y)\, d\mu\quad \text{for} \enspace v \in L^{2}({\omega}).
\end{equation}
Regarding the rate of convergence in $n$, we have that 
\begin{align}
\label{train_size_fixed_rate}\| u - P_{k}^{L}[u] \|_{L^{2}(\mathcal{D}\times\omega)} &\leq \left(\underset{l = k + 1}{\overset{d(n)}{\sum}} \lambda_{l}^{n}\right)^{1/2}+ e_{\text{{\tiny int}}}^{1/2} + \mathcal{O}_{P}(n^{-1/4})\\
\text{and} \quad 
\label{train_size_inf_rate}\| u - P_{k}^{L}[u] \|_{L^{2}(\mathcal{D}\times\omega)} &\leq  \sqrt{2} \left(\underset{l = k + 1}{\overset{\infty}{\sum}} \lambda_{l}^{\infty}\right)^{1/2} + e_{\text{{\tiny int}}}^{1/2} + \mathcal{O}_{P}(n^{-1/4}).
\end{align}
If Algorithm \ref{adapt-EPM} converges, i.e. $e_{\text{{\tiny int}}} \leq \varepsilon_{\text{{\tiny tol}}}^{\text{{\tiny int}}}$, the estimates \eqref{train_size_fixed} -- \eqref{train_size_inf_rate} hold with $e_{\text{{\tiny int}}}^{1/2}$ replaced by $(\varepsilon_{\text{{\tiny tol}}}^{\text{{\tiny int}}})^{1/2}$. 

\end{theorem}
\begin{proof}
We begin with splitting the error into a projection error and an integration error:
\begin{equation}\label{apriori_est1}
\| u - P_{k}^{L}[u] \|_{L^{2}(\mathcal{D}\times\omega)} \leq \| u - P_{k}[u] \|_{L^{2}(\mathcal{D}\times\omega)}  + \| P_{k}[u] - P_{k}^{L}[u]\|_{L^{2}(\mathcal{D}\times\omega)} . 
\end{equation}
Thanks to the assumptions on $\Xi$ we can interpret for an arbitrary function $f \in L^{2}(\mathcal{D}\times\omega)$, the  term $I_{n}(F):= (1/n) \sum_{\mu \in \Xi}  \| f(\mu,\cdot) \|_{L^{2}({\omega})}^{2}$ as a numerical approximation of the integral $I(f) = \int_{\mathcal{D}} \int_{{\omega}} f^{2} \, d{y} \, d\mu$ with the Monte Carlo method. Thus, the strong law of large numbers (see for instance \cite{Fel68}) yields that for every $\delta >0$ there exists an $N'(\delta)$ such that for all $n > N'(\delta)$ 
\begin{align*}
 \| u - P_{k}[u] \|_{L^{2}(\mathcal{D}\times\omega)}^{2} 
 = \left(\frac{1}{n} \sum_{\mu \in \Xi}  \| u(\mu,\cdot) - \sum_{l=1}^{k} \int_{{\omega}}u(\mu,{z}) \kappa_{l}({z}) \,d{z} \,\kappa_{l}\|_{L^{2}({\omega})}^{2}\right)  + \delta \leq \left(\underset{l = k + 1}{\overset{d(n)}{\sum}} \lambda_{l}^{n}\right) + \delta,
\end{align*}
where we have used the classical estimate for the POD error.
Approximating also the integral of the second term in \eqref{apriori_est1} with a Monte Carlo method and using the outcome of Algorithm \ref{adapt-EPM}, we obtain that for every $\delta >0$ there exists an $N^{\prime\prime}(\delta)$ such that for all $n > N^{\prime\prime}(\delta)$ 
\begin{align*}
& \| P_{k}[u] - P_{k}^{L}[u] \|_{L^{2}(\mathcal{D}\times\omega)}^{2} 
 = \left(\frac{1}{n} \sum_{\mu \in \Xi}  \| \sum_{l=1}^{k} \int_{{\omega}}u(\mu,{z}) \kappa_{l}({z})  \,d{z} \,\kappa_{l} - \sum_{l=1}^{k} \int_{{\omega}}\mathcal{I}_{L}[u](\mu,{z})  \kappa_{l}({z}) \,d{z} \,\kappa_{l}\|_{L^{2}({\omega})}^{2}\right)  + \delta\\
& \qquad = \left(\frac{1}{n} \sum_{\mu \in \Xi}  \| \sum_{l=1}^{k} \int_{{\omega}}(u(\mu,{z}) -  \mathcal{I}_{L}[u](\mu,{z}))  \kappa_{l}({z})  \,d{z} \,\kappa_{l}\|_{L^{2}({\omega})}^{2}\right) + \delta
 \leq e_{\text{{\tiny int}}} + \delta.
\end{align*}
Choosing $\delta = \varepsilon/2$ and $N(\varepsilon)=\max\{N'(\delta),N^{\prime\prime}(\delta)\}$ yields \eqref{train_size_fixed}. \\
To show \eqref{train_size_inf} we first note that 
the operator $T:L^{2}({\omega}) \rightarrow L^{2}(\mathcal{D})$, defined as
$
(Tv)(\mu) := \int_{{\omega}} u(\mu,{y}) v({y}) \, d{y}, \ \text{for} \enspace v \in L^{2}({\omega}),
$
is a Hilbert-Schmidt integral operator and thus compact. Boundedness of the operator $\mathcal{Y}:L^{2}(\mathcal{D}) \rightarrow L^{2}({\omega})$, defined as
$
\mathcal{Y}(w) := \int_{\mathcal{D}} u(\mu,{y}) w(\mu) \, d\mu, \ \text{for} \enspace w \in L^{2}(\mathcal{D}),
$
yields that $B$ is a compact operator as well. The estimate \eqref{train_size_inf}, $\lambda_{l}^{n} \rightarrow \lambda_{l}^{\infty}$ for $1 \leq l \leq k$ as $n \rightarrow \infty$ and $\kappa_{l}^{n} \rightarrow \kappa_{l}^{\infty}$ strongly in $L^{2}(\omega)$ for $1 \leq l \leq k$  can then be proven completely analogous to the argumentation in Section 3.2 of \cite{KunVol02}. Note that the convergence of  $\kappa_{l}^{n}$ to $\kappa_{l}^{\infty}$ strongly in $L^{2}(\omega)$ for $1 \leq l \leq k$ and $n \rightarrow\infty$ leads to the well-definedness of the interpolating functionals $\mathcal{S}^{I}$, $I\in\mathfrak{I}$ also for $n\rightarrow \infty$ and thus to the boundedness of the term $e_{int}$ independent of $n$.\\ 
Finally, the (probabilistic) convergence rate in $n$ is a direct consequence of the central limit theorem (see for instance \cite{Fel68,Caf98}).
\end{proof}

We remark that the assumptions on $\Xi$ can be weakened in the sense that also an adaptive sampling strategy can be considered. This may change the convergence rate of the Monte Carlo method, but does not affect the proof of Theorem \ref{apriori-epm}.   
Alternatively, a quasi-Monte Carlo method may be used, which has an improved convergence rate of approximately $\mathcal{O}_{P}((\log n)^{c}n^{-1})$ for some constant $c$ \cite{Caf98}. \\

Next, we prove under certain assumptions that the integration error $e_{int}$ converges to $0$ if $k \rightarrow \infty$ and thus that the adaptive integration Algorithm \ref{adapt-EPM} converges. The main ingredients of the proof are the classical POD error bound, the exploitation of the properties of the GEIM as recalled in Lemma \ref{prop-eim} on the elements $I \in \mathfrak{I}$, and the bounds of the interpolation error of the localized GEIM. \\

To this end we introduce for each $I \in \mathfrak{I}$ the Lebesgue constant with respect to the $L^{2}(I)$-norm \cite{MaMu13}
as 
\begin{equation}\label{lebesgue}
\Lambda_{k_{I}}^{I} := \sup_{g(\mu) \in \mathcal{M}} \frac{\| \mathcal{I}^{I}_{k_{I}}[g](\mu,\cdot)\|_{L^{2}(I)}}{\| g(\mu,\cdot) \|_{L^{2}(I)}}.
\end{equation}

Based on that we obtain the following bound for $e_{int}$.

\begin{proposition}[Convergence of the adaptive EPM]\label{apriori-epm-convergence} Let $\lambda_{l}^{n}$ be the eigenvalues of the eigenvalue problem \eqref{eigenvalue problem}, $\Lambda_{k_{I}}^{I}$, $I \in \mathfrak{I}$ the Lebesgue constants as defined in \eqref{lebesgue}, and $e_{int}=\sum_{I\in \mathfrak{I}} e^{I}$ with $e^{I}$ defined in \eqref{local_error}. Then there holds
\begin{eqnarray}\label{train_size_fixed_convergence} 
e_{int}^{1/2} \leq \sqrt{k} \left(1 + \left(\sum_{I \in \mathfrak{I}} (\Lambda_{k_{I}}^{I})^{2}\right)^{1/2}\right) \left(\underset{l = k + 1}{\overset{d(n)}{\sum}} \lambda^{n}_{l}\right)^{1/2}. 
\end{eqnarray} 
\end{proposition}
\begin{proof}
Let $I$ be an arbitrary interval in $\mathfrak{I}$. Exploiting $(\kappa_{i},\kappa_{j})_{L^{2}(\omega)} = \delta_{ij}$ twice, we obtain 
\begin{align}\label{est_app1}
e_{int} \leq \frac{1}{n} \sum_{\mu \in \Xi} \sum_{l=1}^{k} \left( \int_{{\omega}}  (u(\mu,{y}) - \mathcal{I}_{L}[u](\mu,{y}) ) \kappa_{l}({y}) \,d{y} \right )^{2} \leq  \frac{1}{n} \sum_{\mu \in \Xi} k \|u(\mu,\cdot) - \mathcal{I}_{L}[u](\mu,\cdot)  \|_{L^{2}({\omega})}^{2}. 
\end{align}
For each $\mu \in \Xi$ we can further estimate:
\begin{align*}
 \|u(\mu,\cdot) - \mathcal{I}_{L}[u](\mu,\cdot)  \|_{L^{2}({\omega})}\leq \underset{(i)}{\underbrace{\| u(\mu,\cdot) - \mathcal{I}_{L}[P_{k}[u]](\mu,\cdot) \|_{L^{2}({\omega})}}} + \underset{(ii)}{\underbrace{\| \mathcal{I}_{L}[P_{k}[u]](\mu,\cdot) - \mathcal{I}_{L}[u] (\mu,\cdot)\|_{L^{2}({\omega})}}}.
\end{align*}
As the GEIM is exact for all $w \in W_{k_{I}}^{I}$ (see Lemma \ref{prop-eim}), we obtain for $(i)$:
\begin{align}\label{est100}
\| u(\mu,\cdot) &- \mathcal{I}_{L}[P_{k}[u]](\mu,\cdot) \|_{L^{2}({\omega})}^{2} = \sum_{I \in \mathfrak{I}} \| u(\mu,\cdot) - \mathcal{I}_{L}[P_{k}[u]](\mu,\cdot) \|_{L^{2}(I)}^{2} 
= \sum_{I \in \mathfrak{I}} \| u(\mu,\cdot) - P_{k}[u](\mu,\cdot) \|_{L^{2}(I)}^{2}.
\end{align}
Using the definition of the Lebesgue constant we get for $(ii)$: 
\begin{align}
\nonumber\| \mathcal{I}_{L}[P_{k}[u]](\mu,\cdot) - \mathcal{I}_{L}[u] (\mu,\cdot)\|_{L^{2}({\omega})}^{2} & =  \sum_{I \in \mathfrak{I}} \| \mathcal{I}_{L}[P_{k}[u]](\mu,\cdot) - \mathcal{I}_{L}[u] (\mu,\cdot)\|_{L^{2}(I)}^{2} \\[-1.5ex]
\label{est200}\\[-1.5ex]
\nonumber & \leq \sum_{I \in \mathfrak{I}} (\Lambda_{k_{I}}^{I})^{2} \| P_{k}[u](\mu,\cdot) - u(\mu,\cdot) \|_{L^{2}(I)}^{2}.
\end{align}
By combining the estimates \eqref{est100} and \eqref{est200} we obtain
\begin{align}\label{lebesgue1}
\| u &(\mu,\cdot)- \mathcal{I}_{L}[u](\mu,\cdot) \|_{L^{2}({\omega})} \leq \left\{1 + \left(\sum_{I \in \mathfrak{I}} (\Lambda_{k_{I}}^{I})^{2}\right)^{1/2} \right\} \| u(\mu,\cdot) - P_{k}[u](\mu,\cdot) \|_{L^{2}({\omega})}.
\end{align}
The estimates \eqref{est_app1} and \eqref{lebesgue1} together with the classical estimate of the POD-error yield the desired result
\begin{align*}
\nonumber & e_{int} \leq  \frac{1}{n} \sum_{\mu \in \Xi} \left\{k \left(1 + \left(\sum_{I \in \mathfrak{I}} (\Lambda_{k_{I}}^{I})^{2}\right)^{1/2} \right)^{2} \| u(\mu,\cdot) - P_{k}[u](\mu,\cdot) \|_{L^{2}({\omega})}^{2}\right\} \\
& \qquad\qquad\leq k \left(1 + \left(\sum_{I \in \mathfrak{I}} (\Lambda_{k_{I}}^{I})^{2}\right)^{1/2} \right)^{2} \left(\underset{l = k + 1}{\overset{d(n)}{\sum}} \lambda^{n}_{l}\right).
\end{align*}
\end{proof}
To obtain convergence of the adaptive EPM we thus need that the Lebesgue constant increases rather moderately for growing $k$. Exploiting the properties of the entries of the matrices $B^{\mathfrak{I},k_{\mathfrak{I}}}$ it can be proven (see \cite{MaMu13}) that  the Lebesgue constants  $\Lambda_{k_{I}}^{I}$, $I \in \mathfrak{I}$ can be bounded as follows:
\begin{equation}\label{Lebsegue_bound}
\Lambda_{k_{I}}^{I} \leq 2^{k_{I}-1} \max_{i=1,...,k_{I}} \| q_{i} \|_{L^{2}(I)}.
\end{equation}
Therefore the POD-error $\sum_{l=k+1}^{d(n)} \lambda_{l}^{n}$ has to converge exponentially fast so that \eqref{train_size_fixed_convergence} yields convergence of the adaptive EPM. However, numerical results (see \cite{MMPY14}) show that the Lebesgue constant increases much slower than anticipated by \eqref{Lebsegue_bound} and in many cases even linear. Very recently it has been demonstrated in \cite{MMPY14} that for $v \in L^{2}(I)$ the localized generalized empirical interpolant $\mathcal{I}^{I}_{k_{I}}[v]$ can be interpreted as a Petrov-Galerkin approximation of $v$ where the approximation space is $W^{I}_{k_{I}}$ and the test space is spanned by the Riesz representations of the functionals $\sigma^{I} \in \mathcal{S}^{I}$ in $L^{2}(I)$. The Lebesgue constant $\Lambda^{I}_{k_{I}}$ then equals the reciprocal of the inf-sup constant associated with those approximation and trial spaces \cite{MMPY14}. This relates the Lebesgue constant to the considered dictionary $\Sigma$ and allows some guidance on how to choose $\Sigma$.\\
We remark that the proofs for the convergence rates of the EIM \cite{MNPP09} and for the GEIM \cite{MaMuTu13} crucially depend on the fact that the set of functions passed to Algorithm 2.1 are chosen by a greedy algorithm. Hence these results do not apply in our setting where we apply a POD. \\

Note also that Proposition \ref{apriori-epm-convergence} yields an upper bound for the (computable) integration error $e_{int}$. Therefore, we employ the a priori bounds \eqref{train_size_fixed} in Theorem \ref{apriori-epm} to derive a rigorous a posteriori estimator by comparing with a superior approximation $P_{k'}^{L'}[u]$. We emphasize that due to the usage of the Monte Carlo method the a posteriori error estimate will be a probabilistic estimate. To determine the number of samples $n$ needed to ensure an integration error due to the Monte Carlo approximation of at most $\varepsilon_{MC}$ with a confidence level $\mathcal{C}$ we introduce the empirical variances 
\begin{align*}
\nonumber &\qquad\quad \varsigma_{1} = \left[\frac{1}{n} \sum_{\mu \in \Xi}\left(\|u(\mu,\cdot) - P_{k}[u](\mu,\cdot)\|_{L^{2}(\omega)}^{2} - \left\{\frac{1}{n} \sum_{\mu \in \Xi}\|u(\mu,\cdot) - P_{k}[u](\mu,\cdot)\|_{L^{2}(\omega)}^{2}\right\}\right)^{2}\right]^{1/2},\\
\nonumber  &\qquad\quad \varsigma_{2} = \left[\frac{1}{n} \sum_{\mu \in \Xi}\left(\|P_{k}[u](\mu,\cdot) - P_{k}^{L}[u](\mu,\cdot)\|_{L^{2}(\omega)}^{2} - \left\{\frac{1}{n} \sum_{\mu \in \Xi}\|P_{k}[u](\mu,\cdot) - P_{k}^{L}[u](\mu,\cdot)\|_{L^{2}(\omega)}^{2}\right\}\right)^{2}\right]^{1/2}.
\end{align*}
Then, we obtain the following result.
\begin{proposition}[An a posteriori error estimate for the adaptive EPM]\label{apost-epm}
Let the assumptions of Theorem \ref{apriori-epm} be fulfilled and let $\varepsilon_{\text{{\tiny tol}}}$ be a given tolerance. Then the error estimate 
\begin{eqnarray}
\label{apost-equation-epm} \| u - P_{k}^{L}[u] \|_{L^{2}(\mathcal{D}\times \omega)} \leq  \varepsilon_{\text{{\tiny tol}}} + \Delta^{\text{{\tiny EPM}}} + e_{\text{{\tiny int}}}^{1/2} + \mathcal{O}_{P}(n^{-1/4}) \\
\label{delta_epm} \text{holds with} \qquad \Delta^{\text{{\tiny EPM}}} := \| P_{k'}^{L'}[u] - P_{k}^{L}[u] \|_{L^{2}(\mathcal{D}\times \omega)},
\end{eqnarray}
where $k'$ is defined as the minimal number in \{k+1,...,d(n)\}, such that 
\begin{equation}\label{pod_error_estimate}
 \left(\underset{j = l}{\overset{d(n)}{\sum}} \lambda_{j}^{n}\right)^{1/2} \leq \varepsilon_{\text{{\tiny tol}}},
\end{equation} 
and $L'$ is determined by  Algorithm \ref{adapt-EPM}, requiring $L' > L$. \\
Let  $\varepsilon_{MC}$ be a given tolerance for the error caused by the Monte Carlo approximation, $\mathcal{C}$ a given confidence level, and let $n$ satisfy $n\geq \max\{N_{1},N_{2}\}$. Let in turn $N_{1}$ and $N_{2}$ fulfill $N_{i} \geq \varepsilon_{MC}^{-2}\,\varsigma_{i}^{2} s(\mathcal{C})$, $i=1,2$, and $s(\mathcal{C})$ satisfy $\mathcal{C}=\erf (s(\mathcal{C})/\sqrt{2})$, where $\erf(\cdot)$ denotes the error function. Then the estimate
\begin{equation}
\label{apost-equation-epm-confidence} \| u - P_{k}^{L}[u] \|_{L^{2}(\mathcal{D}\times \omega)} \leq  \varepsilon_{\text{{\tiny tol}}} + \Delta^{\text{{\tiny EPM}}} + e_{\text{{\tiny int}}}^{1/2} + \varepsilon_{MC}
\end{equation}
holds true with the confidence level $\mathcal{C}$.
\end{proposition}
\begin{proof}
We apply  the a priori bound \eqref{train_size_fixed} to obtain
\begin{align*}
\| u - P_{k}^{L}[u] \|_{L^{2}(\Omega)} \leq \left(\underset{l = k + 1}{\overset{d(n)}{\sum}} \lambda_{l}^{n}\right)^{1/2}+ e_{\text{{\tiny int}}}^{1/2} + \mathcal{O}_{P}(n^{-1/4}).
\end{align*}
With the definition of $k'$, the estimates in Theorem \ref{apriori-epm}  and by computing $P^{L'}_{k'}[u]$ with Algorithm \ref{adapt-EPM} we get the result
\begin{equation*}
\| u - P_{k}^{L}[u] \|_{L^{2}(\Omega)} \leq  \varepsilon_{\text{{\tiny tol}}} + \| P_{k'}^{L'}[u] - P_{k}^{L}[u] \|_{L^{2}(\Omega)}  + e_{\text{{\tiny int}}}^{1/2} + \mathcal{O}_{P}(n^{-1/4}).
\end{equation*}
Estimate \eqref{apost-equation-epm-confidence} then follows directly from the central limit theorem and Slutsky's theorem (see for instance \cite{EvaSwa00}).
\end{proof}
Note that there might be cases where choosing $k'>k$ results in a situation, in which Algorithm \ref{adapt-EPM} bisects an interval for $k$ but not for $k'$. To ensure that $P^{L'}_{k'}[u]$ yields a better approximation than $P^{L}_{k}[u]$, we require $L'>L$. Note also that the eigenvalues in \eqref{pod_error_estimate} are computed when solving an eigenvalue problem to determine the POD basis.

\subsection{The adaptive EPM based on the EIM instead of the GEIM}\label{sect:EPM_EIM}

If $u \in L^{2}(\mathcal{D},L^{\infty}(\omega))$ is sufficiently regular to allow point evaluations one might  want to consider point evaluations instead of evaluating functionals as the former might be easier to implement within a programming code. To this end, we present in this subsection the changes that have to be made if we employ the EIM as introduced in \cite{BMNP04} instead of the GEIM. We suppose that the considered functions are regular enough to allow for point evaluations, which is for instance satisfied in a discrete setting where we employ a conforming finite element approximation. 

First, for a function $v(\mu,\cdot) \in L^{\infty}(\omega)$ we replace the evaluation by a functional $\sigma_{j}^{I} \in \Sigma^{I}$ as $\sigma_{j}^{I}(v(\mu,\cdot))$ by the point evaluation $v(\mu,t_{j})$, $t_{j} \in I$ for $I \in \mathfrak{I}$. Apart from that no changes are required in Algorithm \ref{eim} and this algorithm becomes the construction of the `magic points' \cite{MNPP09}. Then, we apply the EIM in Line \ref{local GEIM1} and \ref{local GEIM2} in Algorithm \ref{adapt-EPM} to the localized function sets $\mathcal{K}^{I_{left}}$ and $\mathcal{K}^{I_{right}}$, where the latter have been defined in Line \ref{local_POD} of Algorithm \ref{adapt-EPM}. Note that the statements for the GEIM in Lemma \ref{prop-eim} analogously hold true for the EIM. We emphasize that if we do not refine $\omega$ in Algorithm \ref{adapt-EPM}, the latter reduces to the application of the EIM to a POD basis as considered also for instance in \cite{UrbWie2012}. In this paper it has also been demonstrated that this yields the same approximation as the DEIM. \\
Theorem \ref{apriori-epm} remains valid for the adaptive EPM based on the EIM and can be proven analogously as in the previous subsection. We just note that thanks to the assumption $u \in L^{2}(\mathcal{D},L^{\infty}(\omega))$ we have that the eigenfunctions $\kappa_{l}^{n}$ are bounded with respect to the $L^{\infty}$-norm on $\omega$ for all $n \in \mathbb{N}$. Therefore, we may extract a weakly-$*$ converging subsequence in $L^{\infty}(\omega)$ and obtain that the limit eigenfunctions satisfy $\kappa_{l}^{\infty} \in L^{\infty}(\omega)$, $1 \leq l \leq k$. Hence, the selection of the interpolation points with the EIM is well-defined also in the limit $n\rightarrow \infty$, which in turn yields the uniform boundedness of $e_{int}$. One may then proceed as in Proposition \ref{apost-epm} to derive an a posteriori error estimator for the adaptive EPM based on the EIM. We emphasize that by running Algorithm \ref{adapt-EPM} with $N_{\text{{\tiny max}}}^{\text{{\tiny int}}} = 0$ and additionally computing $e^{I}$ in \eqref{local_error} for $I = \omega$, we obtain in this way rigorous a priori and a posteriori bounds for the DEIM \cite{ChaSor10}.\\ 
Regarding the proof of the convergence of the adaptive EPM we note that the Lebesgue constant
$
\tilde{\Lambda}_{k_{I}}^{I} :=  $ $\sup_{g(\mu) \in \mathcal{M}} (\| \tilde{\mathcal{I}}^{I}_{k_{I}}[g](\mu,\cdot)\|_{L^{\infty}(I)}/\| g(\mu,\cdot) \|_{L^{\infty}(I)})
$
can in general not be bounded by the $L^{2}$-based operator norm of the interpolation operator. Here, the $\enspace\tilde{}\enspace$ indicates that the respective quantities are defined for the adaptive EPM based on the EIM. However, if we restrict to a discrete setting an analogous result to Proposition  \ref{apriori-epm-convergence} may be obtained.  To this end we introduce a partition $\tau_{h}$ of ${\omega}$ with elements $\tau_{j} = ({y}_{j-1},{y}_{j})$ of width $h_{j} = {y}_{j} - {y}_{j-1}$ and maximal step size $h:= \max_{\tau_{j}}\, h_{j}$, and a conforming finite element space $Y^{h} \in L^{\infty}(\Omega)$ of dimension $n_{h}$. Then we may exploit the inverse estimate $\|\upsilon^{h}\|_{L^{\infty}({\omega})} \leq  h^{-1/2}  \|\upsilon^{h}\|_{L^{2}({\omega})}$,  $\upsilon^{h} \in Y^{h}$ to obtain
\begin{equation}\label{h-dependent est200}
\sup_{g(\mu) \in \mathcal{M}} \frac{\| \tilde{\mathcal{I}}_{L}[g](\mu,\cdot)\|_{L^{2}(\omega)}}{\| g(\mu,\cdot) \|_{L^{2}(\omega)}} \leq \left(\sum_{I \in \mathfrak{I}}|I| (\tilde{\Lambda}_{k}^{I})^{2}\right)^{1/2}h^{-1/2}. 
\end{equation}
Replacing the estimate in \eqref{est200} by the one in \eqref{h-dependent est200} yields the convergence of the adaptive EPM for a fixed mesh size $h$ for $k\rightarrow n_{h}$ under certain assumptions as stated in the following corollary.
\begin{crllr}[Convergence of the adaptive EPM in the discrete setting]\label{apriori-epm-disc} There holds
\begin{eqnarray}\label{train_size_fixed_disc} 
e_{int}^{1/2} \leq \sqrt{k} \left(1 + \left(\sum_{I \in \mathfrak{I}}|I| (\tilde{\Lambda}_{k}^{I})^{2}\right)^{1/2}h^{-1/2}\right) \left(\underset{l = k + 1}{\overset{d(n)}{\sum}} \tilde{\lambda}^{n}_{l}\right)^{1/2}. 
\end{eqnarray} 
\end{crllr}
For the Lebesgue constant $\tilde{\Lambda}_{k_{I}}^{I}$, $I \in \mathfrak{I}$ it can been shown as in \cite{BMNP04,GMNP07} that $\tilde{\Lambda}_{k_{I}}^{I} \leq 2^{k_{I}} - 1$. Although this bound can be actually reached \cite{MNPP09}, $\tilde{\Lambda}_{k}^{I} \leq 2^{k_{I}} - 1$ is a very pessimistic result and in numerical experiments a very moderate behavior is observed (cf.~\cite{MNPP09,GMNP07,DroHaaOhl2012}). Note that \eqref{train_size_fixed_disc} only yields convergence of the EPM if the POD-error converges faster than $\sqrt{k} (1 + (\sum_{I \in \mathfrak{I}} |I|(\tilde{\Lambda}_{k}^{I})^{2})^{1/2}h^{-1/2}))^{-1}$. We emphasize the dependence on $h^{-1/2}$ in \eqref{train_size_fixed_disc}. Therefore using the EIM  
within the adaptive EPM seems reasonable for moderate mesh sizes, whereas for $h\rightarrow 0$ we should rely on the GEIM. \\
Note that theoretically also the a posteriori bound for the EIM derived in \cite{EfGrPa10} can be employed to obtain an a posteriori estimate for the adaptive EPM. As the theory developed in \cite{EfGrPa10} however requires that the considered functions are parametrically smooth, it is not applicable within our context. 

\section[HMR for nonlinear PDEs]{Hierarchical Model Reduction for nonlinear PDEs}\label{sect-hmr-nonlin}

The goal of this section is the efficient construction of a low-dimensional reduction space and a collateral basis space, which yield a fast convergence of the RB-HMR approximation to the full solution. We recall that the reduction space is used to define the reduced space in which we search our reduced RB-HMR solution. In contrast the collateral basis space is constructed for the approximation of the range of nonlinear operator and therefore facilitates the evaluation of the nonlinear term at low cost. Following the approach in \cite{OS10}, we derive in \S \ref{1dproblem-nonlin} a parametrized nonlinear 1D PDE whose solution is employed for the definition of parametrized 1D operator evaluations in the transverse direction in \S \ref{gen-snap-nonlin}. The sets of solution and operator snapshots are generated simultaneously by an adaptive training set extension algorithm in \S \ref{adapt-RB-HMR-epm}. The principal components of the snapshot sets then form the reduction space and the collateral basis space. We begin with formulating the RB-HMR approach with the adaptive EPM in \S \ref{formulate_hmrrb_epm}.

\subsection{Formulation of the reduced problem in the RB-HMR framework employing the EPM} \label{formulate_hmrrb_epm}\label{hmrrb-epm}

We follow the hierarchical model reduction (HMR) framework introduced in \cite{PerErnVen10,ErnPerVen2008} and extended to the RB-HMR setting in \cite{OS10}.
We recall our assumption that the considered domain is a tensor product, i.e. $\Omega = \Omega_{1D} \times \omega$, where $\Omega_{1D} = (x_{0},x_{1})$ denotes the computational domain in the dominant direction, and ${\omega}= (y_{0},y_{1})$ the domain in the transverse direction. In more general situations a mapping to such a reference domain needs to be employed (cf.~\cite{OS10}). For $A: H^{1}_{0}(\Omega) \rightarrow H^{-1}(\Omega)$ and $f\in H^{-1}(\Omega)$ we consider the nonlinear problem
\begin{equation}\label{fullprob_nonlin2}
\text{Find} \enspace p \in H^{1}_{0}(\Omega): \quad \langle F(p), v \rangle  = 0 \quad \forall v \in H^{1}_{0}(\Omega), \qquad \text{where} \quad \langle F(p), v \rangle:= \langle A(p), v \rangle  - \langle f, v \rangle.
\end{equation}
Problem \eqref{fullprob_nonlin2} is denoted the full problem, and existence and  uniqueness of a solution $p \in H^{1}_{0}(\Omega)$ of \eqref{fullprob_nonlin2} is assumed. Following the HMR framework, we introduce a set of $L^{2}$-orthonormal basis functions $\{\phi_{k}\}_{k \in \mathbb{N}} \in H^{1}_{0}({\omega})$. At this point, we assume that the basis functions $\{\phi_{k}\}_{k \in \mathbb{N}}$ are given to us. Possible choices are trigonometric or boundary-adapted Legendre polynomials \cite{PerErnVen10} or a posteriori determined basis functions, whose construction will be detailed in this section.  We combine the reduction space $Y_{m}:= \spanlin \{\phi_{1},...,\phi_{m}\}$ with $H^{1}_{0}(\Omega_{1D})$ and define the reduced space
\begin{equation}\label{V_m_nonlin}
V_{m} = \left\{ v_{m}(x,y) = \underset{k = 1}{\overset{m}{\sum}}\, \overline{v}_{k}(x)\, \phi_{k}(y),\enspace \mbox{with} \enspace \overline{v}_{k}(x) \in H^{1}_{0}(\Omega_{1D}),\, x \in \Omega_{1D}, \, y \in \omega\, \right \},
\end{equation}
where $
\overline{v}_{k}(x) = \int_{{\omega}}  v_{m}(x,y)\,\phi_{k}({y})\, d{y},  k = 1,...,m.
$ 
The reduced solution $p_m \in V_{m}$ may then be obtained by Galerkin projection, i.e.
\begin{equation} \label{reprob_nonlin}
\text{Find} \enspace p_{m} \in V_{m}: \quad \langle F(p_{m}), v_{m} \rangle  = 0 \quad \forall v_{m} \in V_{m}.
\end{equation}
Based on this reduced problem, fully discrete reduced approximations can be derived by replacing $H^{1}_{0}(\Omega_{1D})$ in the definition of $V_{m}$ by some suitable one-dimensional finite element subspace.\\
We emphasize that in contrast to the case of linear PDEs \cite{PerErnVen10,ErnPerVen2008,OS10}, the integrals in the transverse direction in \eqref{reprob_nonlin} cannot be precomputed due to the nonlinear operator $A$. This implies that \eqref{reprob_nonlin} is still of full dimension. To overcome this difficulty and hence perform a dimensional reduction of \eqref{reprob_nonlin} we apply the adaptive EPM introduced in \S \ref{epm}.  \\

We suppose that a set of collateral basis functions $\{\kappa_{n}\}_{n=1}^{k}$ is given to us. The reduced problem based on the adaptive EPM then reads
\begin{align} \label{reprob_nonlin_epm}
\nonumber\text{Find} \enspace p_{m,k} \in V_{m}: \quad \langle P_{k}^{L}[A(p_{m,k})], v_{m} \rangle  &= \langle f, v_{m} \rangle \quad \forall v_{m} \in V_{m}, \qquad \text{where} \\
 P_{k}^{L}[A(p_{m,k})](x,y) &= \sum_{n = 1}^{k} \int_{{\omega}} \mathcal{I}_{L}[A(p_{m,k})](x,z)\, \kappa_{n}({z})\,  d{z}\,\, \kappa_{n}(y)\\
\nonumber &=\sum_{n = 1}^{k} \sum_{I \in \mathfrak{I}} \sum_{j=1}^{k_{I}}\int_{{\omega}} \sigma^{I}_{j}(A(p_{m,k}(x,\cdot)))\,\vartheta_{j}^{I}({z})\, \kappa_{n}({z})\,  d{z}\,\, \kappa_{n}(y).
\end{align}
Note that for some nonlinear operators it might be necessary to apply the adaptive EPM component-wise, exploiting that for any $z \in H^{1}_{0}(\Omega)$ there exist functions $u_{1},u_{2} \in L^{2}(\Omega)$ such that
\begin{equation}\label{eq:h-1rep}
\langle A(z), v\rangle = \int_{\Omega} u_{1}\, \partial_{x} \, v + u_{2}\, \partial_{y} \, v\, dx\, dy\quad \text{for all} \enspace v\in H^{1}_{0}(\Omega) \enspace \text{(cf.~\cite{Evans1998}, p.~283)}.
\end{equation}
Note also, that for the major part of problems which fall in the category of \eqref{nonlinear_pde}, we expect that for $A(z) \in H^{-1}(\Omega)$, $z \in H^{1}_{0}(\Omega)$ we actually have $A(z) \in L^{2}(\Omega)$ thanks to the lemma of J.~L. Lions, which states that for distributions $v$ on $\Omega$ which are in $H^{-1}(\Omega)$ and whose all partial derivatives are in $H^{-1}(\Omega)$ there holds $v \in L^{2}(\Omega)$ (see \cite{Cia97} and references therein). To simplify notations we do not introduce a separate notion for the cases where the adaptive EPM has to be applied component-wise but instead assume that such cases are covered by the formulation in \eqref{reprob_nonlin_epm}. \\

Rewriting $p_{m,k}$ as $p_{m,k}(x,y) = \sum_{s = 1}^{m}\, \overline{p}_{s,k}(x)\, \phi_{s}(y)$ we obtain: Find $\overline{p}_{s,k} \in H^{1}_{0}(\Omega_{1D}), s = 1,...,m$, such that 
\begin{eqnarray}\label{reduced_problem_nonlin_epm}
 \underset{n=1}{\overset{k}{\sum}} \sum_{I \in \mathfrak{I}} \sum_{j=1}^{k_{I}} \left\langle \sigma_{j}^{I}\left(A\left(\sum_{s = 1}^{m}\overline{p}_{s,k} \phi_{s}\right)\right) \int_{{\omega}} \vartheta_{j}^{I}({y}) \, \kappa_{n}({y}) \, d{y}\, \kappa_{n} , \xi \phi_{j}\right \rangle= \langle f, \xi \phi_{j} \rangle 
  \enspace  \forall\, \xi \in H^{1}_{0}(\Omega_{1D}) \enspace \mbox{and} \enspace j = 1,...,m.
\end{eqnarray}
To compute an approximation of $p_{m,k}$ we introduce a partition $\mathcal{T}_{H}$ of $\Omega_{1D}$ with elements $\mathcal{T}_{i} = (x_{i-1},x_{i})$ of width $H_{i} = {x}_{i} - {x}_{i-1}$ and maximal step size $H:= \max_{i}\, H_{i}$. Moreover, we introduce a conforming finite element space $X^{H} \subset H^{1}_{0}(\Omega_{1D})$ of dimension $N^{H} < \infty$ and basis $\xi^{H}_{i}$, $i=1,...,N^{H}$. Then the corresponding discrete reduced problem reads: Find $\overline{p}_{s,k}^H \in X^H$, $s = 1,...,m$, such that
\begin{eqnarray}\label{red_prob_hmr_epm}
 \underset{n=1}{\overset{k}{\sum}} \sum_{I \in \mathfrak{I}} \sum_{j=1}^{k_{I}} \left\langle \sigma_{j}^{I}\left(A\left(\sum_{s = 1}^{m}\overline{p}_{s,k}^H \phi_{s}\right)\right) \int_{{\omega}} \vartheta_{j}^{I}({y}) \, \kappa_{n}({y}) \, d{y}\, \kappa_{n} , \xi^H_{i} \phi_{j}\right \rangle= \langle f, \xi^H_{i} \phi_{j} \rangle,
\end{eqnarray}
for $i = 1,...,N_{H}$, $j=1,...,m$, which is equivalent to the short notation
\begin{eqnarray}\label{red_prob_hmr_epm2}
\text{Find} \enspace p_{m,k}^{H} \in V_{m}^{H}: \quad \langle P_{k}^{L}[F(p_{m,k}^{H})], \xi^H_{i} \phi_{j} \rangle  = 0 \quad \text{for} \enspace i = 1,...,N_{H} \enspace \mbox{and} \enspace j = 1,...,m, 
\end{eqnarray}
where $\langle P_{k}^{L}[F(p_{m,k}^{H})], \xi^H_{i} \phi_{j} \rangle  = \langle P_{k}^{L}[A(p_{m,k}^{H})], \xi^H_{i} \phi_{j} \rangle - \langle f, \xi^H_{i} \phi_{j} \rangle$, $i = 1,...,N_{H}$, $j = 1,...,m$.
We emphasize that thanks to the application of the adaptive EPM we can now precompute the integrals in
the transverse direction in \eqref{reduced_problem_nonlin_epm} and \eqref{red_prob_hmr_epm} and as a result the computation of $\bar{p}_{s,k}$ and $\bar{p}_{s,k}^{H}$ reduces to the solution of a coupled system of nonlinear one-dimensional PDEs of size $m$ \eqref{reduced_problem_nonlin_epm} or $m \cdot N_{H}$ \eqref{red_prob_hmr_epm}.\\

Problem \eqref{red_prob_hmr_epm} can be efficiently solved by Newton's method. It is possible to reuse the collateral basis for a nonlinear operator also for the approximation of its Fr\'{e}chet derivative \cite{DroHaaOhl2012}. To obtain a better approximation of $A^{\prime}(p^{H}_{m,k}(x,{y}))$ and thus ideally a faster convergence of the Newton scheme solving for $p_{m,k}^{H}$, we propose to use a second collateral basis space $W_{f,k_{f}}:=\spanlin \{\kappa_{1}^{f},...,\kappa^{f}_{k_{f}}\}$ for this approximation. 
Assuming that $W_{f,k_{f}}$ is given, the Newton scheme is defined as follows: 
\begin{eqnarray}
\nonumber \langle P_{k_{f}}^{L_{f}}[F^{\prime}((p_{m,k}^{H})^{j})]\,  \delta (p_{m,k}^{H})^{j}, v_{m}^{H} \rangle &=& - \langle P_{k}^{L}[F((p_{m,k}^{H})^{j})], v_{m}^{H} \rangle \quad  \forall v_{m}^{H} \in V_{m}^{H}, \quad j = 0,1,2,... \\[-1.5ex]
\label{newton_epm}&&\\[-1.5ex]
 \nonumber  \qquad (p_{m,k}^{H})^{j+1} &=& (p_{m,k}^{H})^{j} + \delta (p_{m,k}^{H})^{j},
\end{eqnarray}
where $(p_{m,k}^{H})^{0}$ is a suitable initial datum and $P_{k_{f}}^{L_{f}}[F^{\prime}(p_{m,k}^{H})]$ is computed analogous to $P_{k}^{L}[F(p_{m,k}^{H})]$ with the adaptive EPM. For well-posedness of the Newton scheme for nonlinear PDEs in general we refer to \cite{DenSch96} and for this particular framework to \cite{Sme13}. \\

For future reference we finally introduce a two-dimensional finite element solution which will serve as a reference for our approximation. To this end we introduce the subdivision ${T} := \mathcal{T}_{H} \times \tau_{h}$ of ${\Omega}$ with elements $T_{i,j} := \mathcal{T}_{i} \times \tau_{j}$, $\mathcal{T}_{i} \in \mathcal{T}_{H}$ and $\tau_{j} \in \tau_{h}$, and the reference  FE-space
\begin{equation}\label{tensor-FEM-nonlin}
V^{H\times h} := \left\{ v^{H \times h} \in C^{0}({\Omega}) \, \mid \, v^{H \times h}|_{T_{i,j}} \in \mathbb{Q}_{k,l}, T_{i,j} \in {T} \right\}.
\end{equation}
Here, $\mathbb{Q}_{k,l}$ is defined as $
\mathbb{Q}_{k,l}:= \{ \sum_{j} c_{j} v_{j}(x)w_{j}(y) \enspace : \enspace v_{j} \in \mathbb{P}^{1}_{k}, w_{j} \in \mathbb{P}^{1}_{l}  \},
$ and $\mathbb{P}^{1}_{l}$ denotes the space of polynomials of order $\leq l$ in one variable. 
We will see in Section \ref{1dproblem-nonlin} and \ref{adapt-RB-HMR-epm} that we have for the RB-HMR approach $Y_{m} \subset Y^{h}$ and as a consequence $V_{m}^{H} \subset V^{H\times h}$.  The reference FE approximation of problem \eqref{fullprob_nonlin2} reads:
\begin{equation}
 \label{truth_nonlin}
\text{Find} \enspace p^{H \times h} \in V^{H\times h}: \quad \la F(p^{H \times h}),v^{H \times h} \ra = 0 \quad \forall \, v^{H \times h} \in V^{H \times h},
\end{equation}
where $\la F(p^{H \times h}),v^{H \times h} \ra = \la A(p^{H \times h}),v^{H \times h} \ra - \la f, v^{H \times h} \ra$ for all $v^{H\times h} \in V^{H\times h}$.

\subsection{Derivation of a parametrized 1D problem in transverse direction}\label{1dproblem-nonlin}
To derive a lower dimensional parametrized PDE in the transverse direction we proceed as in \cite{OS10} and assume that
\begin{equation}\label{one tensor}
p(x,{y}) \approx U(x) \cdot \mathcal{P}({y}),
\end{equation}
where the function $U(x)$ represents the unknown behavior of the full solution in the dominant direction. Using the test functions $v(x,y) = U(x) \cdot \upsilon({y})$ for all $\upsilon \in H^{1}_{0}({\omega})$ yields the reduced problem with quadrature
\begin{equation}\label{1D_prob_quad_nonlin}
\text{Given any} \enspace U \in H^{2}_{0}(\Omega_{1D}), \enspace \text{find} \enspace  \mathcal{P} \in H^{1}_{0}({\omega}): \quad \langle A(U \mathcal{P}), U \upsilon \rangle^{q} = \langle f, U \upsilon \rangle^{q} \enspace \forall \upsilon \in H^{1}_{0}({\omega}).
\end{equation}
Here, we denote by $\langle \cdot , \cdot \rangle^{q}$ the approximation  obtained by substituting the integral $I(t):= \int_{{\omega}} \int_{\Omega_{1D}} t(x,{y})\,  dx d{y}$ in $\langle \cdot , \cdot \rangle$ by the quadrature formula
\begin{equation}\label{quad_formula_nonlin}
Q(t) := \sum_{l = 1}^{Q} \alpha_{l} \int_{{\omega}} t(x_{l}^{q},{y}) \, \, d{y}, 
\end{equation}
where $\alpha_{l}$, $x^{q}_{l}$, $l = 1,...,Q$ denote quadrature weights and points respectively.
Note that we require $U \in H^{2}_{0}(\Omega_{1D})$ to facilitate point evaluations of $U$ and thus obtain the well-definedness of \eqref{1D_prob_quad_nonlin}. Note also that this assumption is reasonable in the sense that in many cases we have indeed that the solution $p$ belongs to a Sobolev space of higher order (see for instance \cite{CalRap1997}). 
To include the unknown dynamics in dominant direction $U$ in the lower-dimensional problem in transverse direction and to find optimal locations of the quadrature points  with RB methods (see \S \ref{adapt-RB-HMR-epm} below), we parametrize \eqref{1D_prob_quad_nonlin} by 
introducing a parameter vector $\mu = \left(x_{l}^{q}, U(x_{l}^{q}),\partial_{x} U(x_{l}^{q})\right)_{l = 1,...,Q}$.
The $P$-dimensional parameter space $\mathcal{D}$ containing all admissible parameter values of $\mu$, is defined as $\mathcal{D} := [\Omega_{1D} \times I_{0} \times I_{1} ]^{Q}$, where the intervals $I_{k} \subset \mathbb{R}$ contain the ranges of $\partial^{k}_{x}U(x)$, $k = 0,1$. Compared to the linear setting, we expect a greater sensitivity of the RB-HMR approach with respect to the choice of the intervals $I_{k} \subset \mathbb{R}$, $k=0,1$, as the nonlinearity of $A$ also applies to the parameter via the term $A(U \mathcal{P})$. This can indeed be observed in the numerical experiments provided in \S \ref{numerics_nonlin}. To get a rough estimate on the possible ranges of $\partial^{k}_{x}U(x)$, $k = 0,1$, and therefore obtain an optimal convergence rate of the RB-HMR approach, we may for instance compute a coarse approximation of the solution $p$ of \eqref{fullprob_nonlin2}. 
Using the definition of $\mu$, problem (\ref{1D_prob_quad_nonlin}) can be recast into a parametrized 1D nonlinear PDE in  transverse direction as follows:
\begin{equation}\label{1D_prob_quad_para_nonlin}
\text{Given any} \enspace \mu \in \mathcal{D}, \enspace \text{find} \enspace  \mathcal{P}(\mu) \in H^{1}_{0}({\omega}): \la A(\mathcal{P}(\mu);\mu),\upsilon;\mu\ra^{q} = \la f(\mu), \upsilon;\mu\ra^{q} \enspace \forall \upsilon \in H^{1}_{0}({\omega}).
\end{equation}
Here, $\la \cdot,\cdot ;\mu \ra^{q}$ denotes the parameter dependent dual pairing of $H^{-1}(\omega)$ and $H^{1}_{0}(\omega)$, $A(\cdot ;\mu): H^{1}_{0}(\omega) \rightarrow H^{-1}(\omega)$, and $f(\mu) \in H^{-1}(\omega)$.
Possible choices for the quadrature formula \eqref{quad_formula_nonlin} are a modified rectangle formula or a standard composite trapezoidal rule. The number of quadrature points is chosen automatically by an adaptive algorithm, described in \S \ref{adapt-RB-HMR-epm}. To compute snapshots we use the subdivision $\tau_{h}$ of ${\omega}$ and the associated conforming FE space $Y^{h} \subset H^{1}_{0}({\omega})$ with basis $\upsilon^{h}_{j}, \, j=1,...,n_{h}$ as introduced in Section \ref{sect:EPM_EIM}. We obtain the parameter dependent discrete 1D problem:
\begin{equation}\label{disc_1dprob_nonlin}
\text{Given any} \enspace \mu \in \mathcal{D}, \enspace \text{find} \enspace \mathcal{P}^{h}(\mu) \in Y^{h}: \la A(\mathcal{P}^{h}(\mu);\mu),\upsilon^{h}_{j};\mu\ra^{q} = \la f(\mu), \upsilon^{h}_{j};\mu\ra^{q} \enspace \text{for} \enspace j = 1,...,n_{h},
\end{equation}
which can be solved by Newton's method. Well-posedness of \eqref{disc_1dprob_nonlin} and the conditions for the convergence of Newton's method may be verified a posteriori (cf.~\cite{CaToUr09,Sme13}). We may then define the solution manifold 
$\mathcal{M}^{\mathcal{P}}$ as
\begin{equation}\label{manifold-nonlin}
\mathcal{M}^{\mathcal{P}}:=\lbrace \mathcal{P}^{h}(\mu) \,|\, \mu \in   \mathcal{D}\rbrace .
\end{equation}
Finally, we remark that instead of the heuristic assumption in \eqref{one tensor} one might alternatively consider a linear combination of tensor products. Note that we would then have to consider a system of nonlinear equations in \eqref{1D_prob_quad_para_nonlin} and that the number of parameters would of course increase. How this change affects the approximation properties of the RB-HMR approach is subject of future research.

\subsection{The generation of  parametrized 1D operator evaluations}\label{gen-snap-nonlin}

In this subsection we define a manifold of operator evaluations which is formed by parametrized 1D operator evaluations of the nonlinear operator $A$ in the transverse direction. For this purpose we consider \eqref{1D_prob_quad_para_nonlin} and \eqref{disc_1dprob_nonlin} and define parametrized 1D operator evaluations $\mathcal{A}(\mu)$ and $\mathcal{A}^{h}(\mu)$ of the operators $A(p(x,{y}))$ and $A(p^{H\times h}(x,{y}))$ as 
\begin{equation}\label{op-snap-dis}
\mathcal{A}(\mu) := \sum_{l=1}^{Q} \frac{\alpha_{l}}{|\Omega_{1D}|} A(\mathcal{P}(\mu)  ;\mu_{l}), \quad \text{and} \quad \mathcal{A}^{h}(\mu) := \sum_{l=1}^{Q} \frac{\alpha_{l}}{|\Omega_{1D}|} A(\mathcal{P}^{h}(\mu)  ;\mu_{l}).
\end{equation}
Here, $|\Omega_{1D}|$ denotes the length of the interval $\Omega_{1D}$, $\mu_{l}:= (x_{l}^{q}, U(x_{l}^{q}),\partial_{x}^{k} U(x_{l}^{q}))$, $\mathcal{P}(\mu) $ is the solution of \eqref{1D_prob_quad_para_nonlin} and $\mathcal{P}^{h}(\mu) $ solves \eqref{disc_1dprob_nonlin}.  Provided that $\mathcal{P}(\mu) $ is able to capture the behavior of the full solution $p$ in the transverse direction, we expect that $\mathcal{A}(\mu)$ is a good approximation of the range of $A(p(x,{y}))$ in that direction, which will be validated in \S \ref{numerics_nonlin}. 
Moreover, we define parametrized 1D operator evaluations of the respective Fr$\acute{\rm e}$chet derivatives $A^{\prime}(p(x,{y}))$ and $A^{\prime}(p^{H\times h}(x,{y}))$ as 
\begin{equation}\label{op-snap-dis-fre}
\mathcal{A}^{\prime}(\mu) := \sum_{l=1}^{Q} \frac{\alpha_{l}}{|\Omega_{1D}|} A^{\prime}(\mathcal{P}(\mu)  ;\mu_{l}), \quad 
\text{and} \quad (\mathcal{A}^{h})^{\prime}(\mu) := \sum_{l=1}^{Q} \frac{\alpha_{l}}{|\Omega_{1D}|} A^{\prime}(\mathcal{P}^{h}(\mu)  ;\mu_{l}).
\end{equation}
Finally, we define a manifold of operator evaluations $\mathcal{M}^{\mathcal{A}}$ through 
\begin{equation}\label{operator-manifold}
\mathcal{M}^{\mathcal{A}}:=\lbrace \mathcal{A}^{h}(\mu) \,|\, \mu \in   \mathcal{D}\rbrace .
\end{equation}

\subsection{Reduced and collateral basis generation --- the \textsc{Adaptive-RB-HMR} algorithm}\label{adapt-RB-HMR-epm}
In this subsection we introduce the \textsc{Adaptive-RB-HMR} algorithm which simultaneously
constructs the reduction space $Y_{m}=\spanlin \{\phi_{1}, \dots, \phi_{m}\} \subset Y^{h}$ and the collateral basis space $W_{k}=\spanlin \{\kappa_{1},\dots,\kappa_{k}\}$ using sampling strategies from the RB framework. First, the snapshot sets
\begin{equation}\label{disc_mani}
\mathcal{M}^{\mathcal{P}}_{\Xi} := \lbrace \mathcal{P}^{h}(\mu)\, |\, \mu \in  \Xi \rbrace \subset \mathcal{M}^{\mathcal{P}}, \quad \text{and} \quad \mathcal{M}^{\mathcal{A}}_{\Xi}  := \lbrace \mathcal{A}^{h}(\mu) \,|\, \mu \in  \Xi \rbrace \subset \mathcal{M}^{\mathcal{A}}, \enspace \Xi \subset \mathcal{D},  
\end{equation}
are efficiently constructed in Algorithm \ref{adapt-train} by an adaptive training set extension which generalizes the algorithm proposed in \cite{OS10}. Subsequently, we apply a POD to determine the principal components of $\mathcal{M}^{\mathcal{P}}_{\Xi}$ and $\mathcal{M}^{\mathcal{A}}_{\Xi}$ which in turn span the reduction space $Y_{m}$ and the collateral basis space $W_{k}$, respectively. \\

\paragraph*{Algorithm \ref{adapt-train} (\textsc{AdaptiveTrainExtension})}
Let $G$ denote a hyper-rectangular possibly non-conforming grid in the parameter space $\mathcal{D}$, $g$ a cell of $G$ and $N_{G}$ the number of cells in $G$. The parameter values in the training set $\Xi_{g}$ are sampled from the uniform distribution over the cell $g$, where $\Xi_{g}$ has the same size $n_{\Xi}$ for all cells $g$ and $\Xi_{G} = \cup_{g \in G} \Xi_{g}$. As in  \cite{OS10} and originally in  \cite{HaaOhl2008b,HaaDihOhl2011} we use a local mesh adaptation with a
$\mbox{SOLVE} \rightarrow \mbox{ESTIMATE} \rightarrow \mbox{MARK} \rightarrow \mbox{REFINE}$
strategy for the generation of $G$ and $\Xi_{G}$ beginning with a given coarse partition $G_{0}$ and an associated initial training set $\Xi_{G_{0}}$. In \S \ref{sect-apostest} we derive an a posteriori error estimate $\Delta_{m}^{k}$ for the error between the solution $p_{m,k}^{H'}$ of \eqref{red_prob_hmr_epm} and the reference solution $p^{H'\times h}$ defined in \eqref{truth_nonlin} which takes into account both the model error and the error due to the approximation of the nonlinear operator. For the latter we use the a posteriori bound for the EPM derived in Proposition \ref{apost-epm}. A richer collateral basis space $W_{k'}$ with associated interpolating functionals $\mathcal{S}_{k'}^{\mathfrak{I}}$ has thus to be provided before starting the $\mbox{SOLVE} \rightarrow \mbox{ESTIMATE} \rightarrow \mbox{MARK} \rightarrow \mbox{REFINE}$-loop. Therefore, we initially compute the snapshots $\mathcal{P}^{h}_{c}$ and $\mathcal{A}^{h}_{c}$ for a coarse train sample $\Xi_{c}$ of $G_{0}$ with $|\Xi_{c}| = N_{G_{0}}n_{c}$ in line \ref{coarse_snapshots} in Algorithm \ref{adapt-train}. To compute $W_{k'}$ and the functionals $\mathcal{S}_{k'}^{\mathfrak{I}}$ with Algorithm \ref{k_strich} \textsc{EPM-Indicator}, we first use a POD  to find the principal components $\{\kappa_{l}\}_{l=1}^{k_{POD}}$ such that the POD-error $e^{\mbox{{\tiny POD}}}_{k_{POD}}=(\sum_{l = k_{POD} + 1}^{N_{G_{0}}n_{c}} \lambda_{l})^{1/2}  \leq \varepsilon_{\mbox{{\scriptsize tol}}}^{err}$, where $\varepsilon_{\mbox{{\scriptsize tol}}}^{err} < \varepsilon_{\mbox{{\tiny tol}}}^{\text{{\tiny EPM}}}$. Note that $\varepsilon_{\mbox{{\scriptsize tol}}}^{err}$ has to be chosen rather small to obtain an a posteriori error estimate for the EPM which is as accurate as possible. Next, we apply Algorithm \ref{adapt-EPM} \textsc{adaptive EPM} for the computation of the interpolating functionals $\mathcal{S}_{k_{POD}}^{\mathfrak{I}}$, the basis $\mathcal{Q}_{k_{POD}}^{\mathfrak{I}}$ and the matrix $B_{k_{POD}}^{\mathfrak{I}}$, where the computation of the interpolant in \eqref{epm-approx} necessitates the solution of \eqref{red_prob_hmr_epm} and thus the computation of $\{\phi_{l}\}_{l=1}^{m_{\mbox{{\scriptsize max}}}}$ in line \ref{basis_line}. As the error bound is only employed during the adaptive training set extension, it is sufficient to restrict to $m = m_{\mbox{{\scriptsize max}}}$. Then we use the a priori bound for the EPM from Theorem \ref{apriori-epm} to compute $k'$, which yields $W_{k'}$ and apply again Algorithm \ref{adapt-EPM} to determine $\mathcal{S}_{k'}^{\mathfrak{I}}$. Here, $\varepsilon^{c}_{\text{tol}}$ denotes the tolerance for the POD employed to compute the collateral basis of size $k_{c}$ within Algorithm \ref{adapt-train}. The factor $\text{tol}_{k'}$ results in a smaller tolerance $\text{tol}_{k'}\cdot \varepsilon_{\mbox{{\scriptsize tol}}}^{c}$ for the POD which is used to compute the collateral basis of size $k'$ solely for error estimator purposes and ensures $k'>k_{c}$. $W_{k'}$ and $\mathcal{S}_{k'}^{\mathfrak{I}}$ are updated at the end of each loop over $m$ in line \ref{update_kstrich} to include the information from the snapshots generated during lines \ref{snap1} and \ref{snap2}. \\
\begin{algorithm}[t]
\LinesNotNumbered
\caption{Computation of $k'$ \label{k_strich}}
\textsc{EPM-Indicator}$(\mathcal{P}^{h}_{G},\mathcal{A}_{G}^{h}, \varepsilon_{\mbox{{\scriptsize tol}}}^{err},\text{tol}_{k'}, \varepsilon_{\mbox{{\scriptsize tol}}}^{c}, \varepsilon_{\mbox{{\tiny tol}}}^{\mbox{{\tiny int}}}$, $N_{\mbox{{\tiny max}}}^{\mbox{{\tiny int}}}, m_{\mbox{{\scriptsize max}}},N_{H'}, \Xi_{G}, \Sigma)$\\
$\{\phi_{l}\}^{m_{\mbox{{\scriptsize max}}}}_{l=1}:= \mbox{POD}(\mathcal{P}_{G}^{h}, m_{\mbox{{\scriptsize max}}})$\label{basis_line}\\ 
$[\{\kappa_{l}\}^{k_{POD}}_{l=1},\{\lambda_{l}\}^{k_{POD}}_{l=1}]:= \mbox{POD}(\mathcal{A}_{G}^{h}, \varepsilon_{\mbox{{\scriptsize tol}}}^{err})$\\
$[\mathcal{S}_{k_{POD}}^{\mathfrak{I}}, \mathcal{Q}_{k_{POD}}^{\mathfrak{I}}, B_{k_{POD}}^{\mathfrak{I}}, e_{\mbox{{\tiny int}}}] :=$\textsc{adaptive EPM}$(\{\kappa_{l}\}^{k_{POD}}_{l=1},\Sigma, \mathcal{A}_{G}^{h}$,$\varepsilon_{\mbox{{\tiny tol}}}^{\mbox{{\tiny int}}}$, $N_{\mbox{{\tiny max}}}^{\mbox{{\tiny int}}}$, $\Xi_{G}$, $ \{\phi_{l}\}^{m_{POD}}_{l=1},N_{H'})$\\
$k' =$\textsc{EPM-Aposteriori-Bound}$(e_{\mbox{{\tiny int}}},\text{tol}_{k'}\cdot \varepsilon_{\mbox{{\scriptsize tol}}}^{c},\{\lambda_{l}\}^{k_{POD}}_{l=1})$\\
$W_{k'}:= \{\kappa_{l}\}^{k'}_{l=1}$\\
$\mathcal{S}_{k'}^{\mathfrak{I}} :=$\textsc{adaptive EPM}$(\{\kappa_{l}\}^{k'}_{l=1}$, $\Sigma$, $\mathcal{A}_{G}^{h}$, $\varepsilon_{\mbox{{\tiny tol}}}^{\mbox{{\tiny int}}}$, $N_{\mbox{{\tiny max}}}^{\mbox{{\tiny int}}}$, $\Xi_{G}$, $\{\phi_{l}\}^{m_{POD}}_{l=1},N_{H'})$\\
 \Return $W_{k'}, \mathcal{S}_{k'}^{\mathfrak{I}}$
\end{algorithm}

\begin{algorithm}[t]
\LinesNotNumbered
\caption{Adaptive training set extension and snapshot generation \label{adapt-train}}
\textsc{AdaptiveTrainExtension}$(G_{0},\Xi_{G_{0}},\Xi_{c}, m_{\mbox{{\scriptsize{max}}}},i_{max},n_{\Xi},\theta,\sigma_{thres},N_{H'},...$\\ \hspace{200pt}$...\varepsilon_{\mbox{{\scriptsize tol}}}^{err}, \varepsilon_{\mbox{{\scriptsize tol}}}^{c}, \text{tol}_{k'}, \varepsilon_{\mbox{{\tiny tol}}}^{\mbox{{\tiny int}}}, N_{\mbox{{\tiny max}}}^{\mbox{{\tiny int}}},Q_{0},Q_{\mbox{{\scriptsize{max}}}}, \Sigma)$\\
\textbf{Initialize}  $G = G_{0}, \Xi_{G} = \Xi_{G_{0}}, \phi_{0} = \emptyset, \kappa_{0} = \emptyset, \rho(G) = 0, Q = Q_{0}$\\
Compute $\mathcal{P}_{c}^{h}(Q)$, $\mathcal{A}_{c}^{h}(Q)$ \label{coarse_snapshots}\\
$[W_{k'}, \mathcal{S}_{k'}^{\mathfrak{I}}]=$\textsc{EPM-Indicator}$(\mathcal{P}^{h}_{c},\mathcal{A}_{c}^{h}, \varepsilon_{\mbox{{\scriptsize tol}}}^{err},\text{tol}_{k'}, \varepsilon_{\mbox{{\scriptsize tol}}}^{c}, \varepsilon_{\mbox{{\tiny tol}}}^{\mbox{{\tiny int}}}$, $N_{\mbox{{\tiny max}}}^{\mbox{{\tiny int}}}, m_{\mbox{{\scriptsize max}}},N_{H'},\Xi_{c}, \Sigma)$\\
$Q = $\textsc{QP-Indicator}$(\mathcal{P}_{c}^{h},N_{H'},W_{k'}, \mathcal{S}_{k'}^{\mathfrak{I}},Q_{\text{{\tiny max}}})$\label{qp1}\\
Possibly adapt $G$ and $\Xi_{G}$ if $Q$ has changed.\\
\For{$m=1,\dots,m_{\mbox{{\scriptsize max}}}$}{
Compute $\mathcal{P}_{G}^{h}(Q)$, $\mathcal{A}_{G}^{h}(Q)$\label{snap1}\\
$[\eta(G), \sigma(G)]=$\textsc{ElementIndicators}$(\{\phi_{k}\}^{m-1}_{k=1},\mathcal{P}_{G}^{h},\{\kappa_{k}\}^{k_{c}}_{k=1},\mathcal{A}_{G}^{h},W_{k'},\mathcal{S}_{k'}^{\mathfrak{I}},G,\rho(G),N_{H'})$\\
\For{$i = 1,\dots,i_{max}$}{
$\mathcal{G} :=$ \textsc{Mark}$(\eta(G),\sigma(G),\theta,\sigma_{thres})$\\
$(G,\Xi_{G}):=$ \textsc{Refine}$(\mathcal{G},\Xi_{\mathcal{G}},n_{\Xi})$\\
$\rho(G\setminus\mathcal{G}) = \rho(G\setminus\mathcal{G}) + 1$\\
Compute $\mathcal{P}_{\mathcal{G}}^{h}(Q)$, $\mathcal{A}_{\mathcal{G}}^{h}(Q)$\label{snap2}\\
$[\eta(\mathcal{G}),\rho(\mathcal{G}),\sigma(\mathcal{G})] =$\textsc{ElementIndicators}$(\{\phi_{k}\}^{m-1}_{k=1},\mathcal{P}_{\mathcal{G}}^{h},\{\kappa_{k}\}^{k_{c}}_{k=1},\mathcal{A}_{\mathcal{G}}^{h},W_{k'},\mathcal{S}_{k'}^{\mathfrak{I}},N_{H'})$}
 $\{\phi_{k}\}^{m}_{k=1}:= \mbox{POD}(\mathcal{P}_{G}^{h},m)$\\
$\{\kappa_{k}\}^{k_{c}}_{k=1}:= \mbox{POD}(\mathcal{A}_{G}^{h},\varepsilon_{\mbox{{\scriptsize tol}}}^{c})$ \label{POD-op}\\
$[W_{k'}, \mathcal{S}_{k'}^{\mathfrak{I}}]=$\textsc{EPM-Indicator}$(\mathcal{P}^{h}_{G},\mathcal{A}_{G}^{h}, \varepsilon_{\mbox{{\scriptsize tol}}}^{err},\text{tol}_{k'}, \varepsilon_{\mbox{{\scriptsize tol}}}^{c}, \varepsilon_{\mbox{{\tiny tol}}}^{\mbox{{\tiny int}}}$, $N_{\mbox{{\tiny max}}}^{\mbox{{\tiny int}}}, m_{\mbox{{\scriptsize max}}},N_{H'},\Xi_{G}, \Sigma)$\label{update_kstrich}\\
$Q = $\textsc{QP-Indicator}$(\mathcal{P}_{G}^{h},N_{H'},W_{k'}, \mathcal{S}_{k'}^{\mathfrak{I}},Q_{\text{{\tiny max}}})$\label{qp2}\\
Possibly adapt $G$ and $\Xi_{G}$ if $Q$ has changed.
}
 \Return $\mathcal{M}^{\mathcal{P}}_{\Xi}, \mathcal{M}^{\mathcal{A}}_{\Xi},\Xi_{G}$
\end{algorithm}

A main difference to the Algorithm in \cite{OS10} is the usage of the \textsc{QP-Indicator}, which chooses the number of quadrature points $Q$ used in \eqref{disc_1dprob_nonlin}. To decide whether $Q$ has to be increased or not we apply a POD 
to $\mathcal{P}^{h}_{c}$ in line \ref{qp1} ($\mathcal{P}^{h}_{G}$ in line \ref{qp2})
and compare the convergence rates of the eigenvalues of the POD with $\|\bar{p}_{l,k'}^{H'}\|_{L^{2}(\Omega_{1D})}^{2}$, $l=1,..,10$, where the coefficients $\bar{p}_{l,k'}^{H'} \in X^{H'}$ solve \eqref{red_prob_hmr_epm}. If we observe that the decay rate of the coefficients is worse than the rate of the eigenvalues by at least $50\%$ on $5$ consecutive values\footnote{This can be verified by comparing the slope of the tangents.}, and $Q$ is smaller than $Q_{\text{{\tiny max}}}$, we increment $Q$ by one. Note that we want to increment $Q$ only if we observe a significant deviation of the coefficients from the eigenvalues, which is why we proceed rather conservatively. Note also that the \textsc{QP-Indicator} thus enforces the adaptation of the reduction space $Y_{m}$ and the collateral basis space $W_{k}$ to the reference solution $p^{H\times h}$ and the nonlinear operator $A(p^{H\times h})$ by increasing the amount of information on the dynamics in the dominant direction in the spaces $Y_{m}$ and $W_{k}$, if necessary. The initial value $Q_{0}$ is usually set to $1$. Note that the fact that $G$ is a product-like hyper-rectangular grid prevents the applicability of Algorithm \ref{adapt-train} to high parameter dimensions.  However, if $\dim (\mathcal{D}) \gg 1$ we may instead consider an anisotropic adaptive refinement strategy or use a clustering algorithm (cf. \cite{PeBuWiBu13}).\\
Apart from the just stated differences and the additional computation of the snapshots $\mathcal{A}^{h}_{G}$ in line \ref{snap1} and \ref{snap2}, and the POD for the computation of the small collateral basis $\{\kappa_{k}\}^{k_{c}}_{k=1}$ in line \ref{POD-op}, Algorithm \ref{adapt-train} follows the lines of the corresponding Algorithm in \cite{OS10}. Thus, we use the cell indicators  $ \eta(g) := \min_{\mu \in \Xi_{g}} \Delta_{m}^{k}(\mu)$ and $\sigma(g) := \diam(g)\cdot \rho(g)$, where $\rho(g)$ counts the number of loops in which the cell $g$ has not been refined, since its last refinement. We mark for fixed $\theta \in (0,1]$ in each iteration the $\theta N_{G}$ cells $g$ with the smallest indicators $\eta(g)$ and additionally the cells for which $\sigma(g)$ lies above a certain threshold $\sigma_{thres}$. Then, all cells marked for refinement are bisected in each direction. Finally, we note that for each parameter value in $\Xi_{G}$ we compute the snapshots $\mathcal{P}^{h}(\mu)$ and $\mathcal{A}^{h}(\mu)$, add these snapshots to the already computed small bases $\{\phi_{l}\}_{l=1}^{m-1}$ and $\{\kappa_{l}\}_{l=1}^{k_{c}}$, compute the (coarse) solution $p^{H'}_{m,k}$ of \eqref{red_prob_hmr_epm}, and use the a posteriori error estimator to assess whether the span of the small bases and the current snapshots yields a good approximation. Note that both for the computation of $p^{H'}_{m,k}$ and the error estimator within the adaptive refinement procedure we employ a coarser discretization in the dominant direction with a mesh size $H'$ and an associated coarser finite element space $X^{H'}$ of dimension $N^{H'}$. \\

\paragraph*{Algorithm \ref{adapt-RB-HMR} (\textsc{Adaptive-RB-HMR})}
At first, the training sets $\Xi_{G_{0}}$ and $\Xi_{c}$ are formed by sampling $n_{\Xi}$ or $n_{c}$ parameter values from the uniform distribution over each $g \in G_{0}$, where $n_{c} > n_{\Xi}$. Subsequently Algorithm \ref{adapt-train} is called to generate the discrete manifolds $\mathcal{M}^{\mathcal{P}}_{\Xi}$ \eqref{manifold-nonlin} and $\mathcal{M}^{\mathcal{A}}_{\Xi}$ \eqref{operator-manifold}. Finally, we apply a POD to determine the principal components $\{ \phi_{1},...,\phi_{m}\}$ and  $\{ \kappa_{1},...,\kappa_{k}\}$ of $\mathcal{M}^{\mathcal{P}}_{\Xi}$ and $\mathcal{M}^{\mathcal{A}}_{\Xi}$, which then span the reduction space $Y_{m}$ and the collateral basis space $W_{k}$. \\
Finally, we point out that the spaces $Y_{m}$ and $W_{k}$ approximate the discrete manifolds $\mathcal{M}^{\mathcal{P}}_{\Xi}$ and $\mathcal{M}^{\mathcal{A}}_{\Xi}$. However, thanks to the design of the parametrized 1D problem and the parametrized operator evaluations we expect that our choices of $Y_{m}$ and $W_{k}$ also allow for a good approximation of  the reference solution $p^{H\times h}$ and the range of the operator $A(p^{H\times h})$. This is demonstrated in \S \ref{numerics_nonlin}. For details on the choice of the input parameters $m_{max}, i_{max}, n_{\Xi}$, $\sigma_{thres}$ and $N_{H'}$ we refer to \cite{OS10}.
\begin{algorithm}[t]
\LinesNotNumbered
\caption{Construction of the reduction space $Y_{m}$ and the collateral basis space $W_{k}$\label{adapt-RB-HMR}}
\textsc{Adaptive-RB-HMR}$(G_{0},m_{\text{{\tiny max}}},i_{\text{{\tiny max}}},n_{\Xi},n_{c},\theta,\sigma_{\text{{\tiny thres}}},N_{H'},\varepsilon_{\mbox{{\tiny tol}}}^{\text{{\tiny HMR}}},\varepsilon_{\mbox{{\tiny tol}}}^{\text{{\tiny EPM}}},...$\\
\hspace{200pt}$...\varepsilon_{\mbox{{\tiny tol}}}^{\text{{\tiny err}}},\varepsilon_{\mbox{{\tiny tol}}}^{c},\text{tol}_{k'},\varepsilon_{\mbox{{\tiny tol}}}^{\mbox{{\tiny int}}},N_{\mbox{{\tiny max}}}^{\mbox{{\tiny int}}},Q_{0},Q_{\mbox{{\tiny{max}}}},\Sigma)$\\
\textbf{Initialize} $\Xi_{G_{0}}$, $\Xi_c$\\
$[\mathcal{M}^{\mathcal{P}}_{\Xi}, \mathcal{M}^{\mathcal{A}}_{\Xi},\Xi_{G}]=$ \textsc{AdaptiveTrainExtension}$(G_{0},\Xi_{G_{0}},\Xi_{c}, m_{\text{{\tiny max}}},i_{\text{{\tiny max}}},n_{\Xi},... $\\
\qquad \qquad \qquad \qquad \qquad \qquad \qquad \qquad \qquad $... \theta,\sigma_{\text{{\tiny thres}}},N_{H'},\varepsilon_{\mbox{{\tiny tol}}}^{\text{{\tiny err}}}, \varepsilon_{\mbox{{\scriptsize tol}}}^{c}, \text{tol}_{k'},\varepsilon_{\mbox{{\tiny tol}}}^{\mbox{{\tiny int}}}, N_{\mbox{{\tiny max}}}^{\mbox{{\tiny int}}}, Q_{0},Q_{\mbox{{\tiny{max}}}},\Sigma)$\\
$Y_{m}:= \mbox{POD}(\mathcal{P}_{G}^{h},\varepsilon_{\mbox{{\tiny tol}}}^{\text{{\tiny HMR}}})$, such that $e^{\mbox{{\tiny POD}}}_{m}\leq\varepsilon_{\mbox{{\tiny tol}}}^{\text{{\tiny HMR}}}$.\\
$W_{k}:= \mbox{POD}(\mathcal{A}_{G}^{h},\varepsilon_{\mbox{{\tiny tol}}}^{\text{{\tiny EPM}}})$, such that $e^{\mbox{{\tiny POD}}}_{k}\leq\varepsilon_{\mbox{{\tiny tol}}}^{\text{{\tiny EPM}}}$. \\
$\mathcal{S}_{k}^{\mathfrak{I}} = $\textsc{adaptive EPM}$(W_{k},\mathcal{A}_{G}^{h},\varepsilon_{\mbox{{\tiny tol}}}^{\mbox{{\tiny int}}},N_{\mbox{{\tiny max}}}^{\mbox{{\tiny int}}},\Xi_{G}, Y_{m},N_{H'})$\\
\Return $Y_{m}$, $W_{k}$, $\mathcal{S}_{k}^{\mathfrak{I}}$
\end{algorithm}

\subsection{A posteriori error estimates}\label{sect-apostest}\label{subsect-BRR}\label{subsect_constants}
We apply the Brezzi-Rappaz-Raviart (BRR) theory \cite{BRR81,CalRap1997} to  derive a rigorous a posteriori error bound for the error between the reduced solution $p_{m,k}^{H}$ of \eqref{red_prob_hmr_epm} and a reference solution $p^{H\times h}$ \eqref{truth_nonlin}, which takes into account both the contributions of the model reduction and the approximation of the nonlinear operator. 
To this end we first define the inf-sup stability factor and the continuity and the Lipschitz constant: 
\begin{align}
\label{inf-sup-apost} \beta_{p} &:= \underset{|w^{H\times h}|_{W^{1,p}(\Omega)} \neq 0}{\underset{w^{H\times h} \in V^{H\times h}}{\inf}} \underset{|v^{H\times h}|_{W^{1,q}(\Omega)} \neq 0}{\underset{v^{H\times h} \in V^{H\times h}}{\sup}} \frac{\langle F^{\prime}(p_{m,k}^{H}) w^{H\times h}, v^{H\times h} \rangle_{W^{-1,p}(\Omega) W^{1,q}(\Omega)}}{| w^{H\times h}|_{W^{1,p}(\Omega)} | v^{H\times h}|_{W^{1,q}(\Omega)}},\\ 
\label{continuity-apost}
\gamma_{p} &:= \underset{|w^{H\times h}|_{W^{1,p}(\Omega)} \neq 0}{\underset{w^{H\times h} \in V^{H\times h}}{\sup}} \underset{|v^{H\times h}|_{W^{1,q}(\Omega)} \neq 0}{\underset{v^{H\times h} \in V^{H\times h}}{\sup}} \frac{\langle F^{\prime}(p_{m,k}^{H}) w^{H\times h}, v^{H\times h} \rangle_{W^{-1,p}(\Omega) W^{1,q}(\Omega)}}{| w^{H\times h}|_{W^{1,p}(\Omega)} | v^{H\times h}|_{W^{1,q}(\Omega)}},\\
\label{Lipschitz-apost}
L_{p} &:= \underset{w^{H\times h} \in B(p^{H}_{m,k},R)}{\sup} \frac{\| F^{\prime}(w^{H\times h}) - F^{\prime}(p^{H}_{m,k}) \|_{W^{1,p}(\Omega),W^{-1,p}(\Omega)}}{|\, w^{H\times h} - p^{H}_{m,k}\,|_{W^{1,p}(\Omega)}},
\end{align}
where $B(p_{m,k}^{H},R):=\{w^{H\times h} \in V^{H\times h} : \, |\,z- p_{m,k}^{H}\,|_{W^{1,p}(\Omega)} \leq R, R \in \mathbb{R}^{+}\}$ and the index $p$ comes from the space $W^{1,p}(\Omega)$. Note that we compute the Lipschitz constant only on $B(p_{m,k}^{H},R)$ both in order to obtain a sharper estimate and to include  nonlinear operators whose Fr\'{e}chet derivative is not Lipschitz continuous on the whole space $V^{H\times h}$. We comment in Section \ref{subsubsect: estimate constants} on how we may obtain estimates for the constants defined in \eqref{inf-sup-apost}-\eqref{Lipschitz-apost}. Now we may define a proximity indicator \cite{VerPat05,CaToUr09}
$
\tau_{m,p}^{k} := \frac{2L_{p}}{\beta_{p}^{2}} ( \| F(p_{m,k}^{H}) - P_{k}^{L}[F(p_{m,k}^{H})] \|_{W^{-1,p}(\Omega)} + \| P_{k}^{L}[F(p_{m,k}^{H})] \|_{W^{-1,p}(\Omega)} )
$
and obtain the following result.
\begin{proposition}[A rigorous a posteriori error bound]\label{error_bound}  Let $2 \leq p < \infty$. If $\tau_{m,p}^{k} < 1$ then there exists a unique solution $p^{H\times h} \in B(p_{m,k}^{H},\frac{\beta_{p}}{L_{p}})$ of \eqref{truth_nonlin} and the following a posteriori error estimate holds
\begin{equation}\label{delta_BBR}
| p^{H\times h} - p_{m,k}^{H} |_{W^{1,p}(\Omega)} \leq \Delta_{m,p}^{k}  := \frac{\beta_{p}}{L_{p}}(1 - \sqrt{1 - \tau_{m,p}^{k}}).
\end{equation}
\end{proposition}
\begin{proof}
The proof follows the ideas of \cite{CaToUr09}, which in turn is based on \cite{VerPat05,CalRap1997}.   
\end{proof}

Next, we analyze as in \cite{CaToUr09} the effectivity $\Delta^{k}_{m,p}/|p^{H\times h} - p_{m,k}^{H}|_{W^{1,p}(\Omega)}$ of the error bound \eqref{delta_BBR}.

\begin{proposition}[Effectivity]\label{eff}
Let $2 \leq p < \infty$ and let us assume that 
\begin{equation}\label{mod_dom_ei}
\| F(p_{m,k}^{H}) - P_{k}^{L}[F(p_{m,k}^{H})] \|_{W^{-1,p}(\Omega)} \leq c_{\text{{\tiny err}}} \| P_{k}^{L}[F(p_{m,k}^{H})] \|_{W^{-1,p}(\Omega)}
\end{equation}
for $c_{\text{{\tiny err}}} \in [0,1)$ and set 
$
C_{\text{{\tiny err}}} := \frac{1 - c_{\text{{\tiny err}}}}{1 + c_{\text{{\tiny err}}}}. 
$
If $\tau_{m,p}^{k} \leq \frac{1}{2} C_{\text{{\tiny err}}}$ we have
\begin{equation}\label{effectivty}
\Delta_{m,p}^{k} \leq 4 C_{\text{{\tiny err}}}^{-1} \,\frac{\gamma_{p}}{\beta_{p}}\, |p^{H\times h} - p_{m,k}^{H}|_{W^{1,p}(\Omega)}.
\end{equation}
\end{proposition}
\begin{proof}
We simplify notations by setting $\la \cdot , \cdot \ra := \la \cdot , \cdot \ra_{W^{-1,p}(\Omega) W^{1,q}(\Omega)}$. It is easy to see (cf.~\cite{CaToUr09}) that \eqref{mod_dom_ei} implies
\begin{align}
\nonumber \| P_{k}^{L}[F(p_{m,k}^{H})] &\|_{W^{-1,p}(\Omega)} + \| F(p_{m,k}^{H}) - P_{k}^{L}[F(p_{m,k}^{H})] \|_{W^{-1,p}(\Omega)} \\[-1.5ex]
\label{mo_ei_verh}\\[-1.5ex]
\nonumber \qquad \qquad &\leq C_{\text{{\tiny err}}}^{-1} \bigl (\| P_{k}^{L}[F(p_{m,k}^{H})] \|_{W^{-1,p}(\Omega)} - \| F(p_{m,k}^{H}) - P_{k}^{L}[F(p_{m,k}^{H})] \|_{W^{-1,p}(\Omega)} \bigr). 
\end{align}
The following estimate differs from \cite{CaToUr09}, as in \cite{CaToUr09} a quadratic nonlinear PDE in a Hilbert space is considered and the proof of the effectivity of the error bound heavily relies on these two assumptions. As $\tau_{m,p}^{k} \leq \frac{1}{2} C_{\text{{\tiny err}}} \leq 1$ we may apply Proposition \ref{error_bound} to obtain 
\begin{align*}
  \bigl \la F(p_{m,k}^{H}) - P_{k}^{L}[F(p_{m,k}^{H})], &v^{H\times h} \bigr \ra + \left\la P_{k}^{L}[F(p_{m,k}^{H})], v^{H\times h} \right \ra 
= - \la F^{\prime}(p_{m,k}^{H}) (p^{H\times h} - p_{m,k}^{H}) , v^{H\times h}  \ra  \\
&+ \left \la \int_{0}^{1} \left \{ F^{\prime}(p_{m,k}^{H}) - F^{\prime}(p_{m,k}^{H} + t(p^{H\times h} - p_{m,k}^{H})) \right \}  (p^{H\times h} - p_{m,k}^{H}) \, dt , v^{H\times h} \right \ra .
\end{align*}
Exploiting \eqref{continuity-apost}, \eqref{Lipschitz-apost}, and \eqref{mo_ei_verh} then yields 
\begin{align}
\nonumber &\| P_{k}^{L}[F(p_{m,k}^{H})] \|_{W^{-1,p}(\Omega)} + \|  F(p_{m,k}^{H}) - P_{k}^{L}[F(p_{m,k}^{H})] \|_{W^{-1,p}(\Omega)} \\[-1.5ex]
\label{est2} \\[-1.5ex]
\nonumber &\qquad \qquad \qquad\leq C_{\text{{\tiny err}}}^{-1} \bigl(\gamma_{p} \, |p^{H\times h} - p_{m,k}^{H}|_{W^{1,p}(\Omega)} + \frac{L_{p}}{2} |p^{H\times h} - p_{m,k}^{H}|_{W^{1,p}(\Omega)}^{2}\bigr ). 
\end{align}
Thanks to $\tau_{m,p}^{k} \leq 1$ we have $1 - \sqrt{1 - \tau_{m,p}^{k}} \leq \tau_{m,p}^{k}$ and may thus estimate \cite{CaToUr09}
\begin{align}\label{est3}
\Delta_{m,p}^{k} &= \frac{\beta_{p}}{L_{p}} \bigl (1 - \sqrt{1 - \tau_{m,p}^{k}} \bigr ) \leq   \frac{2}{\beta_{p}} \left (\| P_{k}^{L}[F(p_{m,k}^{H})] \|_{W^{-1,p}(\Omega)} + \|  F(p_{m,k}^{H}) - P_{k}^{L}[F(p_{m,k}^{H})] \|_{W^{-1,p}(\Omega)} \right).
\end{align} 
Following the ideas in \cite{CaToUr09} we invoke \eqref{est2}, \eqref{est3} and Proposition \ref{error_bound} to get
\begin{align*}
\frac{1}{2} C_{\text{{\tiny err}}} \beta_{p} \Delta_{m,p}^{k} &\leq \gamma_{p} \, |p^{H\times h} - p_{m,k}^{H}|_{W^{1,p}(\Omega)} + \frac{L_{p}}{2} |p^{H\times h} - p_{m,k}^{H}|_{W^{1,p}(\Omega)}^{2} \\
&\leq \gamma_{p} \, |p^{H\times h} - p_{m,k}^{H}|_{W^{1,p}(\Omega)} + \frac{1}{2} \Delta_{m,p}^{k} (L_{p}\Delta_{m,p}^{k}).
\end{align*}
Finally, we employ \eqref{est3} again and $\tau_{m,p}^{k} \leq \frac{1}{2} C_{\text{{\tiny err}}}$ to obtain
\begin{align*}
\Delta_{m,p}^{k} \leq 4 C_{\text{{\tiny err}}}^{-1} \,\frac{\gamma_{p}}{\beta_{p}}\, |p^{H\times h} - p_{m,k}^{H}|_{W^{1,p}(\Omega)}.
\end{align*}
\end{proof}

Note that the terms $\| F(p_{m,k}^{H}) - P_{k}^{L}[F(p_{m,k}^{H})] \|_{W^{-1,p}(\Omega)}$ and $\| P_{k}^{L}[F(p_{m,k}^{H})] \|_{W^{-1,p}(\Omega)} $ are computable as $V^{H\times h}$ is a finite dimensional space. Alternatively, the dual norms may be further estimated by a localized residual type estimator (cf.~\cite{PouRap94}). To obtain the required interpolation estimate for the terms $|v^{H\times h} - v_{m}^{H}|_{W^{1,p}(\Omega)}$ we propose to replace $v^{H\times h}$ in the latter term by $v_{m'}^{H}$ with $m' > m$.\\

Note that the formulation in \eqref{nonlinear_pde} also includes nonlinear operators which have to be considered as a mapping from $W^{1,p}(\Omega)$ onto $W^{-1,p}(\Omega)$ for $p > 2$ for instance because they are not $C^{1}$-mappings with respect to the space $H^{1}(\Omega)$. Therefore we also derive an error bound for $| p^{H\times h} - p_{m,k}^{H} |_{H^{1}(\Omega)}$ for problems with $p>2$. As in \cite{CalRap1997} we assume that for all $z \in B(p_{m,k}^{H},R)$, $F^{\prime}(z): W^{1,p}(\Omega) \rightarrow W^{-1,p}(\Omega)$ can be continuously extended as an operator in $L(H^{1}(\Omega), H^{-1}(\Omega))$. In general this can be achieved by applying the Hahn-Banach theorem. Furthermore, we require that
\begin{align}
\label{inf-sup-apost-H1} 0 < \beta_{2,p} &:= \underset{|w^{H\times h}|_{H^{1}(\Omega)} \neq 0}{\underset{w^{H\times h} \in V^{H\times h}}{\inf}} \underset{|v^{H\times h}|_{H^{1}(\Omega)} \neq 0}{\underset{v^{H\times h} \in V^{H\times h}}{\sup}} \frac{\langle F^{\prime}(p_{m,k}^{H}) w^{H\times h}, v^{H\times h} \rangle}{| w^{H\times h}|_{H^{1}(\Omega)} | v^{H\times h}|_{H^{1}(\Omega)}},
\end{align}
and that there exist constants $\gamma_{2,p}$ and $L_{2,p}$ such that
\begin{align}
\label{continuity-apost-H1} \langle F^{\prime}(p_{m,k}^{H}) w^{H\times h}, v^{H\times h} \rangle &\leq \gamma_{2,p} | w^{H\times h}|_{H^{1}(\Omega)} | v^{H\times h}|_{H^{1}(\Omega)}, \\
\label{Lipschitz-apost-H1}
\| F^{\prime}(p_{m,k}^{H}) - F^{\prime}(w_{m}^{H}) \|_{H^{1}(\Omega),H^{-1}(\Omega)} &\leq L_{2,p}\, |\, p_{m,k}^{H} - w_{m}^{H}\,|_{W^{1,p}(\Omega)} \enspace \text{for} \enspace w_{m}^{H} \in B(p_{m,k}^{H},R).
\end{align}
Here, the subscript $2,p$ indicates that the argument of $F^{'}(\cdot)$ has to be in $W^{1,p}_{0}(\Omega)$, $p>2$. By transferring ideas of \cite{CalRap1997} we obtain under the assumptions of Proposition \ref{error_bound}
\begin{align*}
| p^{H\times h} - p_{m,k}^{H} |_{H^{1}(\Omega)}&\leq \frac{1}{\beta_{2,p}}\bigl ( L_{2,p} | p^{H\times h} - p_{m,k}^{H} |_{H^{1}(\Omega)} \Delta_{m,p}^{k}  
\\&\qquad \qquad +  \| F(p_{m,k}^{H}) - P_{k}^{L}[F(p_{m,k}^{H})] \|_{H^{-1}(\Omega)} + \| P_{k}^{L}[F(p_{m,k}^{H})] \|_{H^{-1}(\Omega)} \bigr).
\end{align*}
Note that this bound requires the computation or estimation of the dual norms and the appearing constants both for the $W^{1,p}$- and the $H^{1}$-norm. Thus, we employ the inverse estimate $| \, v^{H\times h} \, |_{W^{1,p}(\Omega)} \leq c_{h} | \, v^{H\times h} \, |_{H^{1}(\Omega)}$ with $c_{h}:= c(H^{2} + h^{2})^{(2-p)/(2p)}$ and a constant $c$ which is independent of $H$, $h$, $p$, and $2$ \cite{ErnGue04}. Note that the equivalence of norms on the finite dimensional space of polynomials on an element $T_{i,j}$ of the partition $T$ of $\Omega$ can be used to obtain an estimate for $c$. Note also that thanks to exponent in $c_{h}$ we expect that $c_{h}$ depends only very weakly on $h$ and $H$. Based  on that we introduce  the proximity indicator
\begin{equation}\label{proxi_2}
\tau_{m,2}^{k} := \frac{2L_{2,p}c_{h}}{\beta_{2,p}^{2}} ( \| F(p_{m,k}^{H}) - P_{k}^{L}[F(p_{m,k}^{H})] \|_{H^{-1}(\Omega)} + \| P_{k}^{L}[F(p_{m,k}^{H})] \|_{H^{-1}(\Omega)} )
\end{equation}
to derive the following  computationally more feasible $H^1$-error bound. 
\begin{proposition}[An error bound for the $H^{1}$-norm]\label{H1_bound}  Let $\tau_{m,2}^{k} < 1$ and \eqref{inf-sup-apost-H1}, \eqref{continuity-apost-H1} and \eqref{Lipschitz-apost-H1} be fulfilled. Then there exists an unique solution $p^{H\times h} \in B(p_{m,k}^{H},\frac{\beta_{2,p}}{L_{2,p}c_{h}})$ of \eqref{truth_nonlin} and the following a posteriori error estimate holds
\begin{equation}\label{delta_BBR_H1}
| p^{H\times h} - p_{m,k}^{H} |_{H^{1}(\Omega)} \leq \Delta_{m}^{k}:=\frac{\beta_{2,p}}{L_{2,p}c_{h}}(1 - \sqrt{1 - \tau_{m,2}^{k}}).
\end{equation}
If we further assume that 
\begin{equation}\label{mod_dom_ei_H1}
\| F(p_{m,k}^{H}) - P_{k}^{L}[F(p_{m,k}^{H})] \|_{H^{-1}(\Omega)} \leq c_{\text{{\tiny err}}} \| P_{k}^{L}[F(p_{m,k}^{H})] \|_{H^{-1}(\Omega)}
\end{equation}
for $c_{\text{{\tiny err}}} \in [0,1)$ and $\tau_{m,2}^{k} \leq \frac{1}{2} C_{\text{{\tiny err}}}$, where
$C_{\text{{\tiny err}}} := (1 - c_{\text{{\tiny err}}})/(1 + c_{\text{{\tiny err}}}) ,
$
we have
\begin{equation}\label{H1_estimate}
\Delta_{m}^{k} \leq 
4C_{\text{{\tiny err}}}^{-1} \frac{\gamma_{2,p} c_{h}}{\beta_{2,p}}| p^{H\times h} - p_{m,k}^{H} |_{H^{1}(\Omega)}.
\end{equation}
\end{proposition}
\begin{proof}
The proof uses the same arguments that have been applied in the proofs of Proposition \ref{error_bound} and \ref{eff}.
\end{proof}
To further estimate $\Delta_{m}^{k}$ we invoke the a posteriori error bound for the EPM in Proposition \ref{apost-epm} to replace $F(p_{m,k}^{H})$ by $P_{k'}^{L'}[F(p_{m,k}^{H})]$.  Then we define the Riesz representations $\mathcal{R}_{m}^{H\times h}$ and $\mathcal{E}_{k}^{H \times h}$ as the solutions of 
\begin{alignat}{2}\label{riesz_model}
(\mathcal{R}_{m}^{H \times h}, v^{H \times h})_{H^{1}(\Omega)} &= (P_{k}^{L}[F(p_{m,k}^{H})], v^{H \times h})_{H^{1}(\Omega)} \quad &\forall v^{H \times h} \in V^{H \times h},\\
\text{and} \qquad \qquad
(\mathcal{E}_{k}^{H \times h}, v^{H \times h})_{H^{1}(\Omega)} &= (P_{k'}^{L'}[F(p_{m,k}^{H})] - P_{k}^{L}[F(p_{m,k}^{H})], v^{H \times h})_{H^{1}(\Omega)} \quad &\forall v^{H \times h} \in V^{H \times h}.
\end{alignat}
Here, $(\cdot , \cdot )_{H^{1}(\Omega)}$ denotes the inner product associated with the $H^{1}$-semi norm. We thus obtain 
\begin{align}\label{duale_Norm1}
  |\, \mathcal{R}_{m}^{H \times h} \, |_{H^{1}(\Omega)}  &=\| P_{k}^{L}[F(p_{m,k}^{H})] \|_{H^{-1}(\Omega)}  \quad
\text{and} \quad  |\, \mathcal{E}_{k}^{H \times h} \, |_{H^{1}(\Omega)} =   \| P_{k'}^{L'}[F(p_{m,k}^{H})] - P_{k}^{L}[F(p_{m,k}^{H})] \|_{H^{-1}(\Omega)}.
\end{align}
Note that due to the definition of the snapshot set $\mathcal{M}^{\mathcal{A}}_{\Xi}$ \eqref{operator-manifold}, the a priori bound  \eqref{train_size_fixed} for the EPM is only applicable, if $\mathcal{M}^{\mathcal{A}}_{\Xi}$ is a good approximation of $\{ A(p^{H\times h}(\mu,{y})) , \, \mu \in \Xi\}$. This may be verified by comparing the convergence rates of the eigenvalues $\{ \lambda_{l}^{\text{{\tiny EPM}}} \}_{l=1}^{k_{\text{{\tiny POD}}}}$ of the POD applied to $\mathcal{M}^{\mathcal{A}}_{\Xi}$ and the coefficients $\| \int_{{\omega}} \mathcal{I}_{L}[A(p_{m,k}^{H})] \kappa_{l} \|_{L^{2}(\Omega_{1D})}^{2}$, $l=1,...,k_{\text{{\tiny POD}}}$. If the convergence rates do not coincide one may either increase the number of quadrature points in 
\eqref{disc_1dprob_nonlin} as discussed in Section \ref{adapt-RB-HMR-epm} or replace $\{ \lambda_{l}^{\text{{\tiny EPM}}} \}_{l=1}^{k_{\text{{\tiny POD}}}}$ by $\| \int_{{\omega}} \mathcal{I}_{L}[A(p_{m,k}^{H})] \kappa_{l}\|_{L^{2}(\Omega_{1D})}^{2}$, $l=1,...,k_{\text{{\tiny POD}}}$,  in the a priori bound  \eqref{train_size_fixed} for the EPM. The latter requires only the computation of $k_{\text{{\tiny POD}}}-k$ additional integrals in $y$-direction.  As the behavior of the coefficients $\int_{{\omega}} \mathcal{I}_{L}[A(p_{m,k}^{H})] \kappa_{l}$ strongly influences the convergence behavior of $P_{k}^{L}[A(p_{m,k}^{H})]$ for increasing $k$ we expect that \eqref{train_size_fixed} remains a reliable a priori bound when substituting $\{ \lambda_{l}^{\text{{\tiny EPM}}} \}_{l=1}^{k_{\text{{\tiny POD}}}}$ by $\| \int_{{\omega}} \mathcal{I}_{L}[A(p_{m,k}^{H})] \kappa_{l}\|_{L^{2}(\Omega_{1D})}^{2}$, $l=1,...,k_{\text{{\tiny POD}}}$. This is demonstrated by the numerical experiments in \S \ref{numerics_nonlin}. \\
 
 \subsubsection{Estimation of the constants}\label{subsubsect: estimate constants}\label{subsect_constants}
We close this section by addressing the computation or estimation of the constants $\beta_{p}$, $\gamma_{p}$, and $L_{p}$ for $p\geq 2$. As the constants \eqref{inf-sup-apost}-\eqref{Lipschitz-apost} are in general not computable for $p>2$ or only at unfeasible costs, we rely on estimating these constants in this case. For instance in the case of the nonlinear diffusion equation considered in the numerical examples an estimate of $\gamma_{2,p}$ and $L_{2,p}$ relies on estimates for Friedrich's inequality $\| v \|_{L^{p}(\Omega)} \leq c_{F} | v |_{W^{1,p}(\Omega)}$ and the constant in the Sobolev inequality $\| v \|_{C^{0}(\Omega)} \leq c_{E} \| v \|_{W^{1,p}(\Omega)}$. To obtain an upper bound for the constant $c_{F}$ we suggest to proceed as in \cite{Ortner09,Plu01}. A bound for the constant $c_{E,2}$ in the inequality $\| v \|_{C^{0}(\mathbb{R}^{2})} \leq c_{E,2} | v |_{W^{1,p}(\mathbb{R}^{2})}$ can be found for instance in \cite{Tal94}, Theorem 2.D. To obtain an estimate for $c_{E}$ we to multiply $v \in W^{1,p}_{0}(\Omega)$ with a cut-off function $\eta$ defined as $\eta(x,y) \equiv 1$ for $\dist( (x,y),\partial\Omega) \geq \delta$, $\eta \equiv 0$ outside $\Omega$ and with $\| \eta \|_{C^{0}(\Omega)} \leq 1$ and $\partial_{i} \eta \leq C(\delta)$, $i=1,2$ for a given constant $C(\delta)$. Then we expect that $\| v \eta \|_{C^{0}(\Omega)} \approx \| v \|_{C^{0}(\Omega)}$ as $v \in W^{1,p}_{0}(\Omega)$. Moreover, we have that 
\begin{equation}
\| \eta v \|_{C^{0}(\Omega)} \leq c_{E,2} | \eta v |_{W^{1,p}(\Omega)} \leq c_{E,2} \bigl( C(\delta)\|v \|_{L^{p}(\Omega)} + | v |_{W^{1,p}(\Omega)}\bigr),
\end{equation}
and therefore propose to employ the constant $c_{E,2}C(\delta)$ as an estimate for $c_{E}$, where $C(\delta)$ should be adapted to the considered domain. \\
To derive a lower bound for $\beta_{p}$ we suggest to proceed as in \cite{PouRap94,CalRap1997} where a finite element approximation of the nonlinear diffusion equation is considered and a lower bound of the occurring inf-sup constant is derived. However, such an estimate  is beyond the scope of this paper and therefore subject of future work. For other nonlinear operators we expect the estimates also to rely on the above inequalities. \\
As we have continuously extended $F^{\prime}(z) \in L(W^{1,p}(\Omega),W^{-1,p}(\Omega))$ to an operator in $L(H^{1}(\Omega),H^{-1}(\Omega))$ for $z \in B(p_{m,k}^{H},R)$ and $p>2$ an upper bound for $L_{2,p}$ follows directly from the estimate for $L_{p}$. If we consider $p=2$ the Lipschitz constant $L_{2}$ depends in general on a Sobolev embedding constant (see for instance \cite{CaToUr09, Ortner09,VerPat05}). A simple procedure to obtain an upper bound for this Sobolev embedding constant is described in \cite{Plu01,Ortner09}.\\
Finally, we propose a method for approximating $\beta_{2,p}$ and $\beta_{2}$. We present the approach for $\beta_{2,p}$ but it is identically applicable to $\beta_{2}$. Inspired by the idea in \cite{WirSorHaa12} to employ a matrix-DEIM approximation of the Jacobian for the computation of the Lipschitz constant of the considered nonlinear operator, we propose to use the adaptive EPM to approximate $\beta_{2,p}$. Precisely, we use the a posteriori error bound for the EPM derived in Proposition \ref{apost-epm}, to find $k'$ such that $P_{k'}^{L'}[F^{\prime}(p_{m,k}^{H})]$ approximates $F^{\prime}(p_{m,k}^{H})$ up to a given tolerance and define
\begin{align}
\label{inf-sup-apost-appro} \beta_{2,p}^{\text{\tiny app}} &:= \underset{|w^{H\times h}|_{H^{1}(\Omega)} \neq 0}{\underset{w^{H\times h} \in V^{H\times h}}{\inf}} \underset{|v^{H\times h}|_{H^{1}(\Omega)} \neq 0}{\underset{v^{H\times h} \in V^{H\times h}}{\sup}} \frac{\langle P_{k'}^{L'}[F^{\prime}(p_{m,k}^{H})] w^{H\times h}, v^{H\times h} \rangle}{| w^{H\times h}|_{H^{1}(\Omega)} | v^{H\times h}|_{H^{1}(\Omega)}}.
\end{align}
$\beta_{2,p}^{\text{\tiny app}}$ equals the smallest singular value of the Jacobian associated with $\langle P_{k'}^{L'}[F^{\prime}(p_{m,k}^{H})] w^{H\times h}, v^{H\times h} \rangle$. Thus, we determine the latter to compute $\beta_{2,p}^{\text{\tiny app}}$. Theorem \ref{apriori-epm} yields the convergence of $P_{k'}^{L'}[F^{\prime}(p_{m,k}^{H})]$ to $F^{\prime}(p_{m,k}^{H})$ as $k' \rightarrow K$, which implies $\beta_{2,p}^{\text{\tiny app}} \rightarrow \beta_{2,p}$ as $k' \rightarrow K$. Although we therefore expect $\beta_{2,p}^{\text{\tiny app}}$ to be a very good approximation of $\beta_{2,p}$, which is demonstrated by the numerical experiments in \S \ref{numerics_nonlin}, we note that it cannot be expected that $\beta_{2,p}^{\text{\tiny app}}$ provides a lower bound for $\beta_{2,p}$. 




\section{Numerical Experiments}\label{numerics_nonlin}

In this section we demonstrate the applicability of the RB-HMR approach using the adaptive EPM to nonlinear PDEs by verifying both its good approximation properties and computational efficiency. Moreover, we analyze the effectivity of the a posteriori error estimator derived in \S \ref{sect-apostest} and validate the a priori and a posteriori bounds for the adaptive EPM stated in Theorem \ref{apriori-epm} and Proposition \ref{apost-epm}. For this purpose we consider the following model problem, which is inspired by the model for immiscible two-phase flow in porous media studied in \cite{Michel2003}.
\begin{eqnarray}
\label{model problem_num} \text{Find} \enspace p \in H^{1}_{0}(\Omega): \int_{\Omega} d(p) \nabla p \cdot \nabla v \,  dx dy = \int_{\Omega} s v \, dx dy \quad \forall v\in H^{1}_{0}(\Omega),\\
\label{d}
\text{where} \enspace s \in L^{2}(\Omega), \quad \text{and} \quad d(p) := \frac{36}{c_{4}} \frac{p^{2}(1 - p)^{2}}{(p^{3} + \frac{12}{c_{4}}(1-p)^{3})^{2}} + c_{0}, \qquad \text{for constants} \enspace c_{0}, c_{4} > 0.  
\end{eqnarray}
We consider various functions $s$ in the numerical experiments and specify the function prescribed in each test case at the beginning of the respective subsection. As $c_{0}>0$ ensures uniform ellipticity and $d(p)$, $d^{\prime}(p)$ and $d^{\prime\prime}(p)$ are bounded in the relevant regions, we have that problem \eqref{model problem_num}   is well-posed \cite{CalRap1997}. Existence of a (discrete) reduced RB-HMR solution follows from Brouwer's fixed point theorem. We note that the structure of \eqref{model problem_num} necessitates to consider $F: W^{1,p}_{0}(\Omega) \rightarrow W^{-1,p}(\Omega)$ for $p>2$ which in turn requires that the HMR basis functions fulfill $\phi_{k} \in W^{1,p}(\omega)\cap H^{1}_{0}(\omega)$, $k=1,...,m$. This improved regularity for solutions of \eqref{1D_prob_quad_para_nonlin} can be proven with standard arguments (see for instance \cite{LadUra68}). For further details on well-posedness issues of problem \eqref{model problem_num} in the context of RB-HMR and the corresponding parameter dependent lower dimensional problem \eqref{1D_prob_quad_para_nonlin}  we refer to \cite{Sme13}.

In the first test case we prescribe the analytical solution of test case 1 in \cite{OS10} to compare the convergence rates of the RB-HMR approach for linear and nonlinear problems. Also in the nonlinear case the RB-HMR approach converges exponentially fast in the model order $m$, regardless whether the adaptive EPM is applied or not. However, the convergence rate is worse than for the linear problem. In the other test case we prescribe a discontinuous source term $s$ resulting in a solution with little spatial regularity both in the dominant and transverse direction. Still, we observe an exponential convergence rate of the RB-HMR approach using the adaptive EPM in the model order $m$. Both test cases have been computed employing linear FE in $x$- and $y$-direction, i.e.
$ 	
X^{H}= \left \{ v^{H} \in C^{0}(\Omega_{1D}) \,:\, v^{H}|_{\mathcal{T}_{i}} \in \mathbb{P}^{1}_{1}(\mathcal{T}_{i}), \mathcal{T}_{i} \in \mathcal{T}_{H}\right \}, 
Y^{h} = \left \{ v^{h} \in C^{0}({\omega}) \,:\, v^{h}|_{\tau_{j}} \in \mathbb{P}^{1}_{1}(\tau_{j}), \tau_{j} \in \tau_{h}\right \},
$
and
$
V^{H\times h} =  \{ v^{H \times h} \in C^{0}({\Omega}) \, : \, v^{H \times h}|_{T_{i,j}} \in \mathbb{Q}_{1,1}, T_{i,j} \in {T} \},
$
using equidistant grids in $x$- and $y$-direction. We have used the following quadrature weights in \eqref{quad_formula_nonlin}
\begin{equation}\label{def_weights_nonlin}
\alpha_{1} := \frac{x_{1}^{q}+x_{2}^{q}}{2}-x_{0}, \quad \alpha_{l} := \frac{x_{l+1}^{q} - x_{l-1}^{q}}{2}, \enspace l = 2,...,Q-1, \quad \alpha_{Q}:= x_{1}- \frac{x_{Q-2}^{q}+x_{Q-1}^{q}}{2}, 
\end{equation}
where the quadrature points $x_{l}^{q}$, $l = 1,...,Q$ are expected to be sorted in ascending order. We have only applied a simplified version of Algorithm \ref{adapt-RB-HMR} \textsc{Adaptive-RB-HMR} in the numerical experiments, as we have chosen the number of quadrature points employed in the parameter dependent 1D problem \eqref{disc_1dprob_nonlin} a priori. However, a comparison of the performance of the RB-HMR approach using $1$ or $2$ quadrature points in \eqref{disc_1dprob_nonlin} is provided for the second test case. Furthermore, we have applied the adaptive EPM \ref{adapt-EPM} based on the EIM with $N_{\text{{\tiny max}}}^{\text{{\tiny int}}}=0$ and we thus obtain $k = L$ in \eqref{epm-approx}. To simplify notations we omit the $\sim$ as introduced in Section \ref{sect:EPM_EIM}. We have employed the estimate
\begin{equation}
\int_{\Omega} \nabla \bigl\lbrace\left(d(p_{m,k}^{H}) - d(z)\right) w\bigr\rbrace\nabla v \leq c_{E}(1+c_{F}^{p})^{1/p}\bigl(2c_{2} + c_{3}c_{E}(1+c_{F}^{p})^{1/p} c_{h}|p_{m,k}^{H}|_{H^{1}(\Omega)} \bigr),
\end{equation}
to obtain an estimate for the Lipschitz constant $L_{2,p}$, where $\|d^{\prime}(z)v\| \leq c_{2} \|v\|$, $\|d^{\prime\prime}(z)v\| \leq c_{3} \|v\|$, for $v \in \mathbb{R}^{2}$, $z \in B(p_{m,k}^{H},R)$, and $c_{F}$ and $c_{E}$ have been introduced above. 
Since $d$ is only locally bounded for some choices of $c_{4}$, we computed local approximations of $c_{2}$ and $c_{3}$ by evaluating $d^{\prime}$ and $d^{\prime \prime}$ in the discrete reduced solution $p_{m,k}^{H}$ of \eqref{red_prob_hmr_epm}. Moreover, we have estimated the constants $c$ in $c_{h}$, $(1+c_{F}^{p})^{1/p}$, and $c_{E}$ by $1$, which seems to be a reasonable estimate as for instance the procedure proposed in \cite{Ortner09,Plu01} yields a bound of $1.0856$ for $(1+c_{F}^{p})^{1/p}$ and the value of the sharp bound $c_{E,2}$ stated in \cite{Tal94} is about $1.54$.

Setting $e_{m}^{k}:=  p^{H \times h} - p^{H}_{m,k}$, where $p^{H\times h}$ solves \eqref{truth_nonlin} and $p_{m,k}^{H}$ \eqref{red_prob_hmr_epm}, we define the relative model error in the $H^{1}$-semi norm or $L^{2}$-norm as 
$
|e_{m}^{k}|_{H^{1}}^{rel} := |e_{m}^{k}|_{H^{1}}/|p^{H \times h}|_{H^{1}} \enspace \text{and} \enspace\|e_{m}^{k}\|_{L^{2}(\Omega)}^{rel} := \|e_{m}^{k}\|_{L^{2}(\Omega)}/\|p^{H \times h}\|_{L^{2}(\Omega)}.
$
The relative total error $|e|_{H^{1}}^{rel}$ is either defined as $|e|_{H^{1}}^{rel} := |p - p^{H}_{m,k}|_{H^{1}}/|p|_{H^{1}}$ if the full solution $p$ of \eqref{model problem_num} is available as in test case 1 or as $|e|_{H^{1}}^{rel} := |p_{fine} - p^{H}_{m,k}|_{H^{1}}/|p_{fine}|_{H^{1}}$, where 
$p_{fine}$ denotes a very finely resolved bilinear FE solution. We denote the POD-error associated with the HMR by $e_{m}^{\text{{\tiny POD}}}$ and the POD-error corresponding to the adaptive EPM by $e^{k}_{\text{{\tiny POD}}}$. 
Moreover, we set $\bar{e}_{m}^{L^{2}} := (\sum_{j=m+1}^{M} \| \bar{p}_{j,k}^{H}\|^{2}_{L^{2}(\Omega_{1D})})^{1/2}$ and $e^{k}_{L^{2}} := (\sum_{j=k+1}^{K} \| \int_{{\omega}} \mathcal{I}_{L}[A(p_{m,k}^{H})] \kappa_{k} \|^{2}_{L^{2}(\Omega_{1D})})^{1/2}$, where $M = \dim(\mathcal{M}^{\mathcal{P}}_{\Xi})$, $K = \dim(\mathcal{M}^{\mathcal{A}}_{\Xi})$, and $\mathcal{I}_{L}[\cdot]$ has been defined in \eqref{epm-approx}. 
For the validation of the effectivity of the error bounds, we finally shorten the notation by setting $\|e_{\mbox{{\tiny mod}}}\| :=\| P_{k}^{L}[F(p_{m,k}^{H})] \|$, $\|e_{\mbox{{\tiny EPM}}}\| := \| P_{k'}^{L'}[F(p_{m,k}^{H})] - P_{k}^{L}[F(p_{m,k}^{H})] \|$ and $\|e_{\mbox{{\tiny EPM}}}^{ex}\| := \| F(p_{m,k}^{H}) - P_{k}^{L}[F(p_{m,k}^{H})] \|$ either for the $H^{-1}$- or the $L^{2}$-norm. The implementation of Algorithm \ref{adapt-RB-HMR} \textsc{Adaptive-RB-HMR} has been realized in MATLAB.  All computations have been performed on a computer with an Intel Core i7 (8 cores) with 2.93 GHz.

\begin{figure}[t]
\centering
\subfloat[{\scriptsize $|e_{m}|_{H^{1}}^{rel}$}]{
\includegraphics[scale = 0.3]{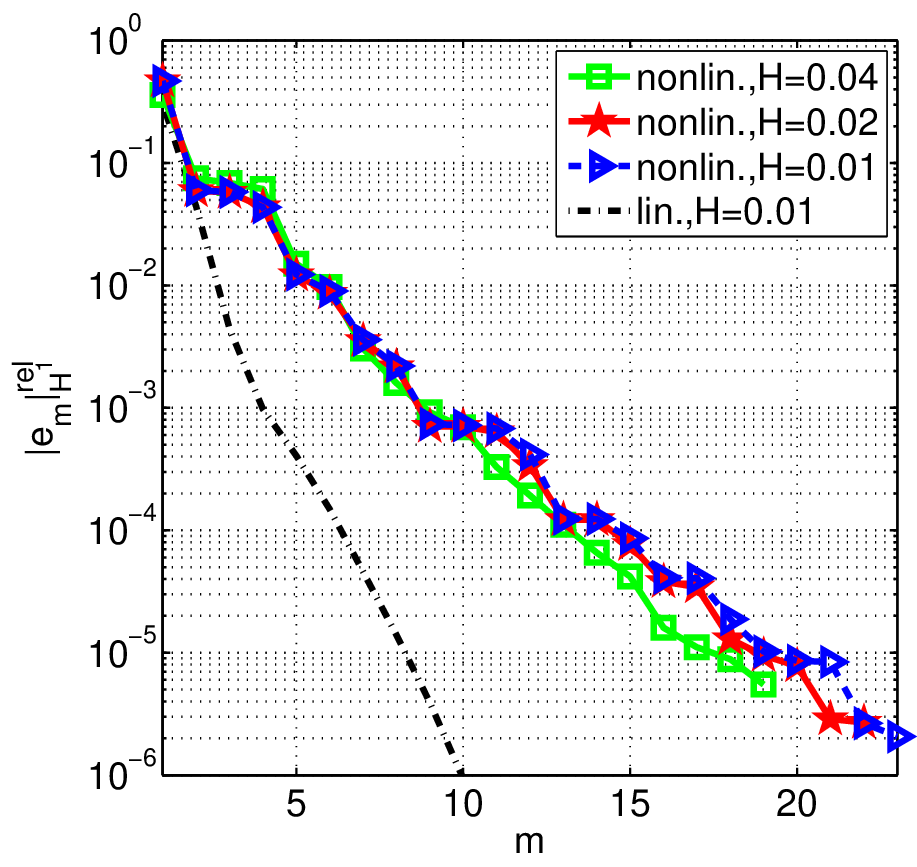} \label{fig4.2a}}
\subfloat[{\scriptsize $|e|_{H^{1}}^{rel}$}]{
\includegraphics[scale = 0.3]{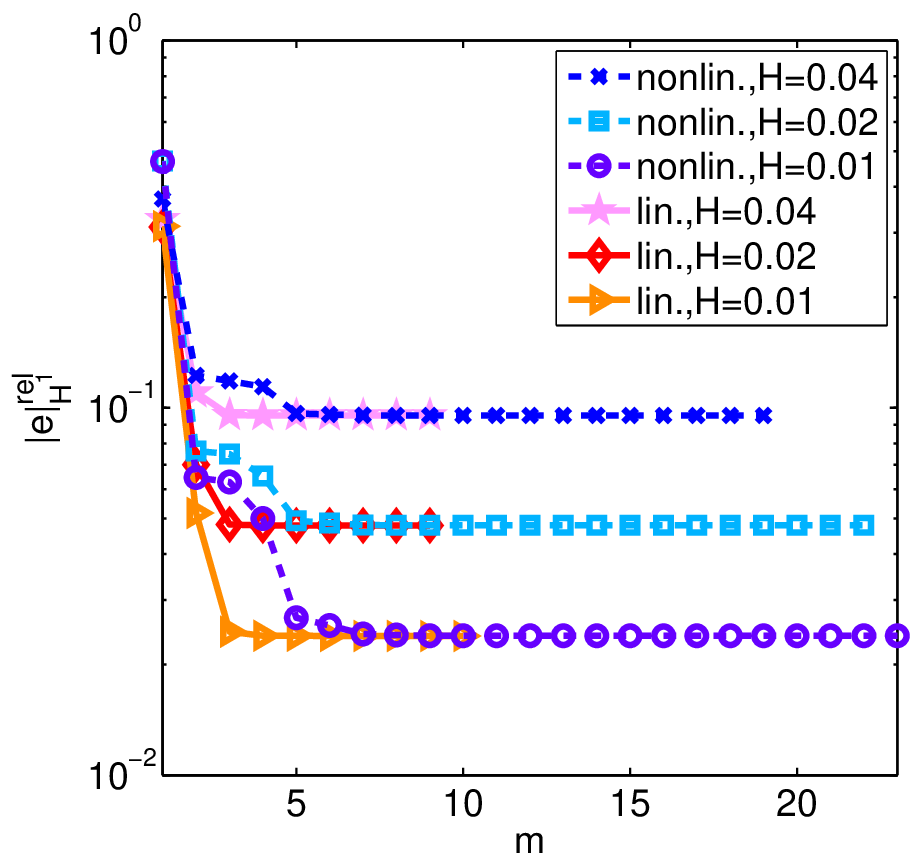}\label{fig4.2b}}
\subfloat[{\scriptsize $|e_{m}|_{H^{1}}^{rel}$ and $|e_{m}^{k}|_{H^{1}}^{rel}$}]{
\includegraphics[scale = 0.3]{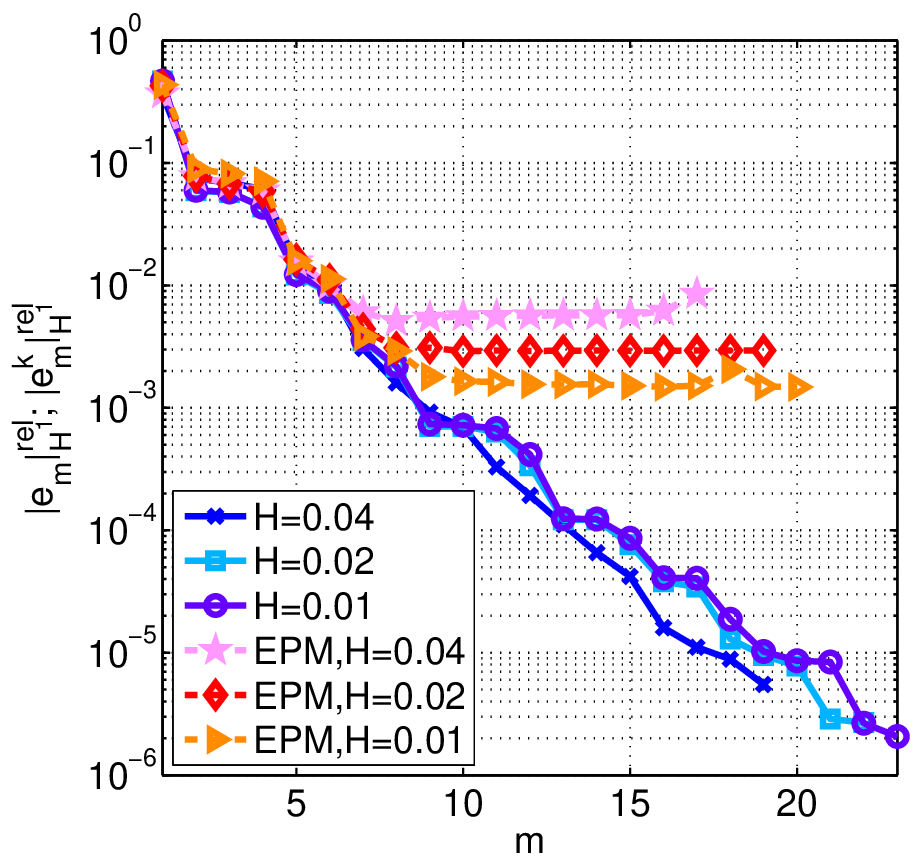}\label{fig4.2c}}
\caption{{\footnotesize Test case 1: Comparison of the behavior of the relative model error $|e_{m}|_{H^{1}(\Omega)}^{rel}= |p^{H\times h} - p_{m}^{H}|_{H^{1}(\Omega)}/|p^{H\times h}|_{H^{1}(\Omega)}$ (A) and the relative total error for $|e|_{H^{1}(\Omega)}^{rel}= |p - p_{m}^{H}|_{H^{1}(\Omega)}/|p|_{H^{1}(\Omega)}$ (B) for the linear problem (test case 1 in \cite{OS10}) and the nonlinear problem \eqref{model problem_num}; Comparison of the behavior of the relative model error when applying the adaptive EPM ($|e_{m}^{k}|_{H^{1}(\Omega)}^{rel}$ ) or not ($|e_{m}|_{H^{1}(\Omega)}^{rel}$) (C).}\label{fig4.2}}
\end{figure}

\begin{figure}[t]
\centering
\subfloat[{\scriptsize $|e_{m}^{k}|_{H^{1}}^{rel}$}]{
\includegraphics[scale = 0.35]{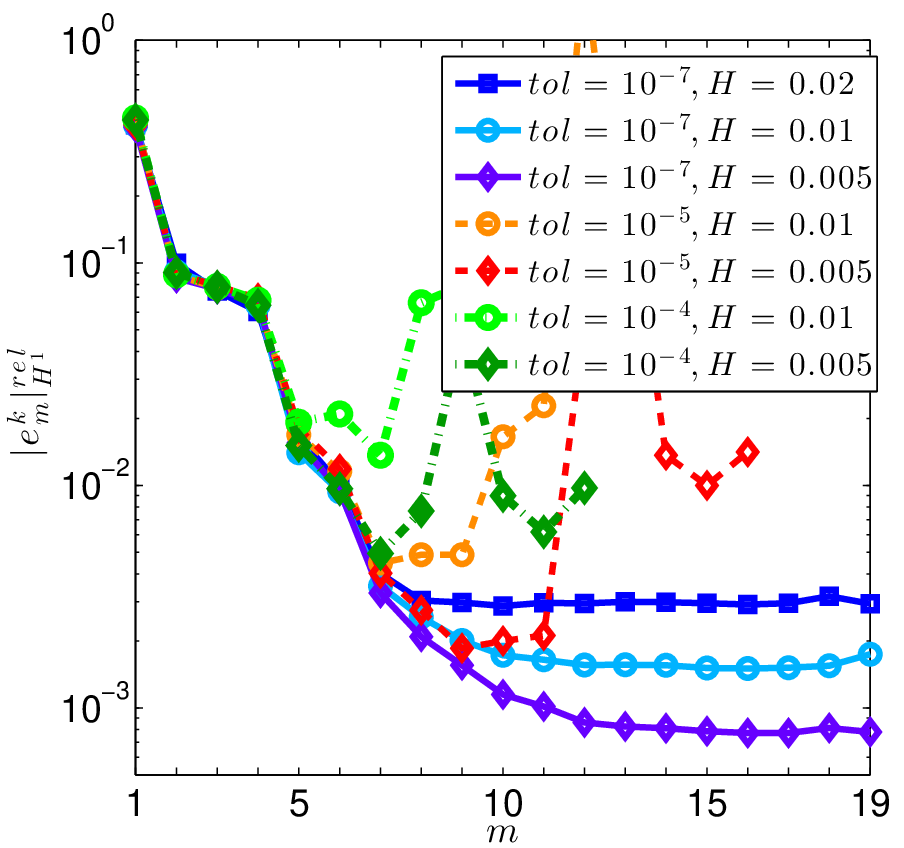} \label{fig4.3a}}
\subfloat[{\scriptsize $|e|_{H^{1}}^{rel}$}]{
\includegraphics[scale = 0.35]{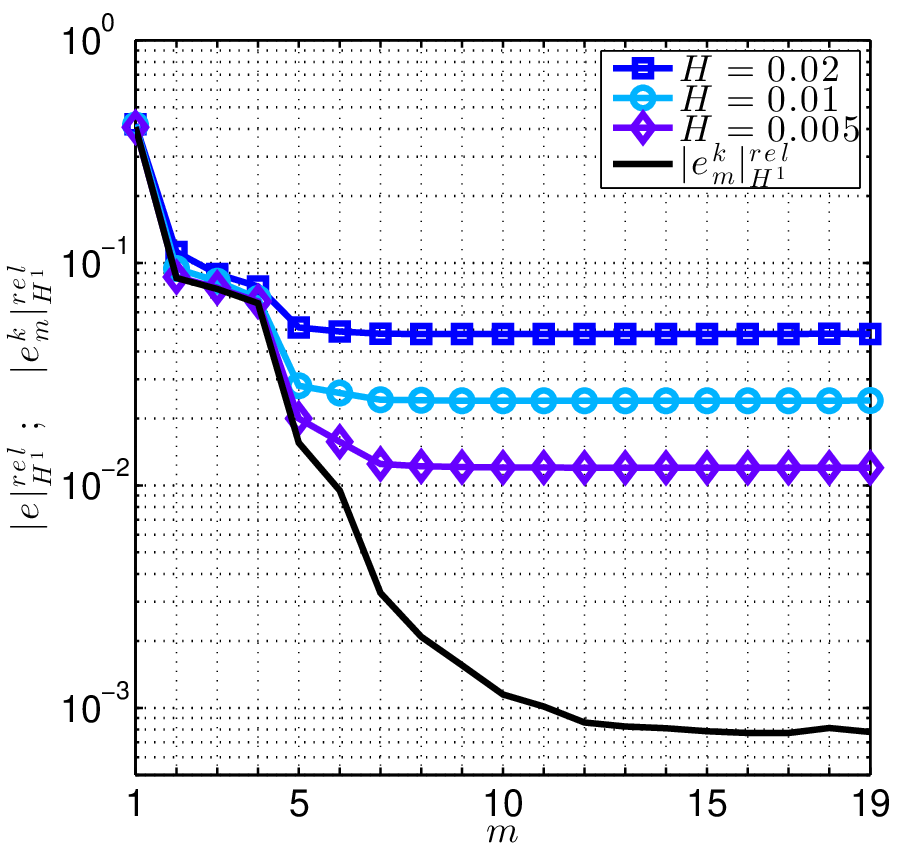}\label{fig4.3b}}
\subfloat[{\scriptsize $\|e_{m}^{k}\|_{L^{2}(\Omega)}^{rel}, e_{m}^{\mbox{{\tiny POD}}},\bar{e}_{m}^{L^2}$}]{
\includegraphics[scale = 0.35]{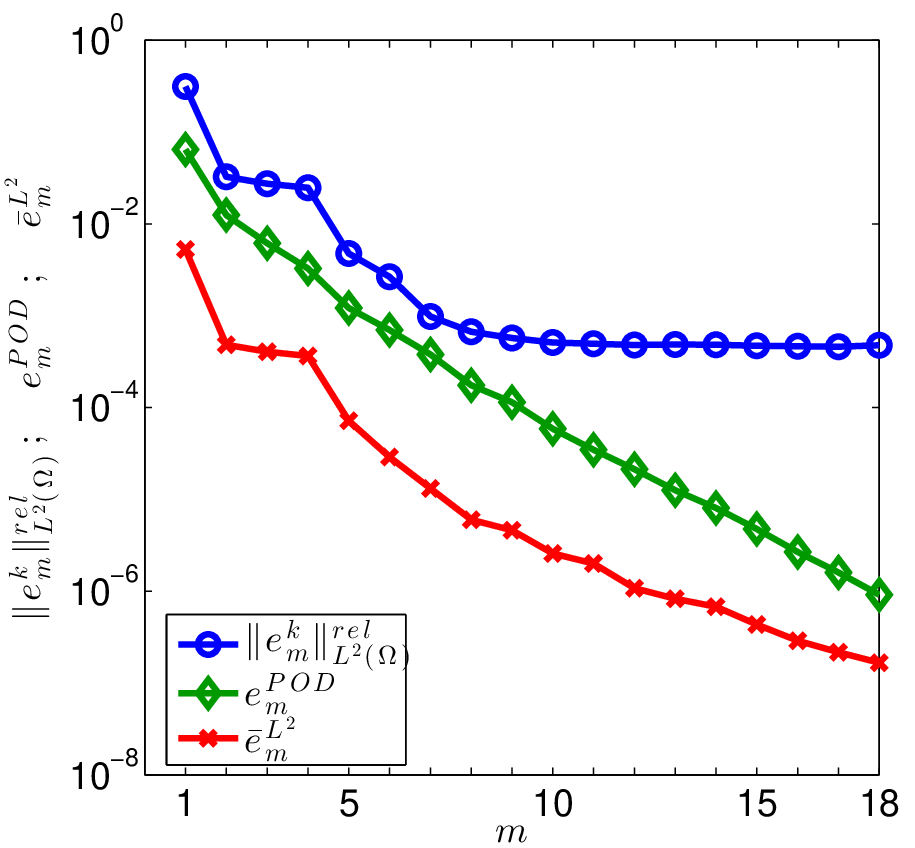}\label{fig4.3c}}\\
\subfloat[{\scriptsize $\lambda_{m}$, $\|\bar{p}_{m,k}^{H}\|_{L^{2}(\Omega_{1D})}^{2}$}]{
\includegraphics[scale = 0.35]{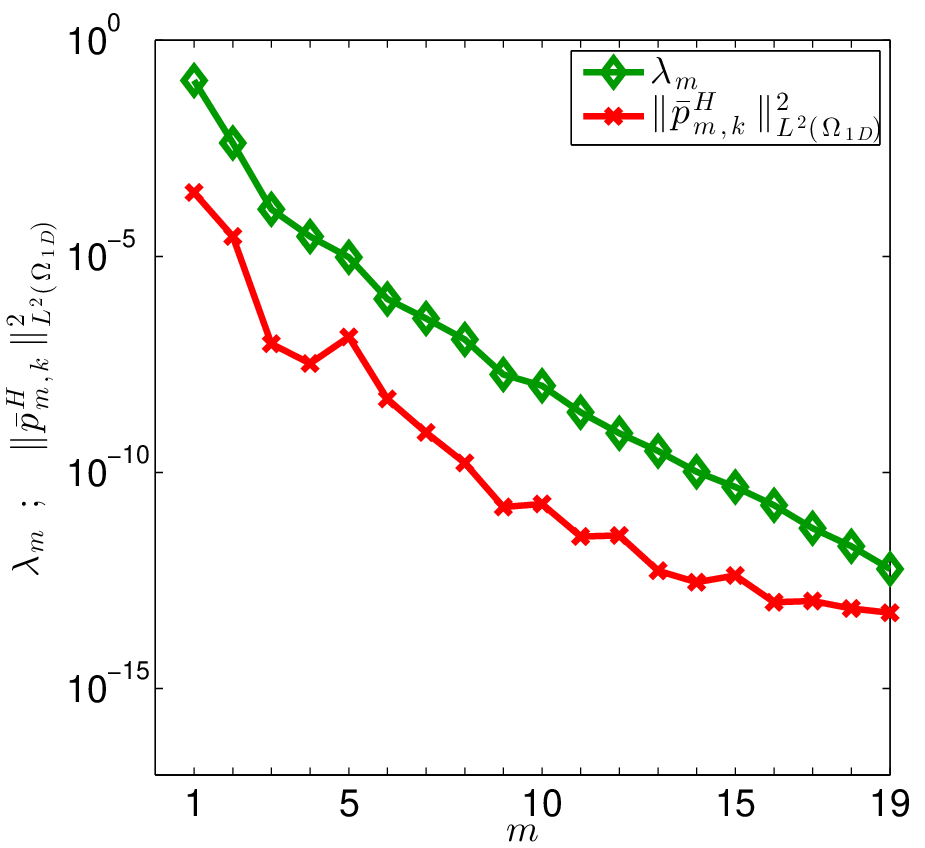}\label{fig4.3d}}
\subfloat[{\scriptsize $|e_{m}^{k}|_{H^{1}}^{rel}$-landscape}]{
\includegraphics[scale = 0.35]{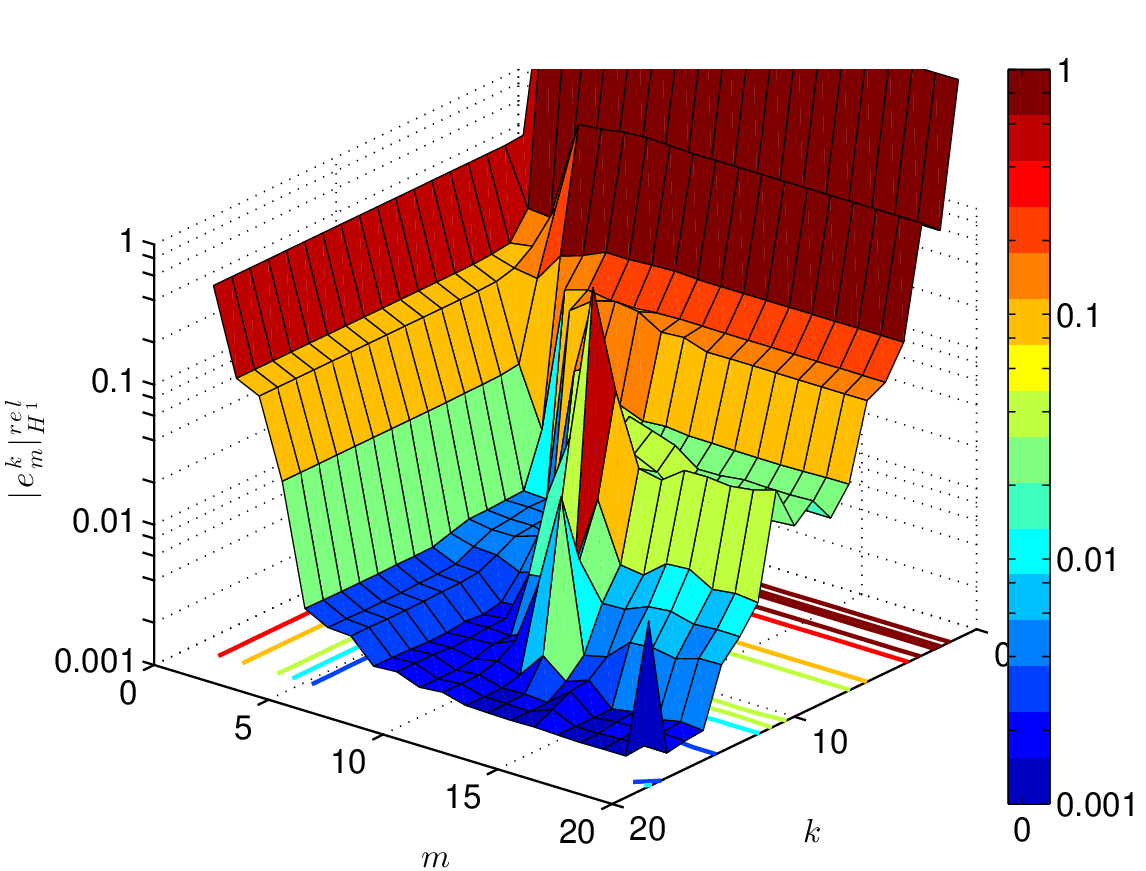}\label{fig4.3e}}
\caption{{\footnotesize Test case 1: Comparison of the convergence behavior of $|e_{m}^{k}|_{H^{1}}^{rel}$ for EPM-tolerances $\varepsilon_{\mbox{{\tiny tol}}}^{\text{{\tiny EPM}}} = 10^{-4}, 10^{-5}, 10^{-7}$ (A), and for $\varepsilon_{\mbox{{\tiny tol}}}^{\text{{\tiny EPM}}} = 10^{-7}$: $|e|_{H^{1}}^{rel}$ for different mesh sizes and $|e_{m}^{k}|_{H^{1}}^{rel}$ for $H=0.005$ (B), $\|e_{m}^{k}\|_{L^{2}(\Omega)}^{rel}, e_{m}^{\mbox{{\tiny POD}}}$ and $\bar{e}_{m}^{L^2}$ (C) and $\lambda_{m}$ and $\|\bar{p}_{m,k}^{H}\|_{L^{2}(\Omega_{1D})}^{2}$ (D) for $H=0.005$; Convergence behavior of $|e_{m}^{k}|_{H^{1}}^{rel}$ for increasing model order $m$ and collateral basis size $k$ for $H = 0.01$ (E); all plots $N_{H'}=10$.}\label{fig4.3}}
\vspace{-10pt}
\end{figure}

\subsection*{Test case 1}
First, we investigate the convergence behavior of the RB-HMR approach for an analytical solution $p(x,y) = y^{2}(1-y)^{2}(0.75-y)x(2-x)\exp(\sin(2\pi x))$, which has already been 
considered in test case 1 in \cite{OS10} and originally in \cite{ErnPerVen2008,PerErnVen10}  solving the Poisson problem. We choose $\Omega = (0,2) \times (0,1)$ and $c_{0}=0.1$ and $c_{4}=36$ in \eqref{d}. We compare  the convergence behavior of the relative model error $|e_{m}|_{H^{1}(\Omega)}^{rel}= |p^{H\times h} - p_{m}^{H}|_{H^{1}(\Omega)}/|p^{H\times h}|_{H^{1}(\Omega)}$  for the nonlinear case, where $p_{m}^{H}$ is the solution of the discrete reduced problem  (discretization of \eqref{reprob_nonlin}) with
the linear case where $p_{m}^{H}$ is an RB-HMR approximation of the solution of a Poisson problem. Note that both the nonlinear problem and the Poisson problem have the same analytical solution $p(x,y)$, where we refer to \cite{OS10} for details on the linear problem and the respective RB-HMR approximation.

We observe an exponential convergence rate of $|e_{m}|_{H^{1}(\Omega)}^{rel}$ for the nonlinear problem \eqref{model problem_num}, which is worse than the one for the Poisson problem (Fig.~\ref{fig4.2a}). Nevertheless, also for the nonlinear case still $9$ basis functions are sufficient to achieve $|e_{m}|_{H^{1}(\Omega)}^{rel} \leq 10^{-3}$ (Fig.~\ref{fig4.2a}). Taking also into account the discretization error and hence considering the relative total error $| e |_{H^{1}}^{rel} = |p - p_{m}^{H}|_{H^{1}}/|p|_{H^{1}(\Omega)}$ we observe that at least for the considered mesh sizes the effects of the detoriation of the convergence rate of the RB-HMR due to the nonlinearity on the behavior of the total error are rather small (Fig.~\ref{fig4.2b}). That is because the discretization error is dominating over the model error in this example already for an RB-HMR approximation using only a small number of basis functions (Fig.~\ref{fig4.2b}). Applying the adaptive EPM preserves the convergence rate of the model error $|e_{m}|_{H^{1}(\Omega)}^{rel}$ (Fig.~\ref{fig4.2c}) until a so-called EPM-plateau (see \cite{Tonn2011,DroHaaOhl2012} for the EIM-plateau) is reached. The model error enters an EPM-plateau if the approximation properties of the collateral basis space $W_{k}$ prevent a further reduction of the model error, i.e. $k$ is chosen too small compared to $m$, and the nonlinear operator is hence not approximated accurate enough. Our experiments showed that the tolerance of the POD for the adaptive EPM $\varepsilon_{\mbox{{\tiny tol}}}^{\text{{\tiny EPM}}}$ should be set to
$\varepsilon_{\mbox{{\tiny tol}}}^{\text{{\tiny EPM}}} = c_{tol} \varepsilon_{\mbox{{\tiny tol}}}^{\text{{\tiny HMR}}}$ with $c_{tol} \in [10^{-4},10^{-3}]$, to ensure that $k$ is chosen large enough. However, even if $W_{k}$ is spanned by all linear independent functions $\mathcal{A}^{h}(\mu) \in \mathcal{M}^{\mathcal{A}}_{\Xi}$ \eqref{disc_mani}, a small error and thus a EPM plateau cannot be avoided due to the necessary projection of the snapshots onto a discrete space and other numerical constraints. Note that the level of the EPM-plateau becomes smaller for decreasing $H$ and lies for all considered mesh sizes well below the total error $|e|_{H^{1}}^{rel}$ (Fig.~\ref{fig4.2b},\ref{fig4.2c}).  Finally, we remark that in all computations for the plots in Fig.~\ref{fig4.2} we used the exact error in the application of the Algorithm \ref{adapt-RB-HMR} \textsc{Adaptive-RB-HMR} (Fig.~\ref{fig4.2c}) to assess only the influence of the nonlinearity in Fig.~\ref{fig4.2a} and Fig.~\ref{fig4.2b} or the application of the adaptive EPM in Fig.~\ref{fig4.2c}. 

Comparing the convergence behavior of $|e_{m}^{k}|_{H^{1}(\Omega)}^{rel}$ for different POD-tolerances for the adaptive EPM in Fig.~\ref{fig4.3a}, we observe that for $\varepsilon_{\mbox{{\tiny tol}}}^{\text{{\tiny EPM}}} = 10^{-4}, 10^{-5}$ the error can even increase when entering the EPM-plateau. For $\varepsilon_{\mbox{{\tiny tol}}}^{\text{{\tiny EPM}}} = 10^{-7}$ the error stagnates in the EPM-plateau and the level of the plateau decreases uniformly for dropping mesh sizes  (Fig.~\ref{fig4.3a}). Fig.~\ref{fig4.3e} illustrates the error convergence of $|e_{m}^{k}|_{H^{1}(\Omega)}^{rel}$ for a simultaneous increase of the model order $m$ and collateral basis size $k$ for $H = 0.01$. Again we see that for small $k$ the scheme might even get unstable if  $m$ exceeds a certain limit, which is however not the case for higher values of $k$. Moreover, we observe that if the approximation of the nonlinear operator is good enough, a further increase of $k$ does not reduce $|e_{m}^{k}|_{H^{1}(\Omega)}^{rel}$ if $m$ is kept fixed. Choosing $\varepsilon_{\mbox{{\tiny tol}}}^{\text{{\tiny EPM}}} = 10^{-7}$ and thereby ensuring that the approximation properties of $W_{k}$ are sufficient, we finally see that for the considered mesh sizes the EPM-plateau has no effect on the relative total error $|e|_{H^{1}(\Omega)}^{rel}$ (Fig.~\ref{fig4.3b}). If we compare $\|e_{m}\|_{L^{2}(\Omega)}^{rel}, e_{m}^{\mbox{{\tiny POD}}}$ and $\bar{e}_{m}^{L^2}$ for $H = 0.005$ in Fig.~\ref{fig4.3c}, we detect that all three quantities exhibit the same exponential convergence rate until $\|e_{m}\|_{L^{2}(\Omega)}^{rel}$ reaches the EPM-plateau. As also the convergence behavior of the eigenvalues of the POD $\lambda_{m}$ and of the coefficients $\|\bar{p}_{m,k}^{H}\|_{L^{2}(\Omega_{1D})}^{2}$ 
coincide (Fig.~\ref{fig4.3d}), we conclude that for the present test case the convergence behavior of the POD transfers to the coefficients $\|\bar{p}_{m,k}^{H}\|_{L^{2}(\Omega_{1D})}^{2}$, $\bar{e}_{m}^{L^2}$ and to the model error $\|e_{m}^{k}\|_{L^{2}(\Omega)}^{rel}$,$|e_{m}^{k}|_{H^{1}(\Omega)}^{rel}$. Thus, we infer that the discrete solution manifold $\mathcal{M}^{\mathcal{P}}_{\Xi}$ \eqref{manifold-nonlin} and the reference solution $p^{H\times h}$ are approximated with the same approximation accuracy by the reduction space $Y_{m}$. Note that thanks to the coincidence of the convergence rates of $\lambda_{m}$ and $\|\bar{p}_{m,k}^{H}\|_{L^{2}(\Omega_{1D})}^{2}$ (Fig.~\ref{fig4.3d}), the \textsc{QP-Indicator} introduced in \S \ref{adapt-RB-HMR-epm} would not have increased the number of quadrature points used in \eqref{1D_prob_quad_para_nonlin}. \\

\begin{figure}[t]
\centering
\includegraphics[scale = 0.35]{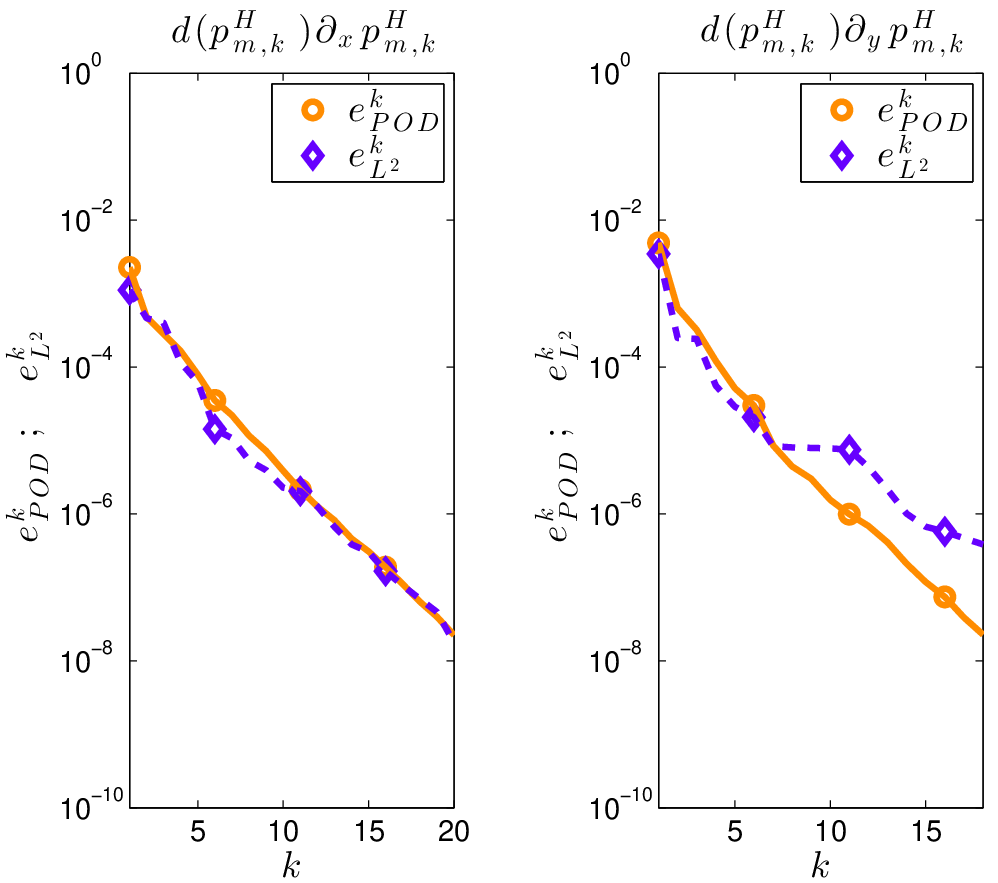}\quad \includegraphics[scale = 0.35]{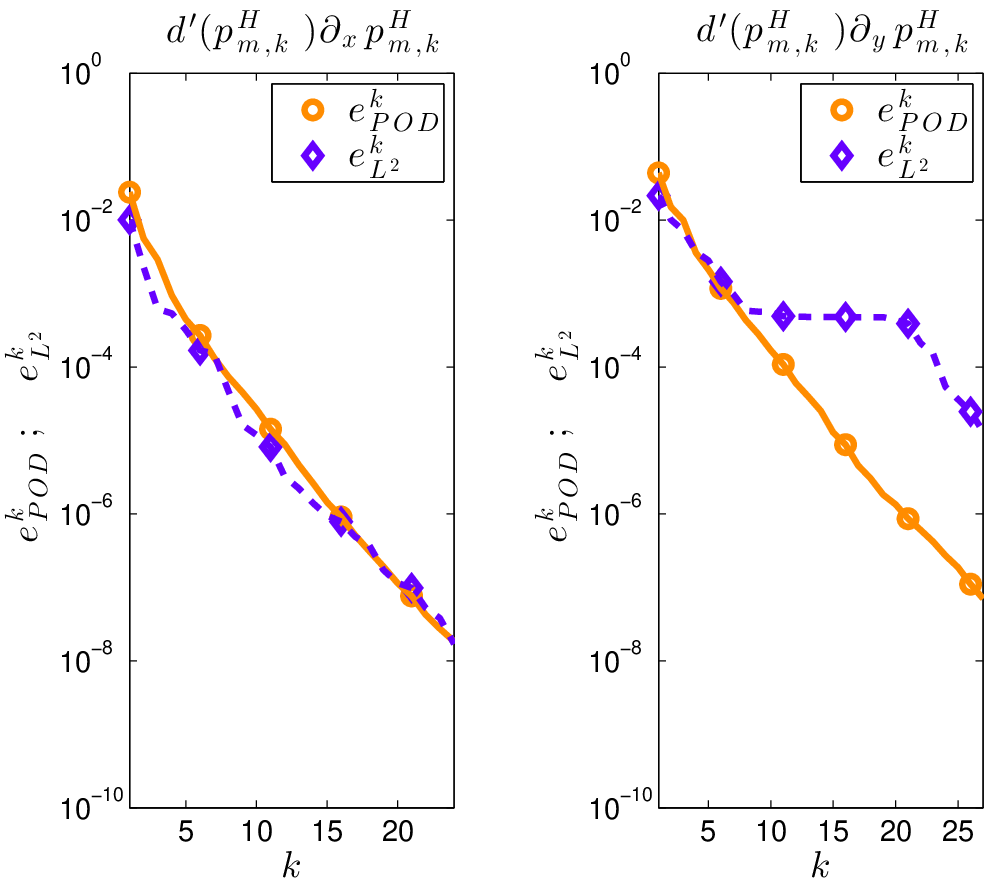} \quad \includegraphics[scale = 0.35]{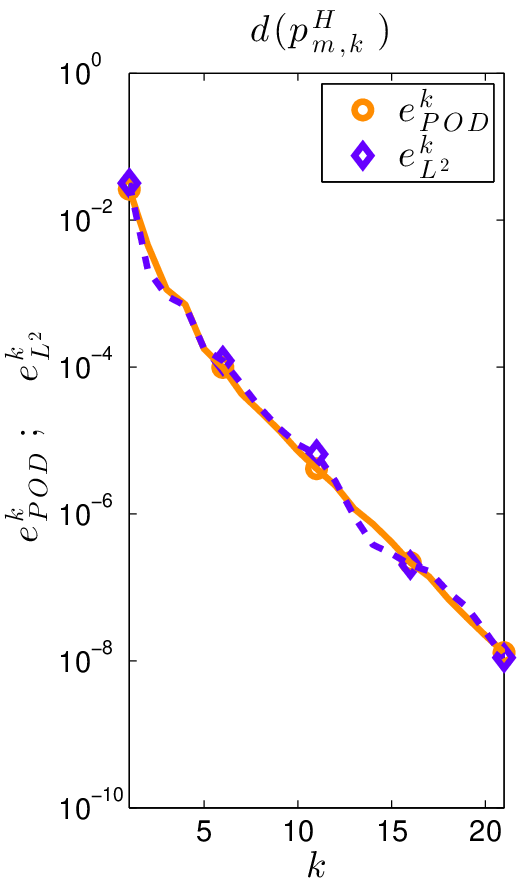}
\caption{{\footnotesize Test case 1: Comparison of $e_{\mbox{{\tiny POD}}}^{k}$ and $e^{k}_{L^{2}}$ for $H=0.005$, $N_{H'}=10$, and $\varepsilon_{\mbox{{\tiny tol}}}^{\text{{\tiny EPM}}} = 10^{-7}$.}\label{fig4.4}}
\end{figure}
To assess the approximation quality of the collateral basis spaces, we finally compare in Fig.~\ref{fig4.4} the convergence rates of the respective POD-error $e^{k}_{\mbox{{\tiny POD}}}$ and $e^{k}_{L^{2}}$. We observe that the rates for the approximation of $d(p_{m,k}^{H})\partial_{x} p_{m,k}^{H}$, $d^{\prime}(p_{m,k}^{H})\partial_{x} p_{m,k}^{H}$ and $d(p_{m,k}^{H})$ coincide perfectly. The deviation for the other two might be explained by the fact, that we have projected the snapshots corresponding to $d(p_{m,k}^{H})\partial_{y} p_{m,k}^{H}$ and $d^{\prime}(p_{m,k}^{H})\partial_{y} p_{m,k}^{H}$
onto the space of piecewise constant functions to account for the structure of the nonlinear operator. This yields a worse convergence behavior for decreasing $h$ as the projection onto the space $Y^{h}$ we have employed for the others.  As apart from this deviation the convergence rates coincide, we nevertheless conclude that the nonlinear operator $A(p^{H\times h})$, its Fr\'{e}chet derivative $A^{\prime}(p^{H\times h})$, and the discrete manifolds of operator evaluations are approximated with the same quality. \\

\begin{table}[t]
\center
\begin{tabular}{|c|c|c|c|c|}
\hline $m$ &$\varepsilon_{\mbox{{\tiny tol}}}^{\text{{\tiny EPM}}} = 10^{-5}$ (app) &$\varepsilon_{\mbox{{\tiny tol}}}^{\text{{\tiny EPM}}} = 10^{-5}$ (ex) & $\varepsilon_{\mbox{{\tiny tol}}}^{\text{{\tiny EPM}}} = 10^{-7}$ (app) & $\varepsilon_{\mbox{{\tiny tol}}}^{\text{{\tiny EPM}}} = 10^{-7}$ (ex)\\
\hline 1 &  0.0984390 & 0.0989506 & 0.0984566 & 0.0989346 \\
\hline 2 &  0.0982283 & 0.0982740 & 0.0981483 & 0.0982654 \\
\hline 3 & 0.0981121 & 0.0982542 & 0.0981655 & 0.0982556 \\
\hline 4 & 0.0981349 & 0.0982402 & 0.0982860 & 0.0982506 \\
\hline 5 & 0.0980797 & 0.0981840 & 0.0982739 & 0.0981820 \\
\hline 10 & 0.0980739 & 0.0981621 & 0.0980968 & 0.0981864\\
\hline 12 & 0.0982229 & 0.0981288 & 0.0982790  & 0.0981867 \\
\hline 15&  0.0981518 & 0.0981433 & 0.0980749 & 0.0981867\\
\hline 
\end{tabular}
\caption{{\footnotesize Test case 1: Comparison of the exact inf-sup stability factor (ex) with its approximate value (app) for $H = 0.02$ and different tolerances $\varepsilon_{\mbox{{\tiny tol}}}^{\text{{\tiny EPM}}}$ in the adaptive EPM and increasing model order $m$.}\label{compare-inf-sup}}
\end{table}

\begin{figure}[t]
\centering
\subfloat[{\scriptsize $\Delta_{m}^{k}, |e^{k}_{m}|_{H^{1}}$}]{
\includegraphics[scale = 0.35]{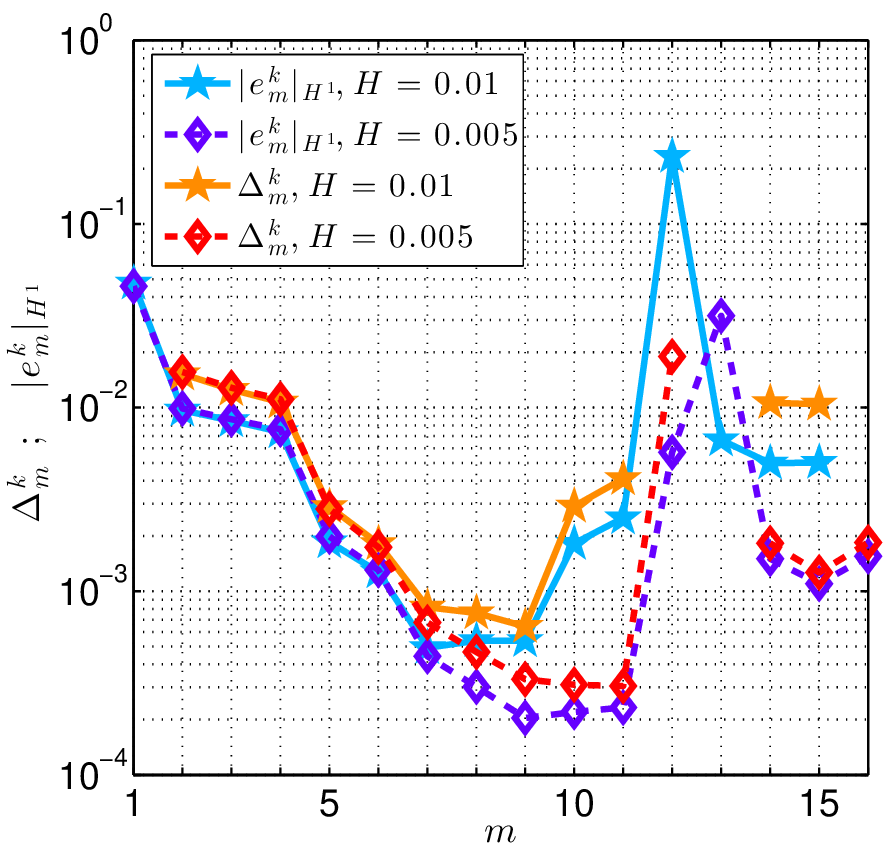}\label{fig4.5a}}
\subfloat[{\scriptsize $\Delta_{m}^{k},|e_{m}^{k}|_{H^{1}}$}]{
\includegraphics[scale = 0.35]{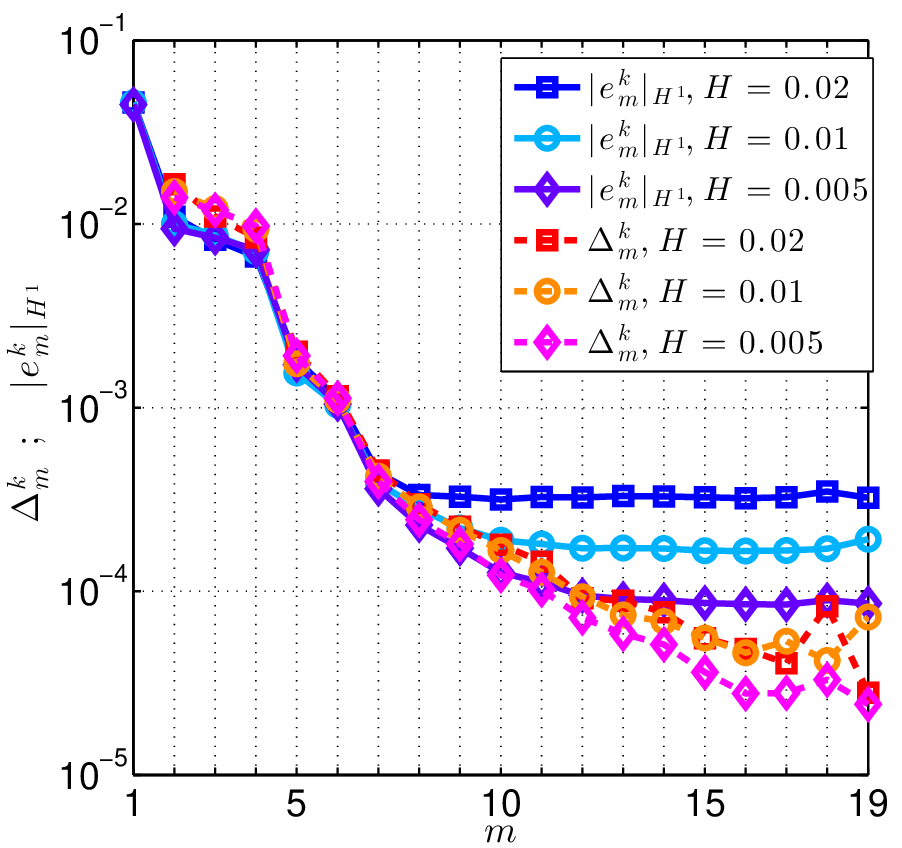} \label{fig4.5c}}
\subfloat[{\scriptsize $e_{EPM}, e_{EPM}^{ex}$ }]{
\includegraphics[scale = 0.35]{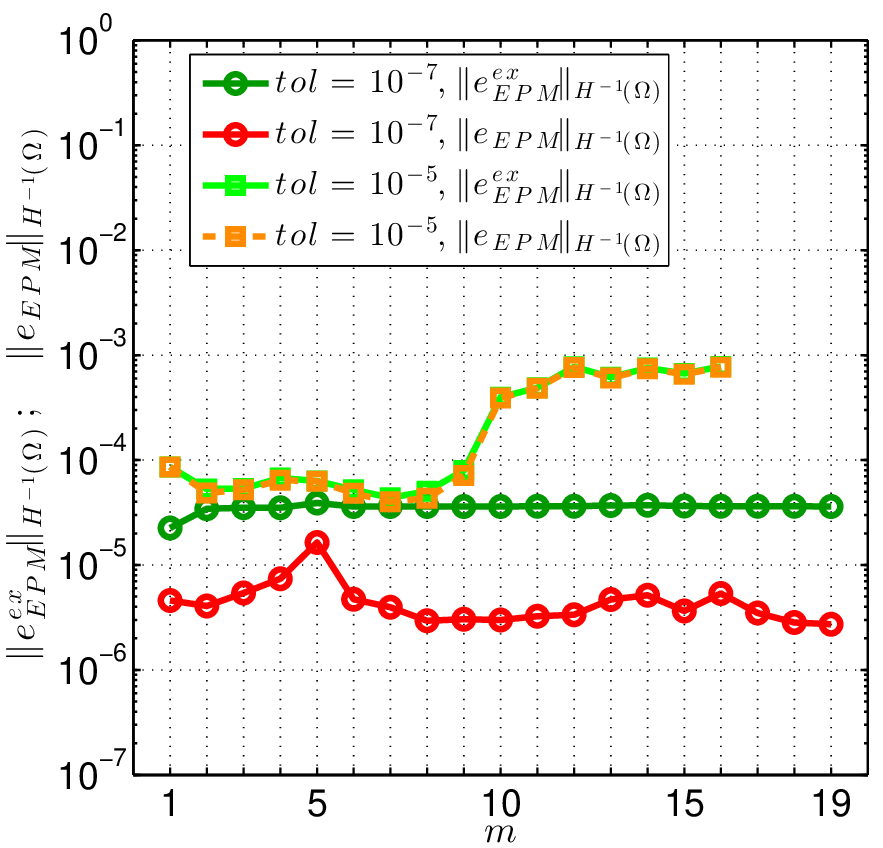}\label{fig4.5d}} \\
\subfloat[{\scriptsize $e_{mod}, e_{EPM}$}]{
\includegraphics[scale = 0.35]{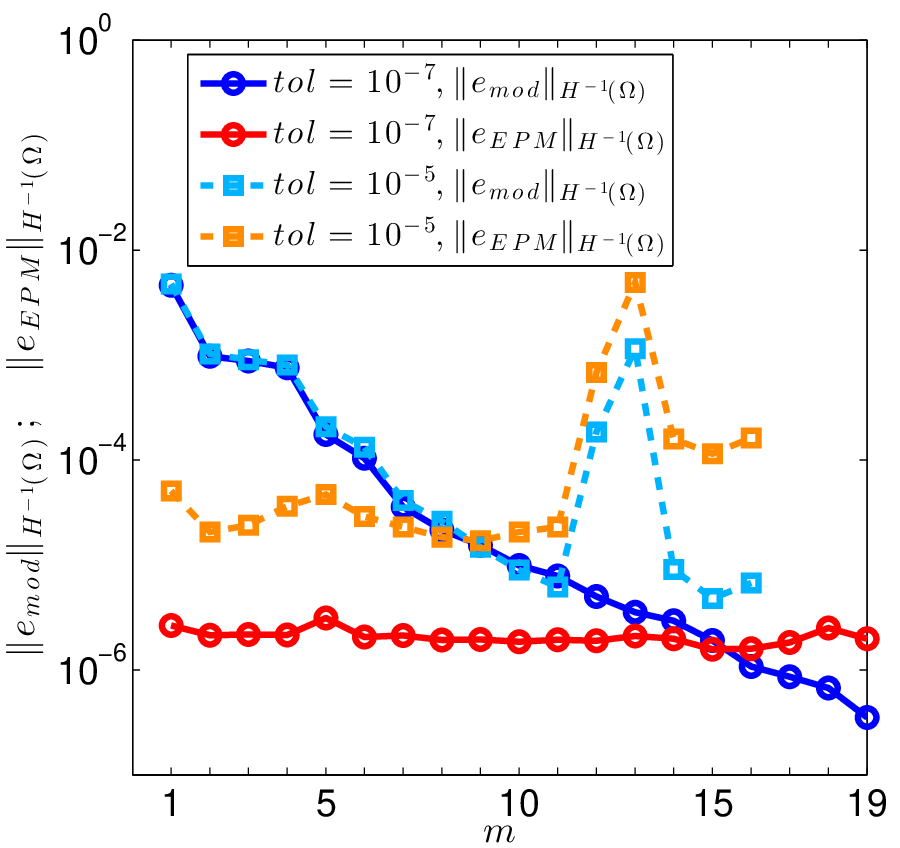}\label{fig4.5e}}
\subfloat[{\scriptsize $\Delta_{m}^{k},\tilde{\Delta}_{m}^{k}, |e^{k}_{m}|_{H^{1}}$}]{
\includegraphics[scale = 0.35]{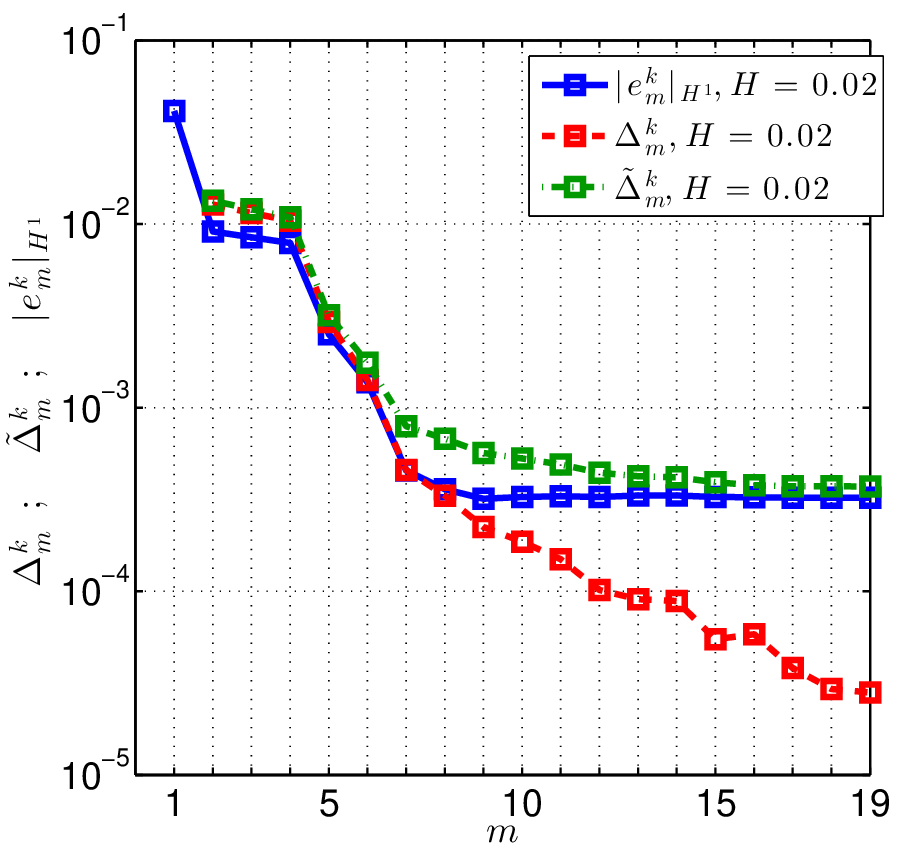}\label{fig4.5b}}
\subfloat[{\scriptsize $\|e_{EPM}\|_{L^{2}}$, $e^{k}_{POD}$ }]{
\includegraphics[scale = 0.35]{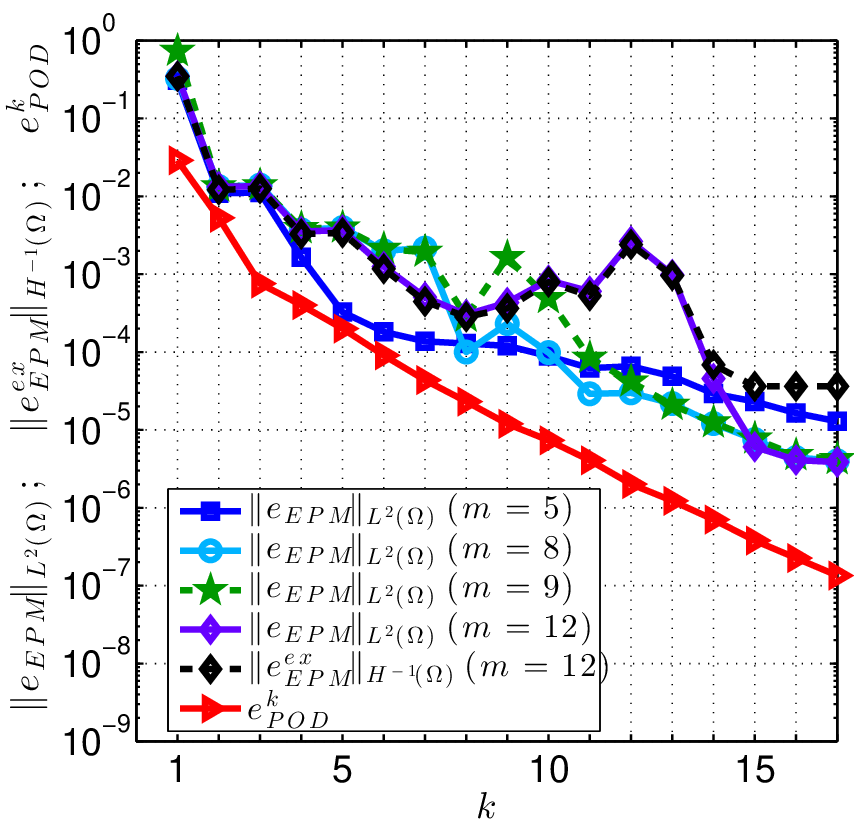}\label{fig4.5f}}
\caption{{\footnotesize Test case 1: Comparison of the a posteriori error estimator $\Delta_{m}^{k}$ with $|e_{m}^{k}|_{H^{1}(\Omega)}$ for decreasing $H$ for $\varepsilon_{\mbox{{\tiny tol}}}^{\text{{\tiny EPM}}}=10^{-5}$ (A) and $\varepsilon_{\mbox{{\tiny tol}}}^{\text{{\tiny EPM}}}=10^{-7}$ (B) . Comparison of $\|e_{\mbox{{\tiny EPM}}}\|_{H^{-1}(\Omega)}$ and $\|e_{\mbox{{\tiny EPM}}}^{\mbox{{\tiny ex}}}\|_{H^{-1}(\Omega)}$ for $\varepsilon_{\mbox{{\tiny tol}}}^{\text{{\tiny EPM}}}=10^{-5},10^{-7}$ and $H =0.02$ (C), and $\|e_{\mbox{{\tiny mod}}}\|_{H^{-1}(\Omega)}$ and $\|e_{\mbox{{\tiny EPM}}}\|_{H^{-1}(\Omega)}$ for $\varepsilon_{\mbox{{\tiny tol}}}^{\text{{\tiny EPM}}}=10^{-5},10^{-7}$ and $H =0.005$ (D). Comparison of  $\Delta_{m}^{k}$ with $\tilde{\Delta}_{m}^{k}$ and $|e_{m}^{k}|_{H^{1}(\Omega)}$ for $\varepsilon_{\mbox{{\tiny tol}}}^{\text{{\tiny EPM}}}=10^{-7}$ and $H=0.02$ (E), and $\|e_{\mbox{{\tiny EPM}}}\|_{L^{2}(\Omega)}$, $\|e_{\mbox{{\tiny EPM}}}^{\mbox{{\tiny ex}}}\|_{H^{-1}(\Omega)}$, and $e^{k}_{\mbox{{\tiny POD}}}$ for $\varepsilon_{\mbox{{\tiny tol}}}^{\text{{\tiny EPM}}}=10^{-7}$ and $H =0.02$ (F).}\label{fig4.5}}
\end{figure}
\begin{figure}[t]
\vspace{-10pt}
\center
\includegraphics[scale = 0.6]{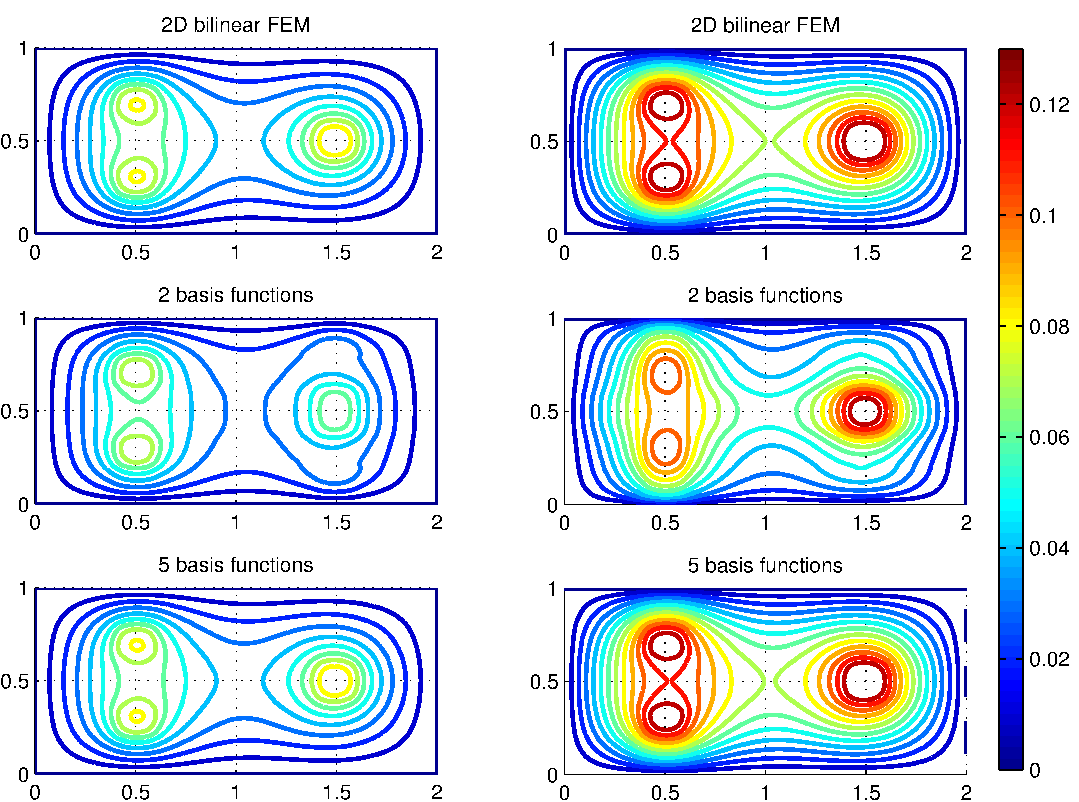}  
\caption{{\footnotesize Test case 2: Comparison from top to bottom: the reference 2D bilinear FE solution $p^{H \times h}$ \eqref{truth_nonlin} and the discrete reduced solution $p_{m,k}^{H}$ using 2 and 5 basis functions for $c_{0}=0.15$ (left) and $c_{0}=0.075$ (right); $c_{4} = 12$, $N_{H}= 400, n_{h}=200, N_{H'} = 10$.}
\label{fig4.7}}
\vspace{-10pt}
\end{figure}

Next, we investigate the effectivity of the a posteriori error estimators derived in \S \ref{sect-apostest}. For the present test case we obtained the following approximate values of the inf-sup stability factor
$
\beta_{2,p}^{\text{\tiny app}} \approx 0.09820 \enspace (H = 0.02)$, $\beta_{2,p}^{\text{\tiny app}} \approx 0.09736 \enspace (H = 0.01)$, $\beta_{2,p}^{\text{\tiny app}} \approx 0.09637 \enspace (H = 0.005),
$
which seems consistent with the considered choice of the ellipticity constant $c_{0}=0.1$. A comparison of the approximate values with the exact inf-sup stability factor in Tab.~\ref{compare-inf-sup} shows that the approximation procedure proposed in \S \ref{subsect_constants} yields a very accurate approximation but does in general not provide a lower bound for the inf-sup stability factor. 

We obtain $\tau_{m,2}^{k} < 1$ for $m\geq 2$ for $\varepsilon_{\mbox{{\tiny tol}}}^{\text{{\tiny EPM}}}=10^{-7}$ for all considered discretizations and for $\varepsilon_{\mbox{{\tiny tol}}}^{\text{{\tiny EPM}}}=10^{-5}$ for all $m \geq 2$ except $m=12,13$ for $H=0.01$ and $m=13$ for $H=0.005$ due to the instability in the EPM plateau. It can be seen in Fig.~\ref{fig4.5a} for $\varepsilon_{\mbox{{\tiny tol}}}^{\text{{\tiny EPM}}}=10^{-5}$ that $\Delta^{k}_{m}$ \eqref{delta_BBR_H1} is an upper bound for $|e_{m}^{k}|_{H^{1}(\Omega)}$, which is very sharp as the effectivities vary between $1.1$ and $3.4$ for the considered mesh sizes in Fig.~\ref{fig4.5a}. 

For $\varepsilon_{\mbox{{\tiny tol}}}^{\text{{\tiny EPM}}}=10^{-7}$ we observe in Fig.~\ref{fig4.5c} that $\Delta^{k}_{m}$ is an upper bound for $|e_{m}^{k}|_{H^{1}(\Omega)}$ for $m\leq 7$ ($H=0.02$), for $m\leq 8$ ($0.01$), and $m\leq10$ ($H= 0.005$), respectively. The fact that $\Delta^{k}_{m}$ underestimates the error for higher model orders for $\varepsilon_{\mbox{{\tiny tol}}}^{\text{{\tiny EPM}}}=10^{-7}$ is probably due to the fact that $\|e_{\text{{\tiny EPM}}}^{\text{{\tiny ex}}}\|_{H^{-1}(\Omega)}$ lies above $\|e_{\text{{\tiny EPM}}}\|_{H^{-1}(\Omega)}$ (see Fig.~\ref{fig4.5d}), as the former also takes discretization errors into account which are not included in the latter. Note that this explanation is consistent with the observation that the finer the mesh the higher the model error for which $\Delta^{k}_{m}$ starts to underestimate the error (see Fig.~\ref{fig4.5c}) and the finding that for $H=0.02$ an error estimator $\tilde{\Delta}^{k}_{m}$ in which $\|e_{\text{{\tiny EPM}}}\|_{H^{-1}(\Omega)}$ is replaced by $\|e_{\text{{\tiny EPM}}}^{\text{{\tiny ex}}}\|_{H^{-1}(\Omega)}$ provides an upper bound for  $|e_{m}^{k}|_{H^{1}(\Omega)}$ (see Fig.~\ref{fig4.5b}). For $\varepsilon_{\mbox{{\tiny tol}}}^{\text{{\tiny EPM}}}=10^{-5}$,  $\|e_{\text{{\tiny EPM}}}\|_{H^{-1}(\Omega)}$ and $\|e_{\text{{\tiny EPM}}}^{\text{{\tiny ex}}}\|_{H^{-1}(\Omega)}$ nearly coincide (see Fig.~\ref{fig4.5d}). We therefore conclude that the a posteriori bound for the adaptive EPM \eqref{apost-equation-epm} yields a very good approximation of $\|e_{\text{{\tiny EPM}}}^{\text{{\tiny ex}}}\|_{L^{2}(\Omega)}$ and thus $\|e_{\text{{\tiny EPM}}}^{\text{{\tiny ex}}}\|_{H^{-1}(\Omega)}$ if the discretization error is not dominant. A (standard) term which estimates this discretization error may be added to error estimator but this is beyond the scope of this paper. 

Here, we have set the tolerance for the POD determining the richer collateral basis space $W_{k'}$ to $tol_{k'} \varepsilon_{\mbox{{\tiny tol}}}^{\text{{\tiny EPM}}}$ with $tol_{k'}  =10^{-2}$ (\S \ref{adapt-RB-HMR-epm}). This yielded on average $k'-k \approx 5$ during the adaptive refinement procedure and for the certification for the method we obtained $k'-k \approx 6.33$ for $\varepsilon_{\mbox{{\tiny tol}}}^{\text{{\tiny EPM}}}=10^{-7}$ and $k'-k \approx 11$ for $\varepsilon_{\mbox{{\tiny tol}}}^{\text{{\tiny EPM}}}=10^{-5}$.

We recall that we have assumed \eqref{mod_dom_ei_H1}
\begin{equation}\label{mod_dom_ei_H1_num}
\|e_{\text{{\tiny EPM}}}\|_{H^{-1}(\Omega)} \leq \|e_{\text{{\tiny EPM}}}^{\text{{\tiny ex}}} \|_{H^{-1}(\Omega)} \leq c_{\text{{\tiny err}}} \| e_{\text{{\tiny mod}}} \|_{H^{-1}(\Omega)} \quad \text{for} \enspace c_{\text{{\tiny err}}} \in [0,1)
\end{equation}
and $\tau_{m,2}^{k} \leq 0.5 C_{err}$ with $C_{err}:=(1-c_{err})/(1 + c_{err})$ to prove the effectivity of $\Delta_{m}^{k}$. Fig.~\ref{fig4.5e} illustrates the convergence behavior of $\|e_{\text{{\tiny EPM}}} \|_{H^{-1}(\Omega)}$ and $\| e_{\text{{\tiny mod}}} \|_{H^{-1}(\Omega)}$ for increasing $m$ and $H = 0.005$. We observe that for $\varepsilon_{\mbox{{\tiny tol}}}^{\text{{\tiny EPM}}}=10^{-5}$ inequality \eqref{mod_dom_ei_H1_num} is satisfied for $m\leq 8$  and for $\varepsilon_{\mbox{{\tiny tol}}}^{\text{{\tiny EPM}}}=10^{-7}$ for $m\leq 15$, keeping in mind that $\|e_{\text{{\tiny EPM}}}^{\text{{\tiny ex}}} \|_{H^{-1}(\Omega)}$ might be higher for $\varepsilon_{\mbox{{\tiny tol}}}^{\text{{\tiny EPM}}}=10^{-7}$. As a consequence the (tighter) requirement $\tau_{m,2}^{k} \leq 0.5 C_{err}$ is satisfied for $H = 0.005$ for $\varepsilon_{\mbox{{\tiny tol}}}^{\text{{\tiny EPM}}}=10^{-5}$ for $m=5,...,8$ and for $\varepsilon_{\mbox{{\tiny tol}}}^{\text{{\tiny EPM}}}=10^{-7}$ for $m=5,...,15$, and for fewer number of basis functions for coarser discretizations. However, we emphasize that also for values of $m$ for which the assumption $\tau_{m,2}^{k} \leq 0.5 C_{err}$ is not fulfilled, we often observe that the effectivity of $\Delta_{m}^{k}$ can be bounded by a constant smaller than $5$ which is independent of $m$ (see Fig.~\ref{fig4.5a} and Fig.~\ref{fig4.5c}).

Finally, we investigate the a priori bound of the adaptive EPM derived in Theorem \ref{apriori-epm}.  As the convergence behavior of $\lambda^{k}$ and $\|\int_{{\omega}}\mathcal{I}_{L}[A(p_{m,k}^{H})]\kappa_{k}\|_{L^{2}(\Omega_{1D})}^{2}$ coincides (see Fig.~\ref{fig4.4}) we may follow Proposition \ref{apost-epm} to obtain an error estimator for the adaptive EPM. Note that \eqref{apost-equation-epm} is a probabilistic result and that the term $\mathcal{O}_{\mathcal{P}}(n^{-1/4})$ does not provide an upper bound for the integration error due to the application of the Monte-Carlo method \cite{Caf98}. However, Fig~\ref{fig4.5f} shows that apart from some deviations due to the EPM-plateau the error $\|e_{\text{{\tiny EPM}}} \|_{L^{2}(\Omega)}$ and the POD-error 
have approximately the same convergence rate. The integration error can thus be estimated by the POD-error. Hence, we employed the sum of the POD-error $e^{k}_{\text{{\tiny POD}}}$ and $\|e_{\text{{\tiny EPM}}}\|_{H^{-1}(\Omega)}$ as a bound in our numerical experiments and decreased the tolerance  $\varepsilon_{\text{{\tiny tol}}}$ in \eqref{apost-equation-epm} by $10^{-1}$ to account for the integration error and thus the deviation between the curves in Fig.~\ref{fig4.5f}. 

For the sake of completeness we note that the input arguments of Algorithm \ref{adapt-RB-HMR} \textsc{Adaptive-RB-HMR}, have been chosen as  $G_{0}=[0,0.5,1,1.5,2]\times[-0.5,0.5]\times[-1,1]$, $m_{\mbox{{\scriptsize max}}} = 2$, $i_{max} = 2$, $n_{\Xi} = 10$, $|\Xi_{c}| = 50$, $\theta = 0.05$, $\sigma_{thres}=(i_{max}-1)\cdot \lceil\diam(g)\rceil + 1$ for an element $g \in G_{0}$ and $\varepsilon_{\mbox{{\tiny tol}}}^{\text{{\tiny HMR}}} = 10^{-5}$, $\varepsilon_{\mbox{{\tiny tol}}}^{\text{{\tiny err}}} = 10^{-9}$, $\varepsilon_{\mbox{{\tiny tol}}}^{\text{{\tiny c}}} = 0.1$ for all computations for this test case. The average sample size has been $n_{\mbox{{\scriptsize train}}} \approx 520.$

\begin{figure}[t]
\centering
\subfloat[{\scriptsize 1 quad. point}]{
\includegraphics[scale = 0.25]{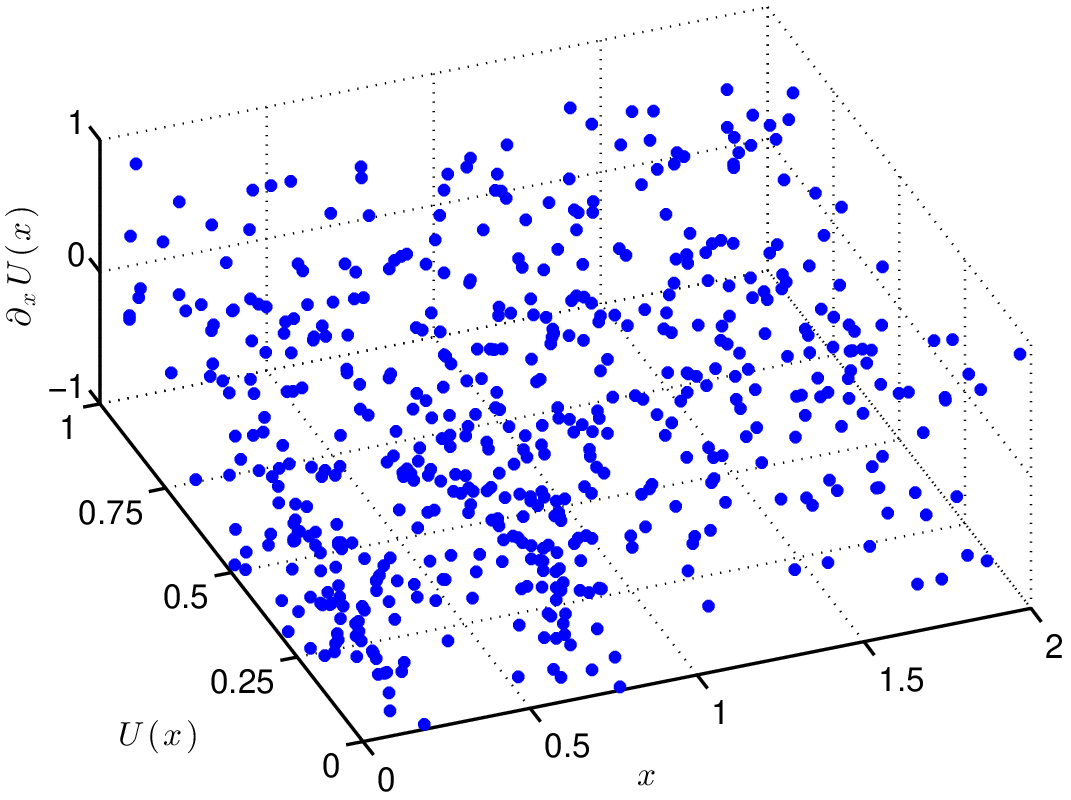} \label{fig4.8a}}
\subfloat[{\scriptsize $[x_{1}^{q},U(x_{1}^{q}),\partial_{x}U(x_{1}^{q})]$ (left), $[x_{2}^{q},U(x_{2}^{q}),\partial_{x}U(x_{2}^{q})]$ (right)}]{
\includegraphics[scale = 0.25]{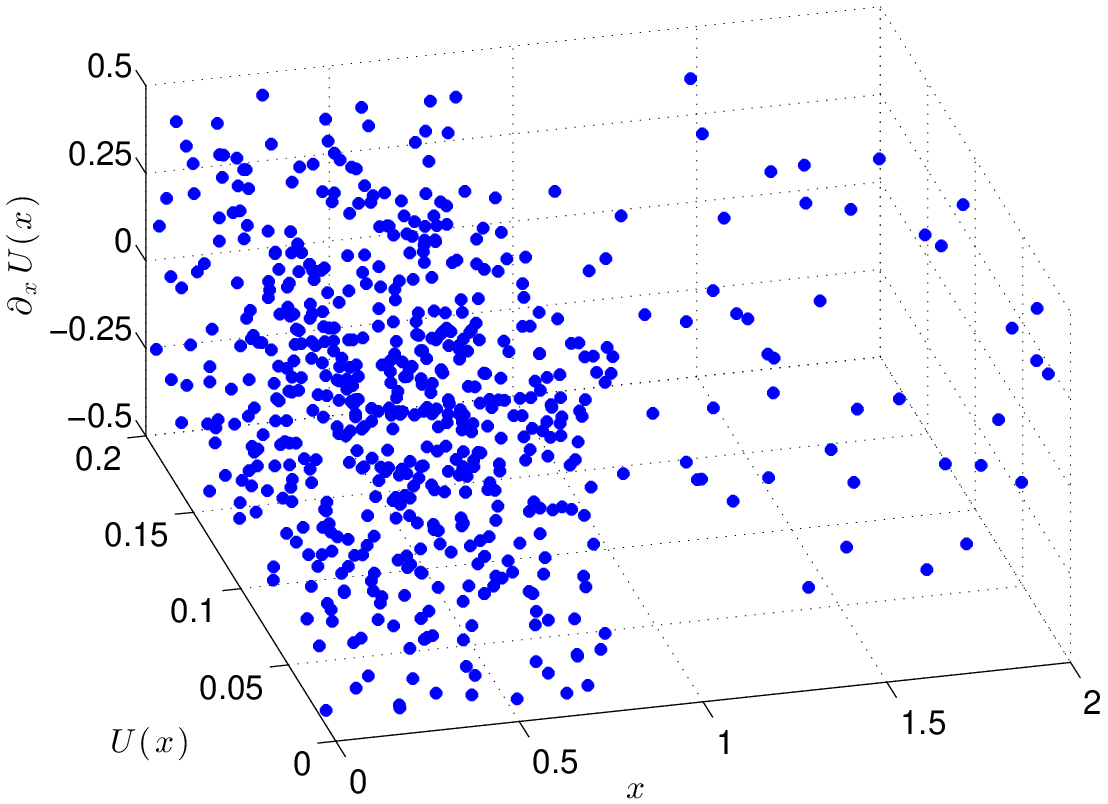}\includegraphics[scale = 0.25]{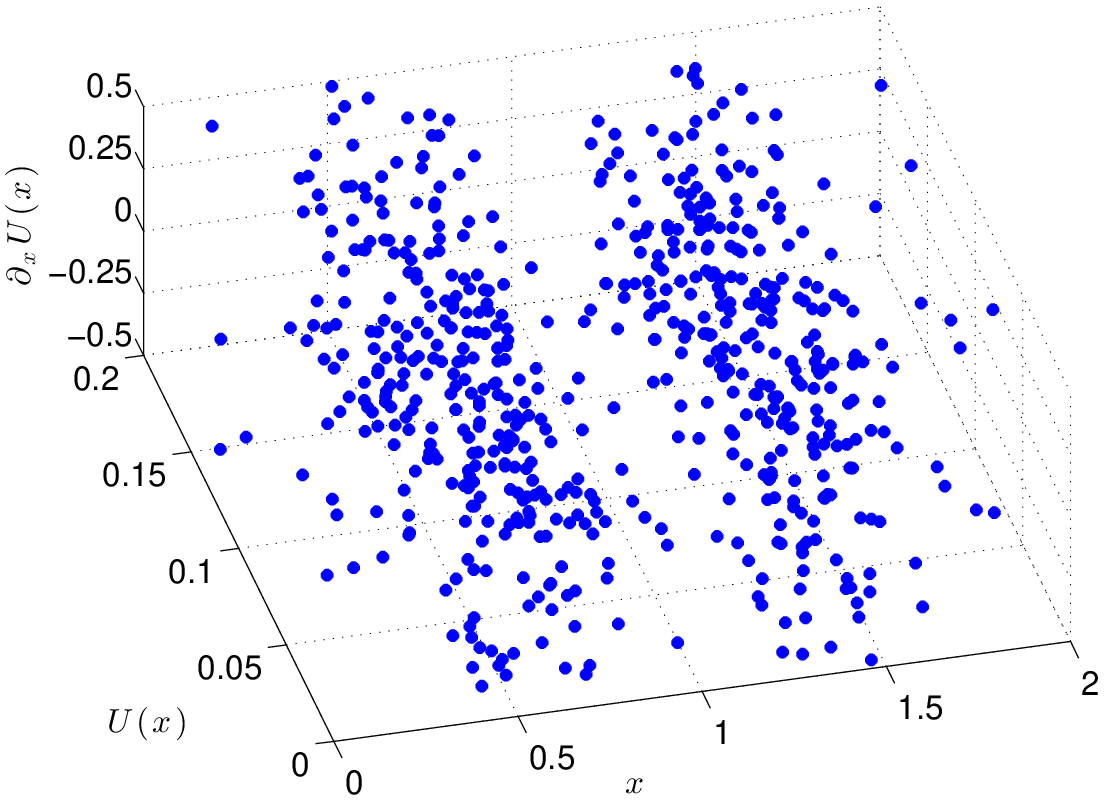}\label{fig4.8b}}
\caption{{\footnotesize Test case 2: Plot of the adaptively refined training set generated by Algorithm \ref{adapt-RB-HMR} \textsc{Adaptive-RB-HMR} when using 1 quadrature point in \eqref{1D_prob_quad_para_nonlin} (A) or 2 (B) for $N_{H}=200$, $n_{h}=100$, $N_{H'}=10$.}\label{fig4.8}}
\end{figure}

\begin{figure}[t]
\centering
\subfloat[{\scriptsize $|e_{m}^{k}|_{H^{1}}^{rel}$}]{
\includegraphics[scale = 0.35]{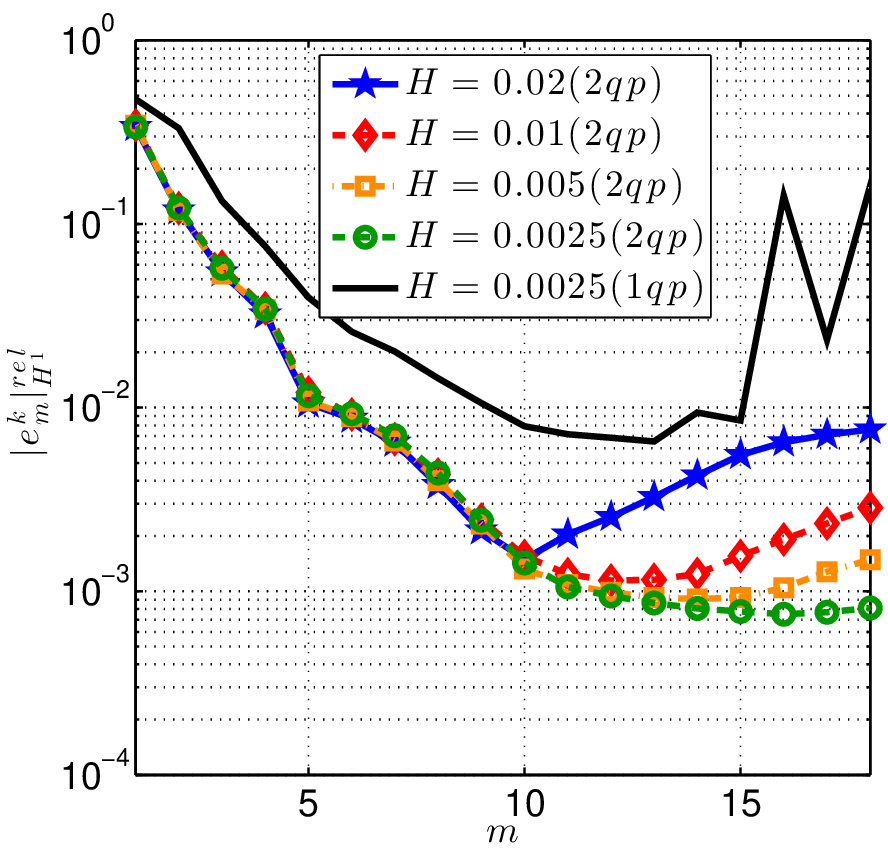}\label{fig4.9d}}
\subfloat[{\scriptsize $\|e_{m}^{k}\|_{L^{2}(\Omega)}^{rel}$, $e_{m}^{\text{{\tiny POD}}}$}]{
\includegraphics[scale = 0.35]{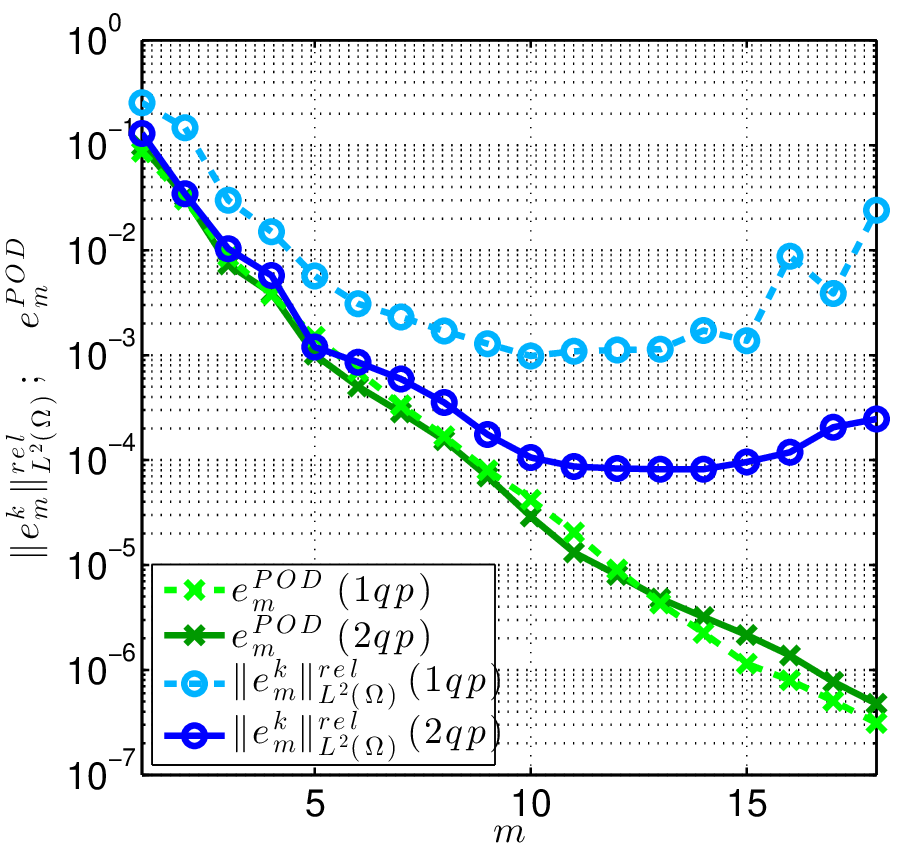}\label{fig4.9e}}
\subfloat[{\scriptsize $\|\bar{p}_{m,k}^{H}\|_{L^{2}(\Omega_{1D})}^{2}$, $\lambda_{m}$}]{
\includegraphics[scale = 0.35]{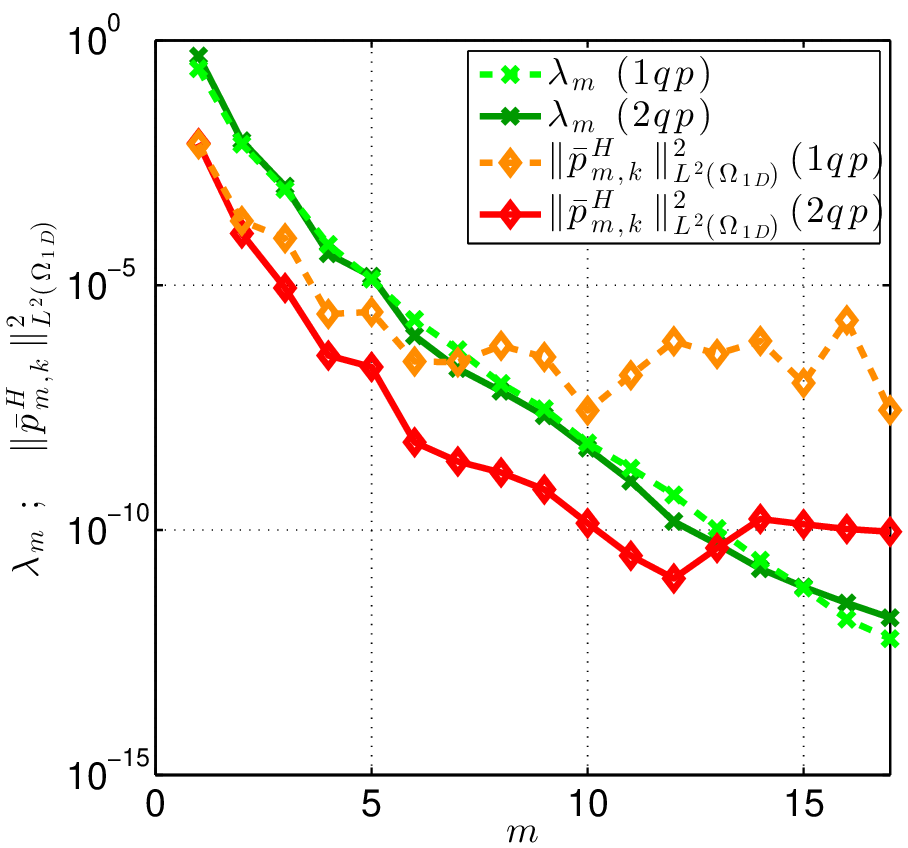}\label{fig4.9f}}
\caption{{\footnotesize Test case 2: Comparison of the convergence behavior of $|e_{m}^{k}|_{H^{1}}^{rel}$ for one quadrature point (1qp) and two quadrature points (2qp) in \eqref{1D_prob_quad_para_nonlin} (A); $\|e_{m}^{k}\|_{L^{2}(\Omega)}^{rel}$ and $e_{m}^{\text{{\tiny POD}}}$ for $Q=1,2$ in \eqref{quad_formula_nonlin} for $H=0.0025$ (B); $\|\bar{p}_{m,k}^{H}\|_{L^{2}(\Omega_{1D})}^{2}$ and  $\lambda_{m}$ for $Q=1,2$ in \eqref{quad_formula_nonlin} for $H=0.0025$ (C). $N_{H'}=10$ for all pictures.}\label{fig4.9}}
\end{figure}


\subsection*{Test case 2}

In this test case we investigate the convergence behavior and computational efficiency of the RB-HMR approach for the approximation of non-smooth solutions of \eqref{model problem_num}. We choose $\Omega = (0,2)\times(0,1)$, $c_{4}=12$ and unless otherwise stated $c_{0}=0.075$. We prescribe as a source term the characteristic function $s(x,y) = \chi_{D_{1}\cup D_{2}\cup D_{3}}$, where $D_{1} = \{(x,y) \in \Omega\,: \, 0.4 \leq x \leq 0.6 \enspace \text{and} \enspace 0.2 \leq y \leq 0.36\}$, $D_{2} = \{(x,y) \in \Omega\,: \, 0.4 \leq x \leq 0.6 \enspace \text{and} \enspace 0.64 \leq y \leq 0.8\}$ and $D_{3} = \{(x,y) \in \Omega\,: \, 1.4 \leq x \leq 1.6 \enspace \text{and} \enspace 0.4 \leq y \leq 0.6\}$ and thus have that the solution $p$ of \eqref{model problem_num} is in $W^{2,q}(\Omega)$ for $q < \infty$ \cite{CalRap1997}. The reference solutions $p^{H\times h}$ \eqref{truth_nonlin} for $c_{0}=0.15$ and $c_{0}=0.075$ are depicted at the top of Fig.~\ref{fig4.7} for $N_{H}= 400$ and $n_{h} =200$, 
where a convergence study has been done to ensure that $p^{H\times h}$ contains all essential features of the exact solution. The strengthened nonlinear effects for decreasing $c_{0}$ can nicely be observed by means of the increased range of $p^{H\times h}$ for $c_{0}=0.075$ and the much more localized peaks for $c_{0}=0.15$. Comparing the reference solutions with its RB-HMR approximations, we see that for both $c_{0}=0.15$ and $c_{0}=0.075$ already $p_{2,20}^{H}$ contains the three peaks and that the contour lines of $p_{5,20}^{H}$ and $p^{H\times h}$ coincide  (Fig.~\ref{fig4.7}). \\
\begin{figure}[t]
\centering
\subfloat[{\scriptsize $|e_{m}^{k}|_{H^{1}}^{rel}$}]{
\includegraphics[scale = 0.35]{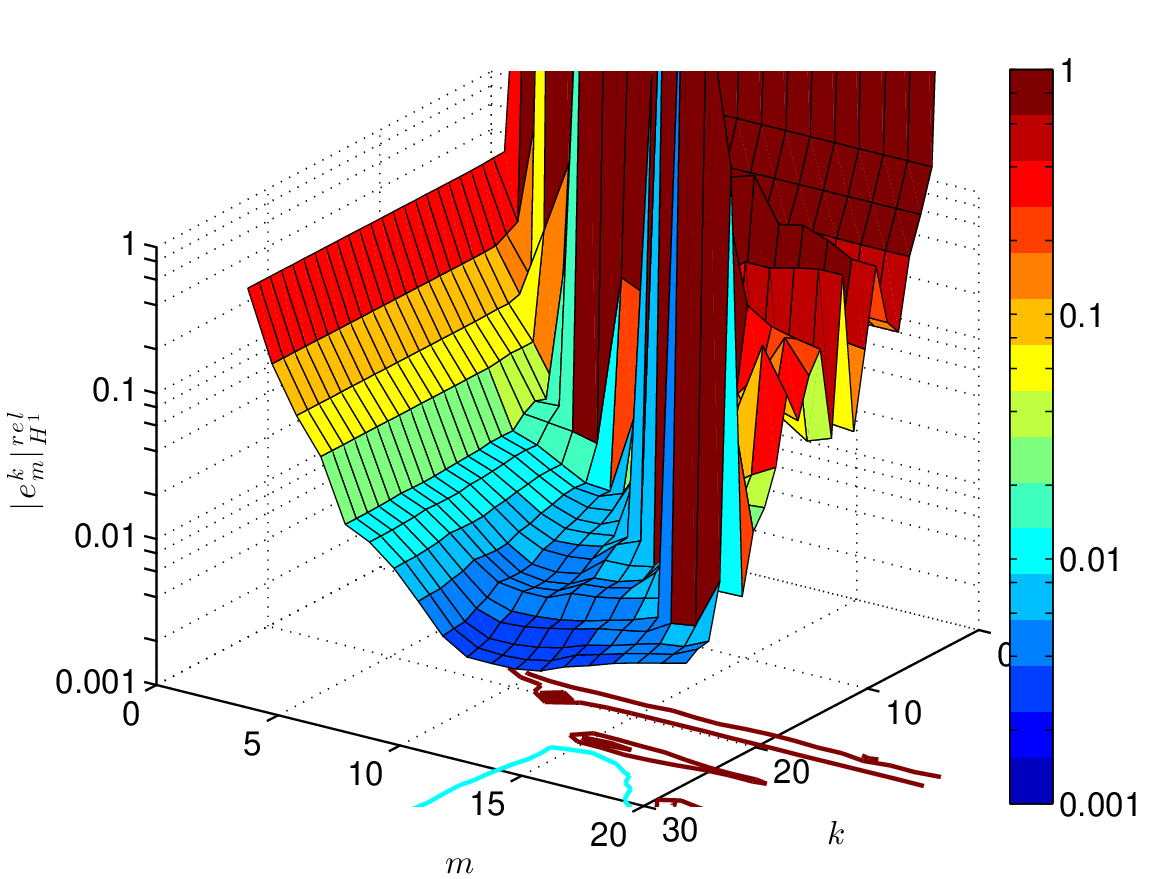} \label{fig4.10a}}
\subfloat[{\scriptsize $|e|_{H^{1}}^{rel}$}]{
\includegraphics[scale = 0.35]{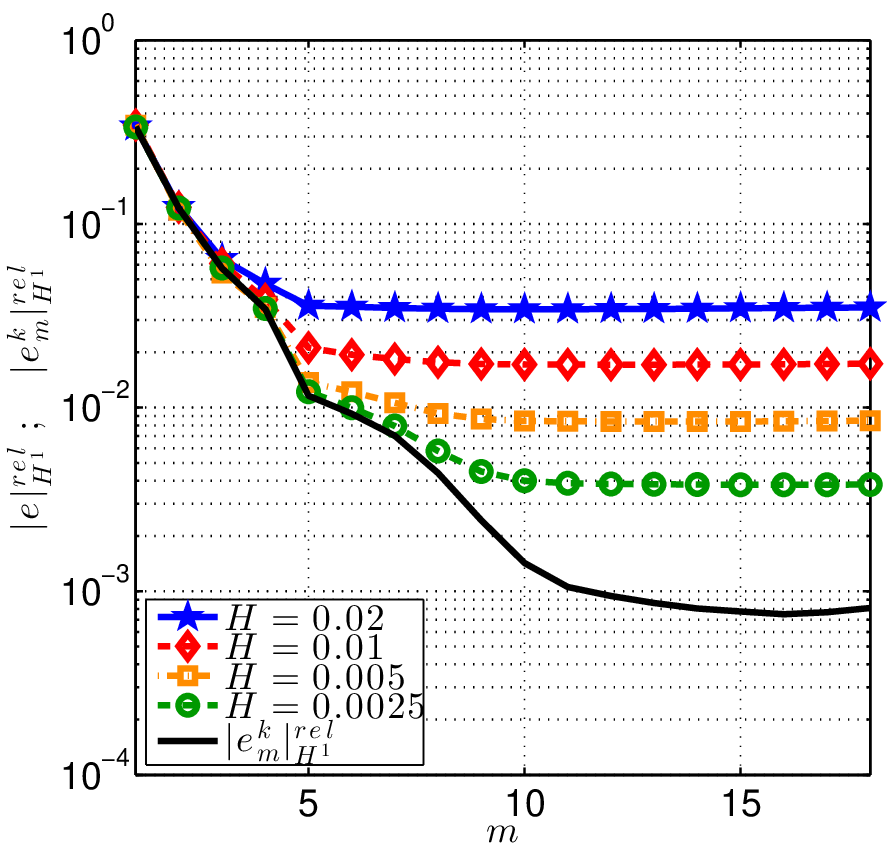}\label{fig4.10b}}
\caption{{\footnotesize Test case 2: Convergence behavior of $|e_{m}^{k}|_{H^{1}}^{rel}$ for increasing $m$ and $k$ for $N_{H}=200$, $n_{h}=100$ (A) and $|e|_{H^{1}}^{rel}$ for decreasing mesh size (B); both pictures: $Q =2$ in \eqref{quad_formula_nonlin}, $N_{H'}=10$.}\label{fig4.10}}
\end{figure}

\begin{figure}[t]
\centering
\includegraphics[scale = 0.37]{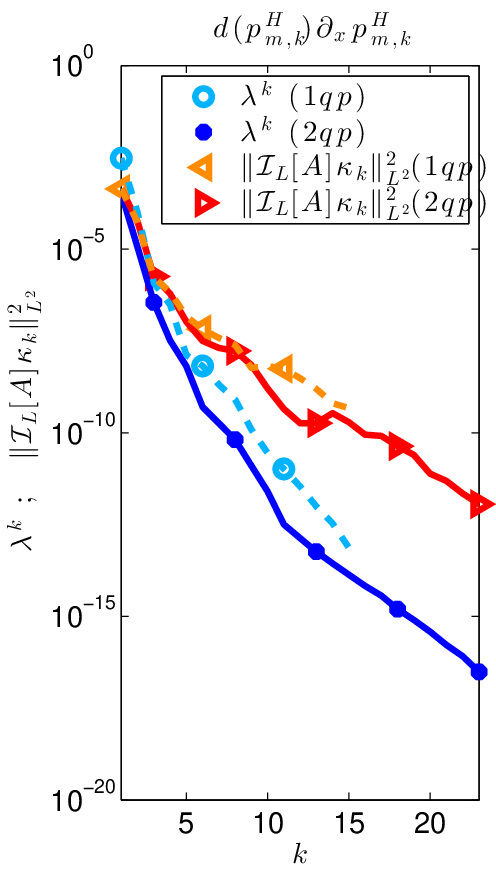}\enspace \includegraphics[scale = 0.37]{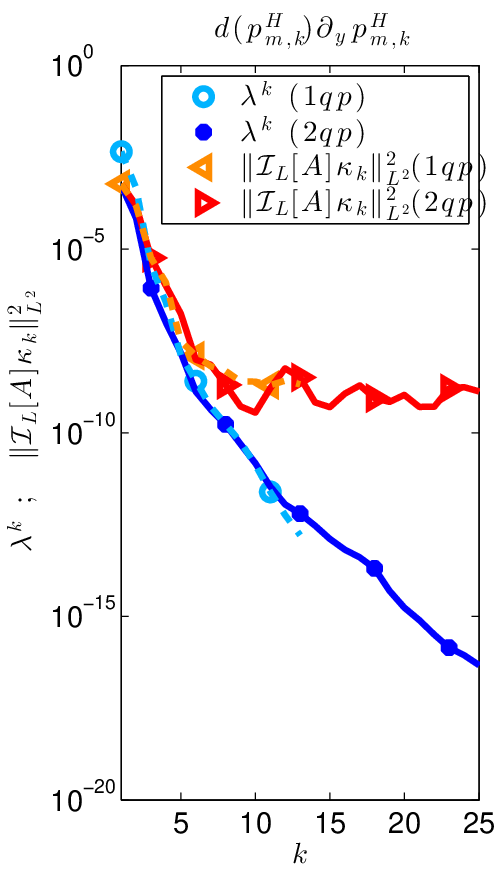}\enspace \includegraphics[scale = 0.37]{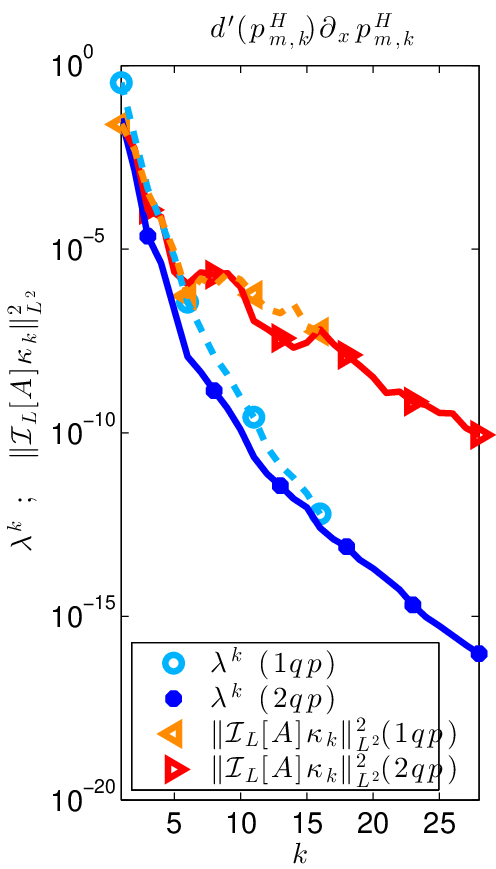}\enspace \includegraphics[scale = 0.37]{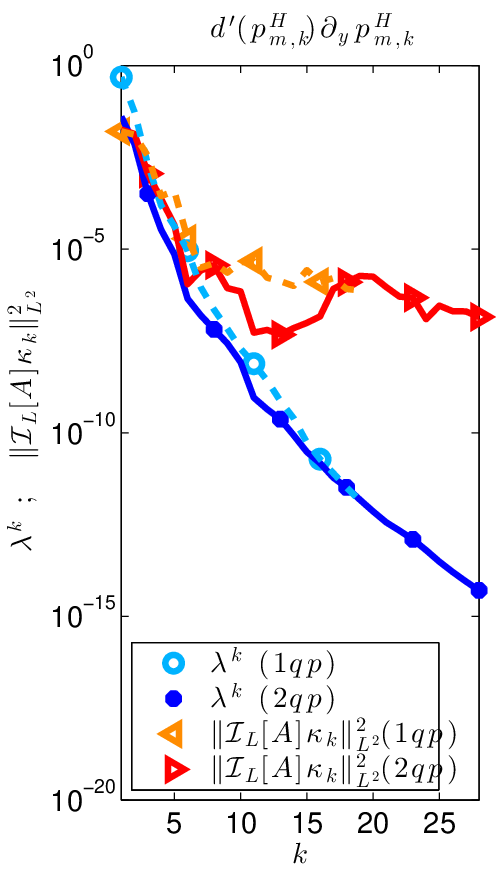}\enspace \includegraphics[scale = 0.37]{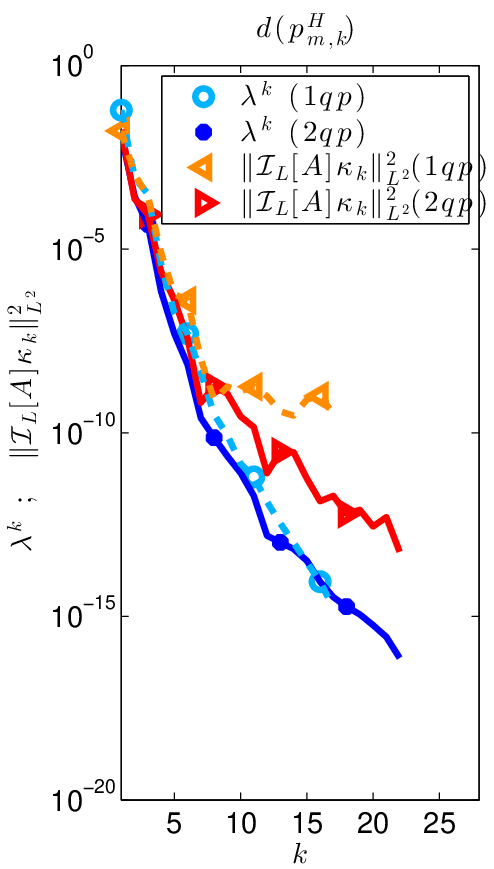}
\caption{{\footnotesize Test case 2: Comparison of $\lambda^{k}$ and $\|\int_{{\omega}}\mathcal{I}_{L}[A(p_{m,k}^{H})]\kappa_{k}\|_{L^{2}(\Omega_{1D})}^{2}$ for $Q=1$ in \eqref{quad_formula_nonlin} (1qp) and $Q=2$ in \eqref{quad_formula_nonlin} (2qp) for $H=0.005$ and $N_{H'}=10$.}\label{fig4.11}}
\end{figure}

\begin{figure}[t]
\centering
\subfloat[{\scriptsize $|e_{m}^{k}|_{H^{1}}$, $\tau_{m,2}^{k}$ (1qp)}]{
\includegraphics[scale = 0.35]{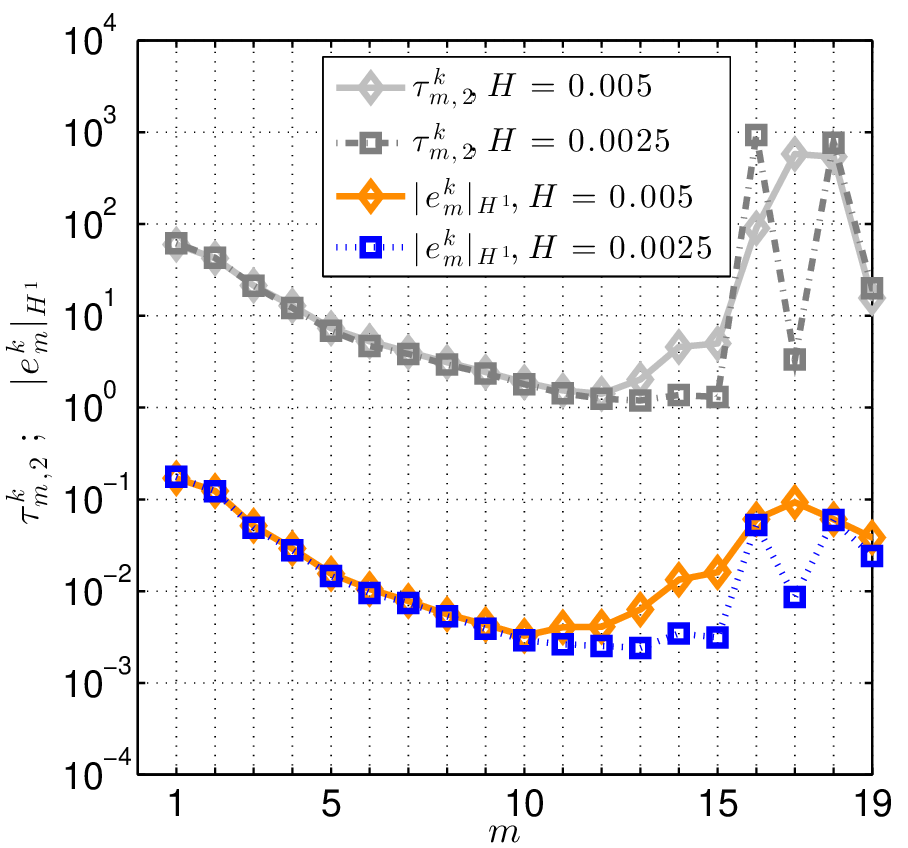} \label{fig4.12a}}
\subfloat[{\scriptsize $|e_{m}^{k}|_{H^{1}}$, $\Delta_{m}^{k}$ (2qp)}]{
\includegraphics[scale = 0.35]{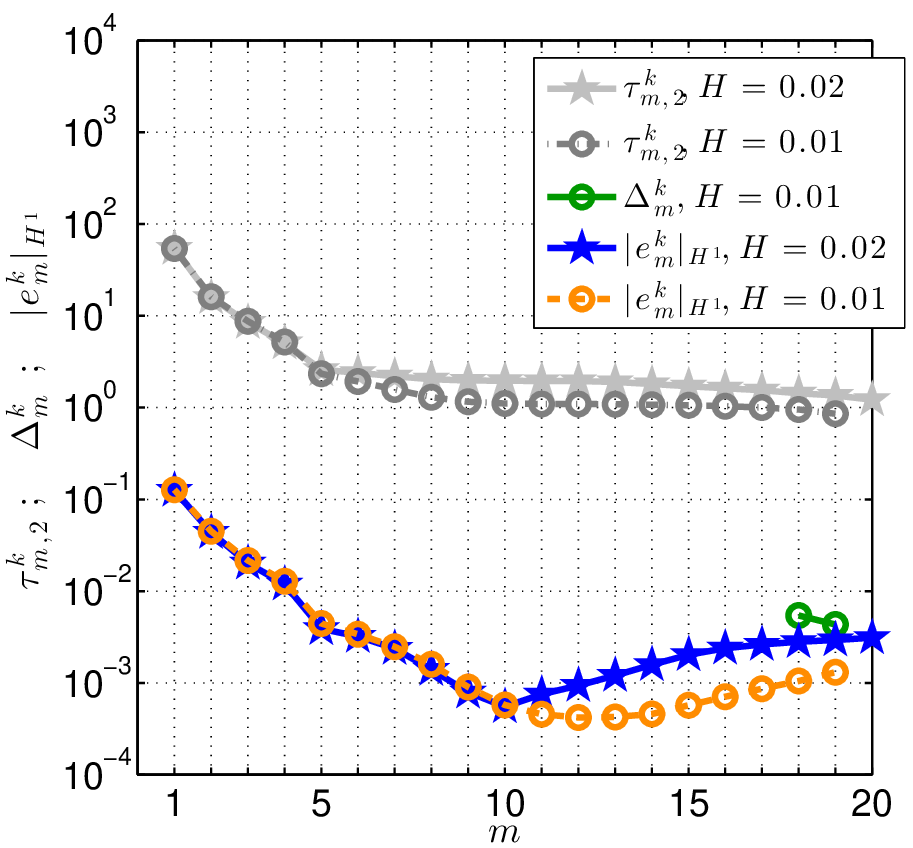}\label{fig4.12b}}
\subfloat[{\scriptsize $|e_{m}^{k}|_{H^{1}}$, $\Delta_{m}^{k}$ (2qp)}]{
\includegraphics[scale = 0.35]{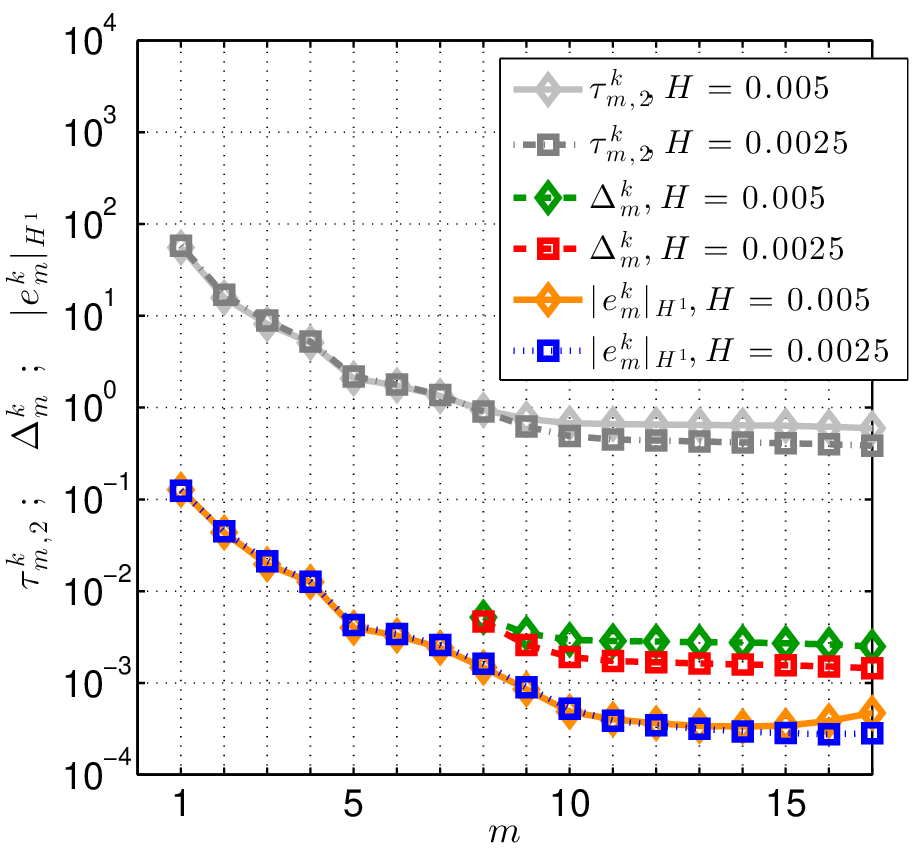}\label{fig4.12c}} \\
\subfloat[{\scriptsize $e_{EPM},e_{EPM}^{ex},e_{mod}$}]{
\includegraphics[scale = 0.35]{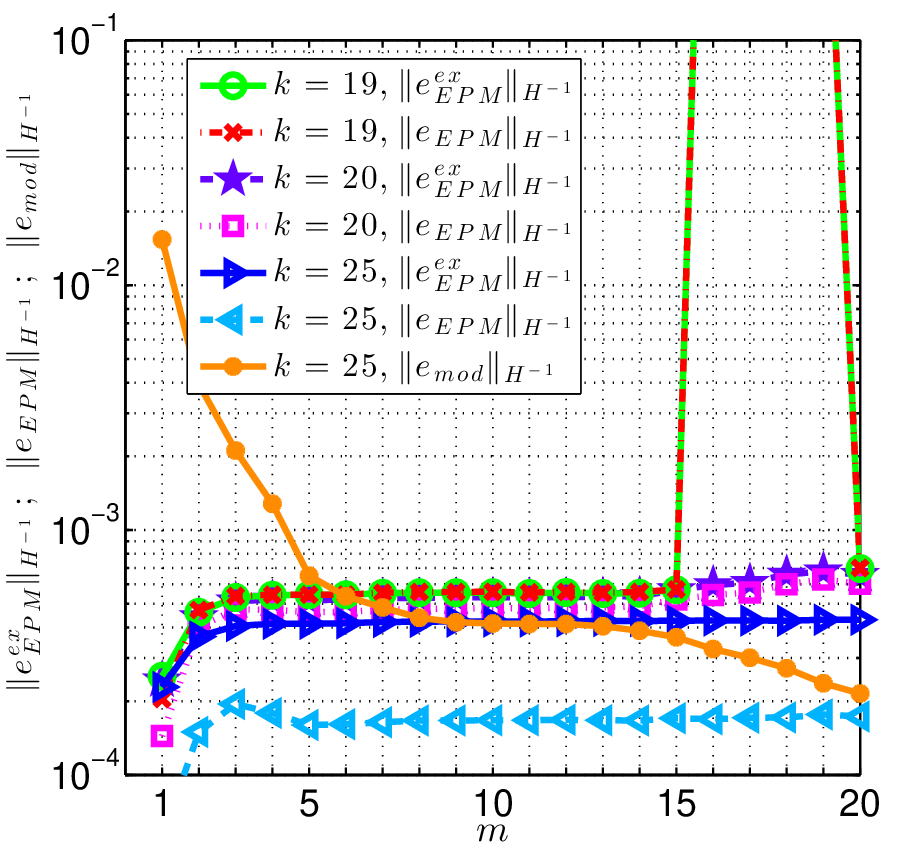}\label{fig4.12d}}
\subfloat[{\scriptsize $e_{EPM},e_{mod}$}]{
\includegraphics[scale = 0.35]{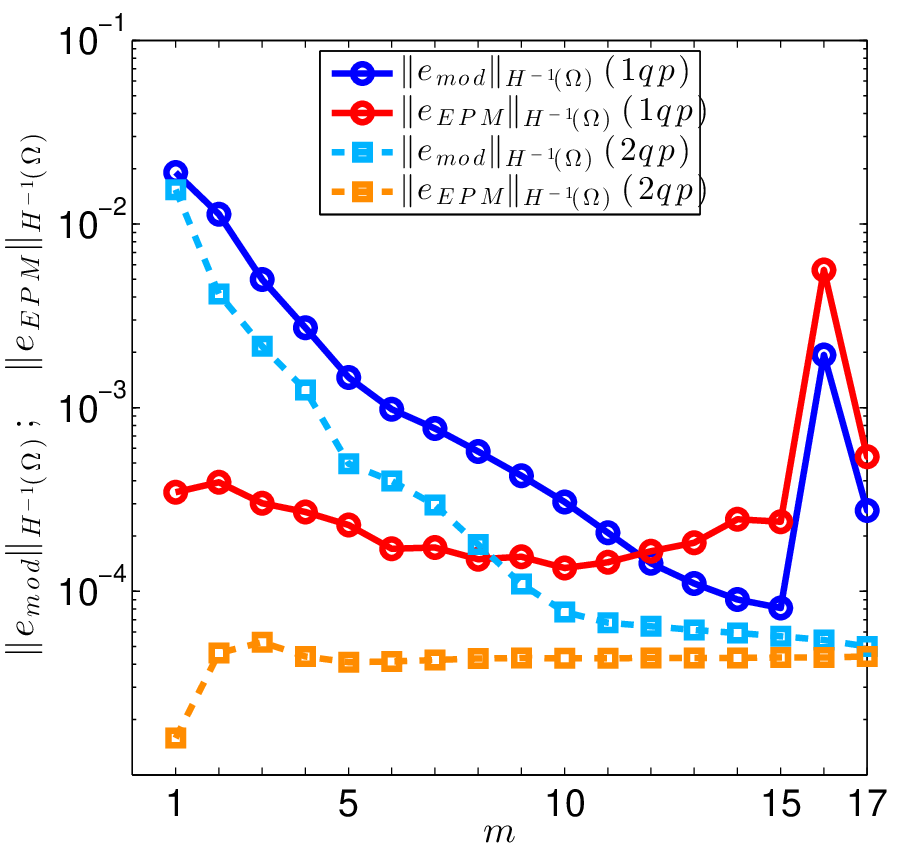}\label{fig4.12e}}
\subfloat[{\scriptsize $e^{k}_{L^{2}},e^{k}_{POD},e_{EPM}$}]{
\includegraphics[scale = 0.35]{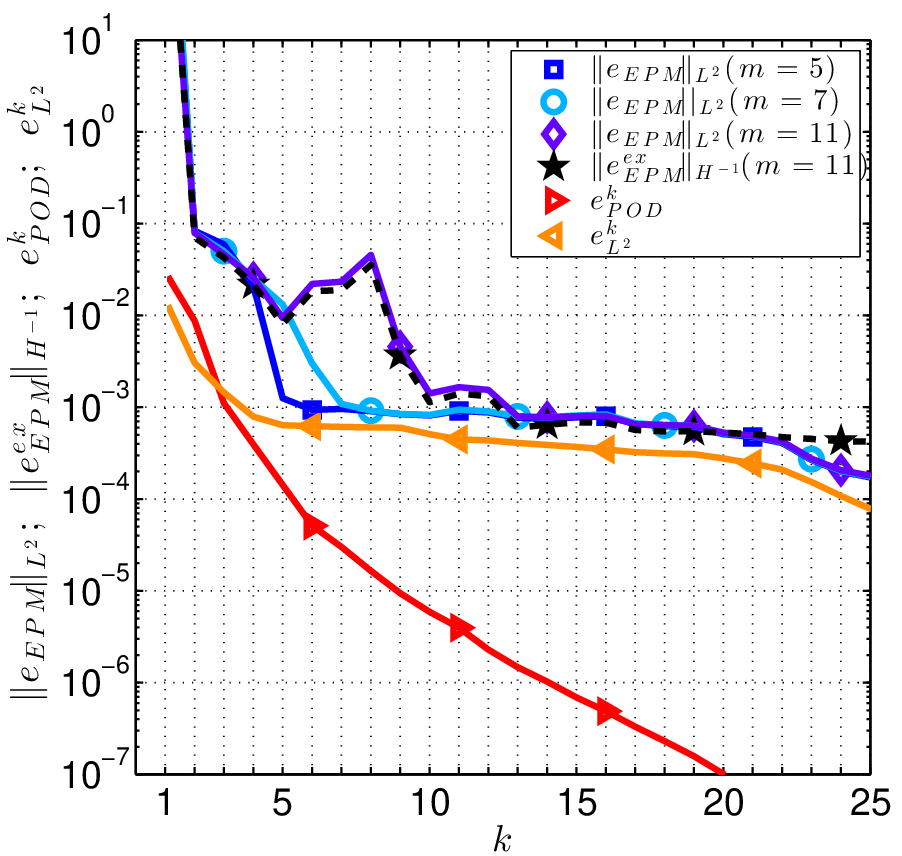}\label{fig4.12f}}
\caption{{\footnotesize Test case 2: Comparison of the a posteriori error estimator $\Delta_{m}^{k}$, the error
indicator $\tau_{m,2}^{k}$ with $|e_{m}^{k}|_{H^{1}(\Omega)}$ for $Q=1$ in \eqref{quad_formula_nonlin} (1qp) and $H = 0.005$ (solid) and $H=0.0025$ (dashed) (A) and  $Q=2$ in \eqref{quad_formula_nonlin} (2qp) and $H = 0.02$ and $H=0.01$ (B) and $H=0.005$ and $H=0.0025$ (C). Comparison of $\|e_{\mbox{{\tiny EPM}}}\|_{H^{-1}(\Omega)}$, $\|e_{\mbox{{\tiny EPM}}}^{\mbox{{\tiny ex}}}\|_{H^{-1}(\Omega)}$, and $\|e_{\mbox{{\tiny mod}}}\|_{H^{-1}(\Omega)}$ for $H= 0.02$ and $Q=2$ (D), and $\|e_{\mbox{{\tiny mod}}}\|_{H^{-1}(\Omega)}$ and $\|e_{\mbox{{\tiny EPM}}}\|_{H^{-1}(\Omega)}$ for $Q=1,2$ and $H =0.0025$ (E), and $\|e_{\mbox{{\tiny EPM}}}\|_{L^{2}(\Omega)}$, $\|e_{\mbox{{\tiny EPM}}}^{\mbox{{\tiny ex}}}\|_{H^{-1}(\Omega)}$, $e^{k}_{\mbox{{\tiny POD}}}$, and $e^{k}_{L^{2}}$ for $Q=2$ and $H =0.02$ (F).}\label{fig4.12}}
\end{figure}

The input arguments of Algorithm \ref{adapt-RB-HMR} (\textsc{Adaptive-RB-HMR}) have been chosen as $G_{0}=[0,0.4,0.8,1.2,1.6,2]$ $\times[0,1]\times[-1,1]$, $m_{\mbox{{\scriptsize max}}} = 2$, $i_{max} = 2$, $n_{\Xi} = 10$, $|\Xi_{c}| = 50$,  $\theta = 0.05$, $\sigma_{thres}=(i_{max}-1)\cdot \lceil\diam(g)\rceil + 1$ for an element $g \in G_{0}$ and $\varepsilon_{\mbox{{\tiny tol}}}^{\text{{\tiny HMR}}} = 10^{-3}$, $\varepsilon_{\mbox{{\tiny tol}}}^{\text{{\tiny EPM}}} = 10^{-7}$, $\varepsilon_{\mbox{{\tiny tol}}}^{\text{{\tiny err}}} = 10^{-9}$ for all computations for this test case employing one quadrature point in \eqref{1D_prob_quad_para_nonlin} and thus setting $Q=1$ in \eqref{quad_formula_nonlin}. This resulted in an average sample size of $n_{\mbox{{\scriptsize train}}} \approx 580$. For two quadrature points in \eqref{1D_prob_quad_para_nonlin} or $Q=2$ in \eqref{quad_formula_nonlin}  we have chosen $G_{0}=[0,0.4,0.8,1.2,1.6,2]\times[0,0.2]\times[-0.5,0.5]$,$m_{\mbox{{\scriptsize max}}} = 2$, $i_{max} = 1$, $n_{\Xi} = 4$, $\theta = 0.01$, $\sigma_{thres}=i_{max}\cdot \lceil\diam(g)\rceil + 1$, $\varepsilon_{\mbox{{\tiny tol}}}^{\text{{\tiny EPM}}} = 10^{-8}$, and $\varepsilon_{\mbox{{\tiny tol}}}^{\text{{\tiny err}}} = 10^{-10}$, which yielded on average $n_{\mbox{{\scriptsize train}}} \approx 600$. 
Fig.~\ref{fig4.8a} illustrates the training set $\Xi$ generated with Algorithm \ref{adapt-train} \textsc{AdaptiveTrainExtension} for $Q =1$. It can be seen that the training set is mainly refined at the two peaks at $x=0.5$ and near $x=0$, but not around the other peak in the solution at $x=1.5$. Using two quadrature points in \eqref{1D_prob_quad_para_nonlin} we observe a refinement of the training set in the expected regions, namely around the peaks  at $x=0.5$ and $x=1.5$ (Fig.~\ref{fig4.8b}). 

Analyzing the convergence behavior of $|e_{m}^{k}|_{H^{1}(\Omega)}^{rel}$ we detect an exponential convergence rate, which is much better for $Q=2$ (Fig.~\ref{fig4.9d}). Furthermore, we have observed for $Q=1$ a much stronger increase of $|e_{m}^{k}|_{H^{1}(\Omega)}^{rel}$ when entering the EPM-plateau especially for coarser mesh sizes. Note that additionally $\mathcal{D}$ has been shrunk when passing from $Q=1$ to $Q=2$, which further improved the rates. However, shrinking $\mathcal{D}$ without increasing $Q$ had no effect. A comparison of $\|e_{m}\|_{L^{2}(\Omega)}^{rel}$ and $e_{m}^{\mbox{{\tiny POD}}}$  in Fig.~\ref{fig4.9e} shows that for $Q=2$ the convergence rates coincide until $\|e_{m}\|_{L^{2}(\Omega)}^{rel}$ approaches the EPM-plateau, but differ for $Q =1$. Regarding $\lambda_{m}$ and $\|\bar{p}_{m,k}^{H}\|_{L^{2}(\Omega_{1D})}^{2}$ we observe in Fig.~\ref{fig4.9f} that their convergence rates significantly differ for $Q=1$, but coincide for $Q = 2$ for $m\leq 12$. The rise of $\|\bar{p}_{m,k}^{H}\|_{L^{2}(\Omega_{1D})}^{2}$ for $m > 12$ might be caused by the EPM-plateau. Thus, we conclude that for the present test case for $Q=2$ the discrete solution manifold $\mathcal{M}^{\mathcal{P}}_{\Xi}$ \eqref{manifold-nonlin} and the reference solution $p^{H\times h}$ are approximated with the same approximation quality by the reduction space $Y_{m}$. Note that the \textsc{QP-Indicator}  would have detected  in line \ref{qp1} of Algorithm \ref{adapt-train} that an increase of $Q$ is necessary, but would not have raised $Q$ further in line \ref{qp2} due to the coincidence of the rates of  $\lambda_{m}$ and $\|\bar{p}_{m,k}^{H}\|_{L^{2}(\Omega_{1D})}^{2}$ for $m\leq 10$. 

Fig.~\ref{fig4.10a} shows the error convergence of $|e_{m}^{k}|_{H^{1}(\Omega)}^{rel}$ for a simultaneous growth of  $m$ and $k$ for $H = 0.01$. We see on the one hand a strong increase of the error if $m$ exceeds $k$ but on the other hand a nice error decay and only a small increase in the EPM-plateau if $k \geq m + 5$ is satisfied. We suppose that the worse behavior of $p_{m,k}^{H}$ in the EPM-plateau compared to the previous test case (compare Fig.~\ref{fig4.3e} and Fig.~\ref{fig4.10a}) is due to the fact that the full solution of the present test case is non-smooth. An investigation of the convergence behavior of the relative total error  demonstrates that for $Q=2$ the EPM-plateau has no effect on $|e|_{H^{1}(\Omega)}^{rel}$ (cf.~Fig.~\ref{fig4.10b}), whereas for $Q =1$ the EPM-plateau influences the convergence of $|e|_{H^{1}(\Omega)}^{rel}$. 

Comparing the convergence rates of $\lambda^{k}$ with $\|\int_{{\omega}}\mathcal{I}_{L}[A(p_{m,k}^{H})]\kappa_{k}\|_{L^{2}(\Omega_{1D})}^{2}$ for the employed collateral basis spaces, we see in Fig.~\ref{fig4.11} that they are comparable for $k\leq5$ for $Q=1$ and $k\leq 12$ for $Q=2$, but clearly differ for higher values. This is due to the behavior of $\|\bar{p}_{m,k}^{H}\|_{L^{2}(\Omega_{1D})}^{2}$ (cf.~Fig.~\ref{fig4.9f}), which stagnate exactly for $m=5$ ($Q=1$) and $m=12$ ($Q=2$). However, we have observed that the level of the plateau of the coefficients $\|\bar{p}_{m,k}^{H}\|_{L^{2}(\Omega)}^{2}$ and $\| \int_{{\omega}}\mathcal{I}_{L}[A(p_{m,k}^{H})]\kappa_{k}\|_{L^{2}(\Omega_{1D})}^{2}$ reduces for decreasing mesh sizes which might indicate that their stagnation is related to the EPM-plateau. We hence suppose that since the solution of the present test case is non-smooth, in contrast to the previous example, the behavior of  $p_{m,k}^{H}$ in the EPM-plateau (compare Fig.~\ref{fig4.3e} and Fig.~\ref{fig4.10a})  also affects the coefficients $\|\bar{p}_{m,k}^{H}\|_{L^{2}(\Omega)}^{2}$ and $\| \int_{{\omega}}\mathcal{I}_{L}[A(p_{m,k}^{H})]\kappa_{k}\|_{L^{2}(\Omega_{1D})}^{2}$. \\

\begin{figure}[t]
\centering
\subfloat[{\scriptsize Runtime [min] for increasing $N_{H}$}]{
\includegraphics[scale = 0.35]{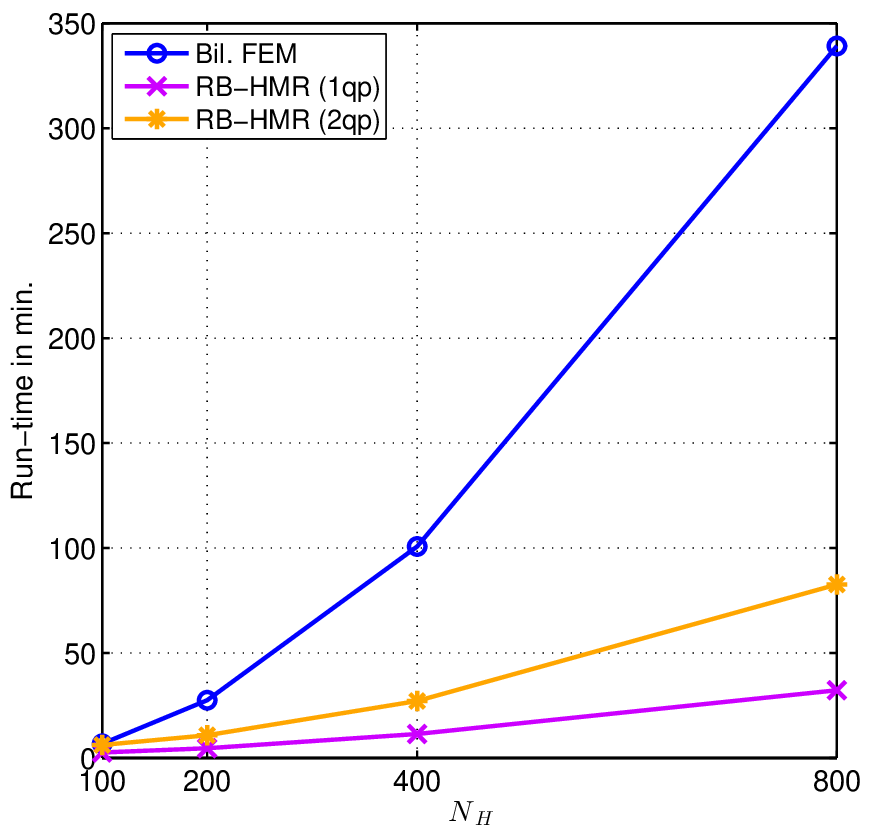} \label{fig4.15a}}
\subfloat[{\scriptsize Runtime [min] incl. certification}]{
\includegraphics[scale = 0.35]{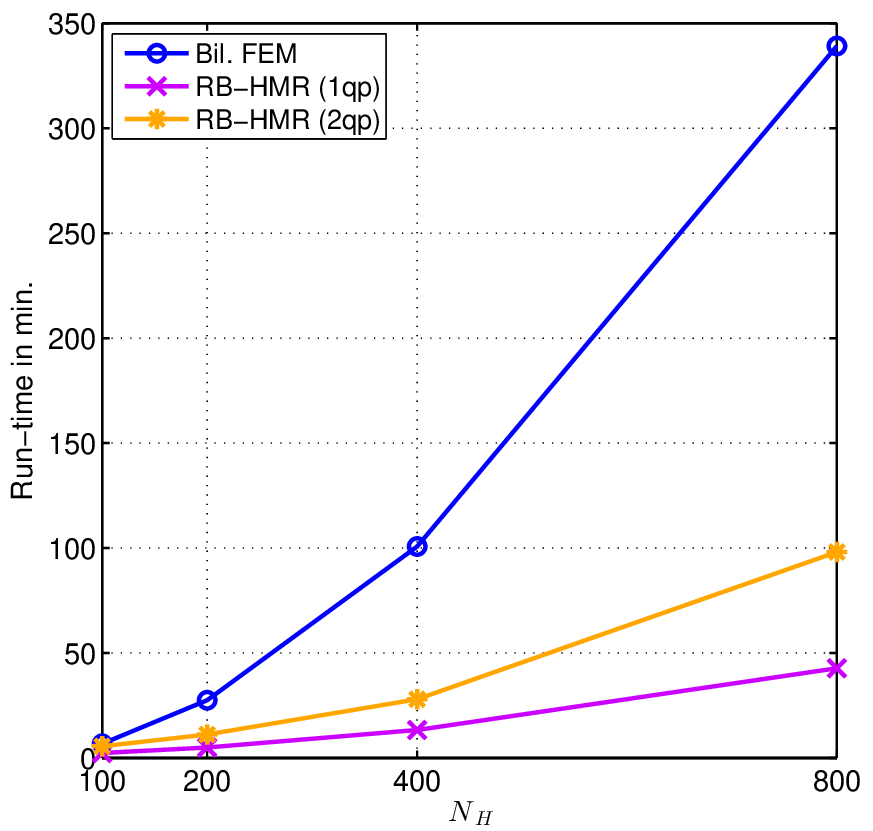}\label{fig4.15c}}
\subfloat[{\scriptsize $|e|_{H^{1}}^{rel}$ vs. runtime [min]}]{
\includegraphics[scale = 0.35]{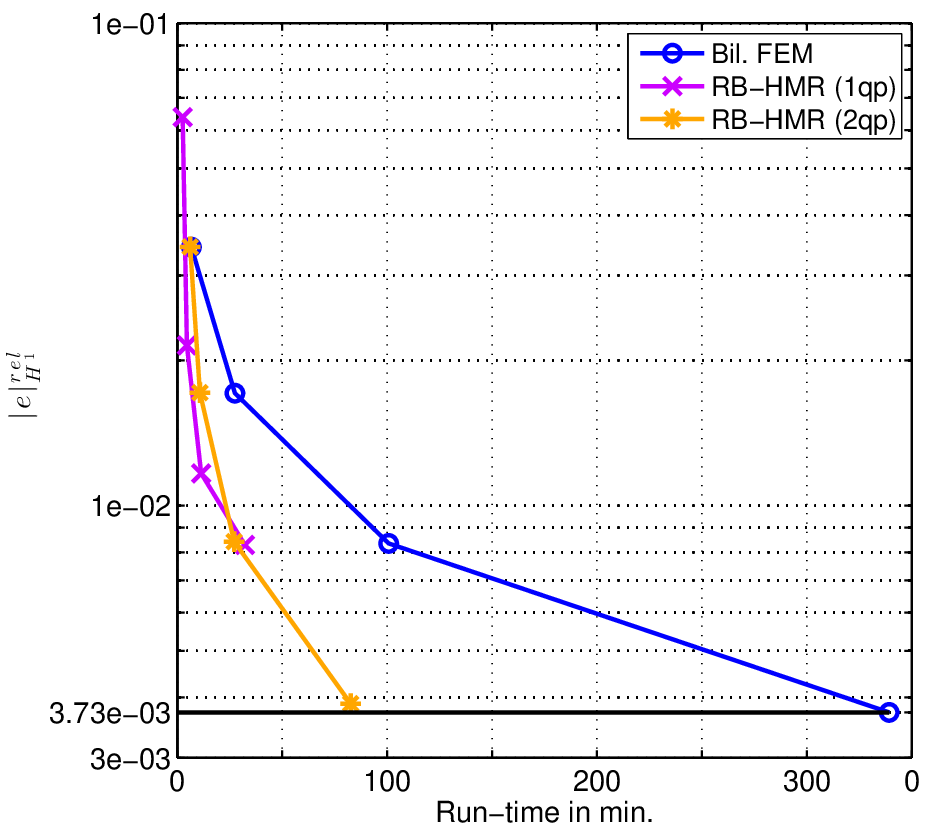}\label{fig4.15b}}
\caption{{\footnotesize Test case 2: Comparison of the total computational costs for the 2D bilinear FEM and the RB-HMR approach for $Q=1$ in \eqref{quad_formula_nonlin} (1qp) and $Q=2$ in \eqref{quad_formula_nonlin} (2qp) and $N_{H'}=10$. Solely in (B) we provide the runtimes if we compute $\Delta_{m}^{k}$ not only during the generation of the bases but also after the computation of $p_{m,k}^{H}$ to certify the approximation.}\label{fig4.15}}
\end{figure}

Next, we analyze the a posteriori error bounds derived in \S \ref{subsect-BRR}. The approximate values of the inf-sup stability factor $\beta_{2,p}^{\text{\tiny app}}$ \eqref{inf-sup-apost-H1} obtained with the method proposed in \S \ref{subsect_constants} are 
$
\beta_{2,p}^{\text{\tiny app}}\approx 0.073907 \,(H=0.02),$ $\beta_{2,p}^{\text{\tiny app}}\approx 0.074120\, (H=0.01),$ $\beta_{2,p}^{\text{\tiny app}}\approx 0.073955 \,(H=0.005),$ and $\beta_{2,p}^{\text{\tiny app}}=0.072911\, (H=0.0025)
$ for $Q=1$ and
$
\beta_{2,p}^{\text{\tiny app}}\approx 0.074403\, (H=0.02), \beta_{2,p}^{\text{\tiny app}}\approx 0.074436\, (H=0.01), \beta_{2,p}^{\text{\tiny app}}\approx 0.073969\, (H=0.005),$ and $\beta_{2,p}^{\text{\tiny app}}=0.072923\, (H=0.0025)
$ for $Q=2$, which seems consistent with the considered ellipticity constant $c_{0}=0.075$. 

For $Q=2$ and $H=0.005$ and $H=0.0025$ we obtain $\tau^{k}_{m,2}<1$ for $m \geq 8$ and we observe in Fig.~\ref{fig4.12c} that $\Delta_{m}^{k}$ provides an upper bound for $|e_{m}^{k}|_{H^{1}(\Omega)}$. Moreover, it can be seen that  $\Delta_{m}^{k}$ reproduces the error behavior very well and provides a sharp bound as the effectivities vary between $3.5$ and $8.5$ for $H=0.005$ and $2.8$ and $5.5$ for $H=0.0025$. Although we have $\|e_{\text{{\tiny EPM}}} \|_{H^{-1}(\Omega)} < \|e_{\text{{\tiny mod}}} \|_{H^{-1}(\Omega)}$ for all $m$ (see Fig.~\ref{fig4.12e}) both for $H=0.005$ and $H=0.0025$, for neither of the two the assumption $\tau^{k}_{m,2} \leq 0.5C_{err}$ in Proposition \ref{H1_bound} is satisfied. However, note that as in the previous test case, $\Delta_{m}^{k}$ reproduces the behavior of $|e_{m}^{k}|_{H^{1}(\Omega)}$ very well also for model orders for which $\tau^{k}_{m,2} \leq 0.5C_{err}$ is not satisfied. For $Q=2$ and the coarser discretizations $H=0.02$ and $H=0.01$ we obtain $\tau^{k}_{m,2}<1$ only for $m=18,19$ for $H=0.01$ (see Fig.~\ref{fig4.12b}) and for $Q=1$ we have $\tau^{k}_{m,2}>1$ for all considered mesh sizes (see Fig.~\ref{fig4.12a}). This is probably due to the fact that  in the present test case the solution exhibits limited spatial regularity, making it difficult to provide a satisfactory approximation with coarser discretizations and $Q=1$. Nevertheless, as $\tau^{k}_{m,2}$ captures the behavior of $|e_{m}^{k}|_{H^{1}(\Omega)}$ quite well (see Fig.~\ref{fig4.12a}-Fig.~\ref{fig4.12c}), $\tau^{k}_{m,2}$ may serve as an error indicator in those cases, suggesting that say $Q$ has to be increased. 

Fig.~\ref{fig4.12d} shows that for $Q=2$ and $k\geq 20$ we obtain $\|e_{\text{{\tiny EPM}}}^{\text{{\tiny ex}}} \|_{H^{-1}(\Omega)}>\|e_{\text{{\tiny EPM}}} \|_{H^{-1}(\Omega)}$, while for $k < 20$ it can be seen that $\|e_{\text{{\tiny EPM}}}^{\text{{\tiny ex}}} \|_{H^{-1}(\Omega)}$ and $\|e_{\text{{\tiny EPM}}} \|_{H^{-1}(\Omega)}$ coincide perfectly. This numerically proves that also for the present test case for a non-dominant discretization error, the a posteriori bound for the adaptive EPM \eqref{delta_epm} results in a very good approximation of $\|e_{\text{{\tiny EPM}}}^{\text{{\tiny ex}}}\|_{L^{2}(\Omega)}$ and thus $\|e_{\text{{\tiny EPM}}}^{\text{{\tiny ex}}}\|_{H^{-1}(\Omega)}$. Note that for $Q=1$ in \eqref{quad_formula_nonlin} the behavior of $|e_{m}^{k}|_{H^{1}(\Omega)}$ is reproduced perfectly  (Fig.~\ref{fig4.12a}) and that  $\|e_{\text{{\tiny EPM}}}^{\text{{\tiny ex}}} \|_{H^{-1}(\Omega)}$ and $\|e_{\text{{\tiny EPM}}} \|_{H^{-1}(\Omega)}$ mainly coincide even for high values of $k$ due to the higher level of the EPM-plateau. As in the previous test case we have set $\varepsilon_{\mbox{{\tiny tol}}}^{\text{{\tiny err}}} = tol_{k'} \varepsilon_{\mbox{{\tiny tol}}}^{\text{{\tiny EPM}}}$ with $tol_{k'}  =10^{-2}$ (\S \ref{adapt-RB-HMR-epm}), which yielded on average $k'-k \approx 4$ for $Q=1$ and $k'-k \approx 6$ for $Q=2$ during the adaptive refinement procedure and for the certification of the RB-HMR approach we obtained $k'-k \approx 2$ for $Q=1$ and $k'-k \approx 4$ for $Q=2$. 

Comparing $\|e_{\text{{\tiny mod}}}\|_{H^{-1}(\Omega)}$ and $\|e_{\text{{\tiny EPM}}} \|_{H^{-1}(\Omega)}$ for $Q=1,2$ in Fig.~\ref{fig4.12e} we observe that $\|e_{\text{{\tiny EPM}}} \|_{H^{-1}(\Omega)}$ has improved much more than $\|e_{\text{{\tiny mod}}}\|_{H^{-1}(\Omega)}$ due to the increase of $Q$. Hence, increasing $Q$ seems to significantly reduce the level of the EPM-plateau which in turn considerably improves the error behavior as has already been assessed in the analysis of Fig.~\ref{fig4.9}. In contrast to the previous test case, also for small tolerances $\varepsilon_{\mbox{{\tiny tol}}}^{\text{{\tiny EPM}}}$ a stagnation of 
$\|e_{\text{{\tiny mod}}}\|_{H^{-1}(\Omega)}$ can be observed (compare Fig.~\ref{fig4.5e} and Fig.~\ref{fig4.12e}). Thus, we suppose that due to the worse behavior of $p_{m,k}^{H}$ in the EPM-plateau compared to the previous test case, the EPM-plateau affects the convergence behavior of the model error for the present example. 

As the convergence behavior of $\lambda^{k}$ and $\|\int_{{\omega}}\mathcal{I}_{L}[A(p_{m,k}^{H})]\kappa_{k}\|_{L^{2}(\Omega_{1D})}^{2}$ does not coincide (see Fig.~\ref{fig4.11}), we replace, as proposed in \S \ref{subsect-BRR}, $\lambda^{k}$ by $\|\int_{{\omega}}\mathcal{I}_{L}[A(p_{m,k}^{H})]\kappa_{k}\|_{L^{2}(\Omega_{1D})}^{2}$ in the a priori bound \eqref{train_size_fixed} of Theorem \ref{apriori-epm}. Fig.~\ref{fig4.12f} shows that $e^{k}_{L^{2}}$ captures the behavior of $\|e_{\text{{\tiny EPM}}}\|_{L^{2}(\Omega)}$ and $\|e_{\text{{\tiny EPM}}}^{\text{{\tiny ex}}}\|_{H^{-1}(\Omega)}$ perfectly for $k \geq m$. The deviations for $k < m$ are due to the EPM-plateau. Although the snapshots set $\mathcal{M}^{\mathcal{A}}_{\Xi}$ and $A(p^{H\times h})$ are not approximated with the same approximation quality due to the EPM-plateau (Fig.~\ref{fig4.11}), we observe that $\|e_{\text{{\tiny EPM}}}\|_{L^{2}(\Omega)}$ and $\|e_{\text{{\tiny EPM}}}^{\text{{\tiny ex}}}\|_{H^{-1}(\Omega)}$ coincide for $k \leq 23$. Thus, we conclude that for the present test case the modified version of the a priori bound \eqref{train_size_fixed} of Theorem \ref{apriori-epm}, obtained by substituting $\lambda^{k}$ by $\|\int_{{\omega}}\mathcal{I}_{L}[A(p_{m,k}^{H})]\kappa_{k}\|_{L^{2}(\Omega_{1D})}^{2}$, can be applied to obtain a robust and efficiently computable a posteriori error estimator $\Delta^{k,rel}_{m}$.   \\

Finally, we compare the total computational costs of the RB-HMR approach using the adaptive EPM to compute $p_{m,k}^{H}$ \eqref{red_prob_hmr_epm}, with the costs of the 2D bilinear FEM for the computation of $p^{H\times h} \in V^{H\times h}$ \eqref{truth_nonlin}. Here, by the term ``total computational costs'' we mean all costs that are required to compute an approximation. Thus the total computational costs for the RB-HMR method comprise the costs for the construction of the reduction space and the collateral basis space by Algorithm \ref{adapt-RB-HMR}. Amongst others the costs for the RB-HMR approach therefore include the costs for the adaptive generation of the snapshot sets by Algorithm \ref{adapt-train} and the within this algorithm employed a posteriori error estimator. They also comprise the costs for the PODs which ultimately yield the reduction space and the collateral basis space. Finally, the total computational costs for the RB-HMR approximation also include the costs for the assembling and solution of the nonlinear system of equations \eqref{red_prob_hmr_epm}. 

For the solution of the nonlinear system of equations within Newton's method we employed in both cases a bicgstab method with the same settings. Also the tolerance for Newton's method has been chosen identically. In Fig.~\ref{fig4.15a} we see that the bilinear FEM scales quadratically in $N_{H}$, while the RB-HMR approach with the adaptive EPM scales nearly linearly in $N_{H}$ both for $Q=1$ and $Q=2$ in \eqref{quad_formula_nonlin}. In detail, we observe for $Q=1$ a scaling in $N_{H}$ of order $1.3$ and for $Q=2$ of order $1.35$. Note that this deviation from a linear scaling is due to the eigenvalue problems which need to be solved for the approximation of the inf-sup stability factor. In  Fig.~\ref{fig4.15b} the total computational costs of the bilinear FEM and the RB-HMR approach are plotted versus the respective relative total error $|e|_{H^{1}(\Omega)}^{rel}$. Due to the EPM-plateau the total error $|e|_{H^{1}(\Omega)}^{rel}$ for $Q=1$ lies well above the one of the bilinear FEM for the same mesh size, while only minimal deviations can be observed for $Q=2$. However, the runtime required to achieve a certain error tolerance is much smaller for the RB-HMR approach than for the bilinear FEM. Finally, we note that even if we additionally compute $\Delta^{k}_{m}$ to estimate the model error $|e_{m}^{k}|_{H^{1}(\Omega)}$ of the RB-HMR approximation and thus certify the approximation, the runtimes of the RB-HMR approach are much smaller than the runtimes of the bilinear FEM approximation (see Fig.~\ref{fig4.15c}).


\section{Conclusions}\label{conclusion}

To generalize the RB-HMR approach, introduced in \cite{OS10}, to nonlinear PDEs we expanded the range of the nonlinear operator in an orthonormal (collateral) basis in the transverse direction. Both for the construction of the reduction space in the RB-HMR approach and the collateral basis space we used a highly nonlinear approximation. A manifold of parametrized lower dimensional operator evaluations has been generated by using the solutions of a parametrized dimensionally reduced problem and the corresponding parametrization.  Solution and operator snapshot sets have been simultaneously generated with an adaptive training set extension and the reduction and collateral basis space have been constructed by applying a POD.  
In this way, we included both in the construction of the manifold of operator evaluations and the selection of the collateral basis information on the evaluation of the nonlinear operator in the unknown full solution. The coefficients of the operator approximation have been computed with the newly introduced adaptive EPM, which is an adaptive integration algorithm based on the (G)EIM \cite{BMNP04,MaMu13}. While for the basis selection with the greedy algorithm and the POD several convergence results have already been proven, to the best of our knowledge no result that could have been employed has been proved until now for the approximation of the range of a nonlinear operator. This has been realized in this article by the introduction of the adaptive Empirical Projection Method and the proven rigorous a priori and a posteriori error bounds. We used these bounds to derive a rigorous a posteriori error estimator based on the Brezzi-Rappaz-Raviart theory, which is employed for the construction of the snapshot sets. Here, we extended the results on the effectivity of the error estimator in \cite{CaToUr09} from quadratically nonlinear PDEs to general nonlinear PDEs of type \eqref{nonlinear_pde}. We note that some of the proposed procedures for estimating the constants within this a posteriori error bound may be improved.

The numerical experiments for the nonlinear diffusion equation show that the reference solution and the set of solution snapshots are approximated by the reduction space with the same approximation quality. Here, a quadrature formula of higher accuracy had to be employed for a test case with a non-smooth solution in the parametrized lower dimensional problem. The evaluation of the nonlinear operator in the reference solution and the set of operator snapshots are approximated by the collateral basis space with the same approximation accuracy for a problem with an analytic solution and a comparable quality for a problem with a non-smooth solution. Hence, we conclude that by employing the suggested ansatz for the generation of the solution manifold and the manifold of operator manifolds we are able to transfer the relevant features of the reference solution and the operator evaluation to the respective manifolds to a great extent.  Furthermore, the numerical experiments demonstrate an exponential convergence behavior of the RB-HMR approach both for a problem with an analytical solution and a test case with a non-smooth solution also for small ellipticity constants. The applicability of the theoretical results including the bounds for the adaptive EPM is demonstrated, too. In particular we observed that in many cases the a posteriori error estimator provides a sharp upper bound of the error. Runtime experiments show a close to linear scaling of the RB-HMR approach in the number of degrees of freedom used for the computations in the dominant direction, while the respective finite element reference approximation scales quadratically. This demonstrates the computational efficiency of the proposed method also in the nonlinear setting.\\

\textbf{Acknowledgements:} We would like to thank the anonymous reviewers very much for their careful review and their helpful remarks which have significantly improved the presentation of the paper. \\

\providecommand{\bysame}{\leavevmode\hbox to3em{\hrulefill}\thinspace}
\providecommand{\MR}{\relax\ifhmode\unskip\space\fi MR }
\providecommand{\MRhref}[2]{%
  \href{http://www.ams.org/mathscinet-getitem?mr=#1}{#2}
}
\providecommand{\href}[2]{#2}

\end{document}